\documentclass{article}

\title{A scattering theory construction of dynamical solitons in 3d}
\author{Istvan Kadar}

\usepackage[T1]{fontenc}
\usepackage[utf8]{inputenc}
\usepackage[english]{babel}
\usepackage{lmodern}
\usepackage[margin=3cm]{geometry}
\usepackage[backend=biber,style=alphabetic,sorting=nty,doi=false,isbn=false,url=false,eprint=false,minalphanames=3]{biblatex}
\AtEveryBibitem{\clearlist{language}}

\usepackage{empheq}
\usepackage{hyperref}
\usepackage{enumitem}
\usepackage{tikz-cd} 	
\usepackage{mathtools}
\usepackage{alltt}
\usepackage{amsfonts}
\usepackage{amsmath}
\usepackage{amssymb}
\usepackage{amsthm}
\usepackage{tikz-cd} 	
\usepackage[nottoc]{tocbibind}
\usepackage{graphicx}
\usepackage{caption}
\usepackage{subcaption}
\usepackage{aliascnt}
\usepackage{cases}
\usepackage{comment}	
\usepackage[capitalize,nameinlink]{cleveref}
\usepackage{tikz}
\usetikzlibrary{tikzmark}
\usepackage{slashed}

\usepackage{xargs}   
\usepackage[colorinlistoftodos,prependcaption,textsize=tiny]{todonotes}
\newcommandx{\typo}[2][1=]{\todo[linecolor=red,backgroundcolor=red!25,bordercolor=red,#1]{#2}}
\newcommandx{\change}[2][1=]{\todo[linecolor=blue,backgroundcolor=blue!25,bordercolor=blue,#1]{#2}}
\newcommandx{\answer}[1]{\todo[linecolor=pink,backgroundcolor=pink!25,bordercolor=pink]{#1}}
\newcommandx{\unsure}[2][1=]{\todo[linecolor=green,backgroundcolor=green!25,bordercolor=green,#1]{#2}}
\newcommandx{\improve}[2][1=]{\todo[linecolor=violet,backgroundcolor=violet!25,bordercolor=violet,#1]{#2}}
\newcommandx{\thiswillnotshow}[2][1=]{\todo[disable,#1]{#2}}

\crefformat{equation}{(#2#1#3)}
\numberwithin{equation}{section}

\theoremstyle{definition}

\newtheorem*{theorem*}{Theorem}
\newtheorem*{conjecture*}{Conjecture}

\theoremstyle{plain}
\newtheorem{theorem}{Theorem}[section]
\newtheorem{corollary}{Corollary}[theorem]
\newtheorem{lemma}[theorem]{Lemma}
\newtheorem{prop}[theorem]{Proposition}
\newtheorem{conjecture}[theorem]{Conjecture}
\newtheorem{cor}{Corollary}[theorem]
\newtheorem*{claim}{Claim}

\theoremstyle{definition}
\newtheorem{definition}[theorem]{Definition}

\theoremstyle{remark}
\newtheorem{remark}[theorem]{Remark}

\crefname{lemma}{Lemma}{Lemmas}
\crefname{prop}{Proposition}{Proposition}
\crefname{conjecture}{Conjecture}{Conjecture}
\crefname{cor}{Corollary}{Corollary}
\crefname{remark}{Remark}{Remark}
\crefname{defi}{Definition}{Definition}
\crefname{equation}{}{}

\newcommand{\Tau}{\mathrm{T}}

\newenvironment{nalign}{
	\begin{equation}
		\begin{aligned}
		}{
		\end{aligned}
	\end{equation}
	\ignorespacesafterend
}

\newcommand{\N}{\mathbb{N}}

\newcommand{\R}{\mathbb{R}}

\newcommand{\T}{\mathbb{T}}

\renewcommand{\H}{\mathbb{H}}
\newcommand{\scri}{\mathcal{I}}

\newcommand{\abs}[1]{\left\lvert #1\right\rvert}

\newcommand{\jpns}[1]{\langle #1 \rangle}

\newcommand{\norm}[1]{\left\lVert #1\right\rVert}

\newcommand{\floor}[1]{\lfloor #1 \rfloor}
\newcommand{\ceil}[1]{\lceil #1 \rceil}

\DeclareMathOperator{\im}{im}

\renewcommand{\d}{\mathrm{d}}

\DeclareMathOperator{\Diff}{Diff}

\DeclareMathOperator{\supp}{supp}


\usepackage{accents}
\usepackage{cancel}
\graphicspath{{figures}}

\usepackage[font={small,it}]{caption}

\newcommand{\masterJ}{\underset{{\mbox{\tiny (k)}}}{\mathfrak{J}}}
\newcommand{\masterK}{\underset{{\mbox{\tiny (k)}}}{\mathfrak{K}}}
\newcommand{\Vlin}{\stackrel{\scalebox{.6}{\mbox{\tiny (1)}}}{V}}
\newcommand{\psia}{ \prescript{}{\tiny{a}}{\psi} }

\newcommand{\alphaa}{\prescript{}{a}{\alpha}}

\newcommand{\Tlin}{\stackrel{\scalebox{.6}{\mbox{\tiny (1)}}}{\T}}
\newcommand{\master}{\underaccent{\scalebox{.8}{\mbox{\tiny $k$}}}{\mathcal{X}}}
\newcommand{\masterNum}[1]{\underaccent{\scalebox{.8}{\mbox{\tiny $#1$}}}{\mathcal{X}}}
\newcommand{\rdom}{\mathcal{R}_{\tau_1,\tau_2}}
\newcommand{\rint}{\int_{\mathcal{R}_{\tau_1,\tau_2}}}
\newcommand{\rtStar}{\partial_r|_{t_\star}}
\newcommand{\Rtild}{\tilde{\mathcal{R}}}

\DeclareFontFamily{U}{rcjhbltx}{}
\DeclareFontShape{U}{rcjhbltx}{m}{n}{<->rcjhbltx}{}
\DeclareSymbolFont{hebrewletters}{U}{rcjhbltx}{m}{n}

\DeclareMathSymbol{\lamed}{\mathord}{hebrewletters}{108}

\usepackage{tocloft}
\begin{document}
	\maketitle
	\begin{abstract}
		We study the energy critical wave equation in 3 dimensions around a single soliton. We obtain energy boundedness (modulo unstable modes) for the linearised problem. We use this to construct scattering solutions in a neighbourhood of timelike infinity ($i_+$), provided the data on null infinity ($\scri$) decay polynomially. Moreover, the solutions we construct are conormal on a blow-up of Minkowski space. The methods of proof also extend to some energy supercritical modifications of the equation.
	\end{abstract}
	\setcounter{tocdepth}{1}
	\tableofcontents
	\section{Introduction}

	It is well-known that the energy critical wave equation
	\begin{equation}\label{main equation}
		\begin{gathered}
			\Box\phi:=(-\partial_t^2+\Delta)\phi=-\phi^5,\quad
			\phi	:\R^{3+1}\to\R
		\end{gathered}
	\end{equation}
	admits a stationary solution, called the ground state soliton
	\begin{equation}\label{soliton}
		\begin{gathered}
			W(x)=\Big(1+\frac{\abs{x}^2}{3}\Big)^{-1/2}.
		\end{gathered}
	\end{equation}
	
	This solution suffers from Derrick's instability (see \cite{derrick_comments_1964},\cite{grillakis_stability_1987}), which in particular implies the existence of an exponentially growing mode for the linearised system
	\begin{equation}\label{linearised eom}
		\begin{gathered}
			(\Box+V)\psi=0, \quad V:=5W^4.
		\end{gathered}
	\end{equation}
	Nonetheless, global solutions that settle down to one or more solitons is one of the expectations from the soliton resolution conjecture. In this work, we initiate the study of such global solutions following the geometric framework as discussed in \cite{christodoulou_formation_2009,dafermos_scattering_2013}. The main result of this paper is

	\begin{theorem}(Rough form of main theorem, see \cref{main theorem})\label{main theorem sketch}
			Given scattering data $\psi^{dat}(u,\omega)$ ($u\in\R$, $\omega\in S^2$) that fall off at a polynomial rate, there exists a solution $\phi$ to (1.1) in the region $t-r\gg 1$ with the following properties: firstly, $\phi(t,x)-W(x)$ decays to $0$ at a quantitative rate as $t\to\infty$; and second, the radiation field of $\phi$ is given by $\psi^{\mathrm{dat}}$, i.e.\ $r\phi(u+r,r,\omega) \to \psi^{\mathrm{dat}}(u,\omega)$.
	\end{theorem}

	We may phrase the theorem in a less geometric way as
	
	\begin{corollary}(See \cref{cor:intro:Cauchy})
		There exists a large class of scattering data $(u^\mu_\infty)\in\mathcal{C}^\infty$ (of arbitrary finite $\dot{H}^1$ norm) for $\mu\in\{0,1,2,3\}$, such that \cref{main equation} admits a solution for $t>T(u^\mu)$ scattering to this data and a soliton, ie.
		\begin{equation}
			\begin{gathered}
				\norm{\partial_\mu(\phi-W)-u^\mu_\infty}_{L^2(\{t=c\})}\to0,\qquad \text{as }c\to\infty.
			\end{gathered}
		\end{equation}
	\end{corollary}

	\paragraph{Overview of the introduction}
	The structure of the rest of the introduction is as follows. In \cref{introduction:motivation}, we give the main motivation, by outlining future research to study \cref{main equation,main theorem sketch}. In \cref{sec:intro_proof ideas}, we present the geometric framework in which our construction works and give a more precise form of the main theorem. In \cref{sec:intro_related} we discuss the scope of the techniques developed in the present paper and state an energy supercritical result to which our methods directly apply. Finally, in \cref{sec:intro_proof ideas} we present the main difficulties for proving \cref{main theorem sketch} and the ideas how to resolve them.
	\subsection{Review of previous works and motivation}\label{introduction:motivation}
	\subsubsection{General relativity}\label{gr motivation} The result of the current paper is an initial attempt by the author to understand the problem of incoming or outgoing localised objects from past or future timelike infinity ($i_\pm$) in general relativity. More precisely, let us note that the following problem is currently open:
	\begin{conjecture}\label{BH conjecture}
		There exists an asymptotically flat solution to the vacuum Einstein equations that approximates $n$ outgoing black holes in a neighbourhood of $i_+$.
	\end{conjecture}
	
	There are many difficulties one has to face while studying \cref{BH conjecture}. Let us discuss some of these and why \cref{main theorem sketch} is a good starting point in our opinion.

	\begin{itemize}
		\item \textbf{Weak null singularities:} One may hope to construct two-black hole solutions in the following trivial way. Write $g_g=g_{guess}=\frac{1}{2}(g_1+g_2)$, where $g_i$ represents boosted single-black hole metrics. Of course $g_{g}$ is not going to solve the Einstein vacuum equations, but we may write the true solution as $g=g_{g}+\epsilon$ where we expect $\epsilon$ to decay as time goes to infinity. Trying to solve for $\epsilon$, let's say with trivial scattering data on the horizons ($\mathcal{H}^+_1,\mathcal{H}^+_2,\scri$), we find that there are source terms from $g_{g}$ that merely decay polynomially. Combined with the blueshift mechanism around the horizon of black holes, the expectation (see \cite{dafermos_scattering_2013}) is that generic polynomially decaying data (or in this case sources) lead to a weak null singularity on the black hole event horizon(s). Even though solutions with such singularities can be constructed, following the techniques from \cite{luk_weak_2018}, they present additional difficulty in the problem.
		Furthermore, one would be more interested in finding solutions that do not contain such artificial singularities.
		Instead, one may try to prescribe data on future horizons, so as to prevent the formation of weak null singularities, but this is open even for the single black hole scenario. 
		
		One possible way to circumvent the appearance of blueshift is to replace \textit{black holes} in \cref{BH conjecture} with solitons that are formed by some matter coupled to the Einstein equations. The existence of such solutions and their linear stability for Klein-Gordon coupling was considered in \cite{lee_stability_1989}. However, the understanding of global properties of massive fields is far less advanced (see recent progress \cite{pasqualotto_asymptotics_2023,sussman_massive_2023}), so we choose to study a wave-like model instead. Moreover, massless fields are also more closely related to the vacuum equations.
		
		\item \textbf{Quasilinearity:}	The Einstein equations, under a suitable gauge choice, form a system of quasilinear wave equations. The analysis of such systems usually involves the construction of geometric vector fields adapted to the geometry, notably in the original work on the stability of Minkowski space \cite{christodoulou_global_1993}. Controlling the variation of these vector fields from the background significantly complicates and lengthens the works on quasilinear equations. Even though to resolve \cref{BH conjecture}, one needs to face these challenges, we think that given the past work on quasilinear systems (see \cite{keir_weak_2018} and references therein for a comprehensive guide), the main new ingredients to study this problem can be found in a semilinear analogue. Therefore, we will consider a semilinear wave equation that still admits solitons. This is the energy critical wave equation \cref{main equation}. For energy critical wave equation in certain dimensions these solutions were constructed in \cite{martel_construction_2016,yuan_multi-solitons_2019,yuan_construction_2022}.\footnote{ Note, that coupling this to the Einstein equations yields a matter term that does not satisfy any of the usual energy conditions.}
		
		\item \textbf{Reduction to single soliton:} As discussed above, creating multisoliton solutions may proceed by supplying an approximate solution plus scattering data for the error term. To the author's knowledge, the \textit{geometric} scattering problem around $i^+$\footnote{We use the term \textit{geometric} to refer to data prescribed at $\scri$ instead in the Lax-Philips sense, where one takes the inverse evolution with respect to the linear operator on $\{t=\text{const}\}$ slices. For more details on the relationship between these two approaches, see \cite{friedlander_radiation_1980}. } has not been studied for the energy critical wave equation, nor for any similar semilinear problems. The only nonlinear geometric scattering theory construction around solitons that the author is aware of are \cite{dafermos_scattering_2013}. Both of these work with exponentially decaying data, which is necessary to avoid blueshift instability, but greatly simplifies the work in many other regards as well. It is unclear to the author whether such exponentially decaying data results could be applied in multisoliton constructions, as the interactions fall off polynomially.
	\end{itemize}

	\subsubsection{Dispersive equations:}

	 Another motivation to study \cref{main theorem sketch} is to introduce the methods developed for the study of the Einstein equations to hyperbolic dispersive problems. For similar transfer of ideas to compressible fluid equations see the recent review \cite{disconzi_recent_2023}. Before we explain further, let us briefly recall some works about \cref{main equation} without intending to be comprehensive about this vast research program.
	
	A major objective in the study of nonlinear dispersive equations is the soliton resolution conjecture. It predicts that any global solution with norm bounded in some function space decouples as a sum of solitons, some radiation and a term that vanishes in the energy space as time goes to infinity. This definition is crucially tied to a $t,x$ foliation of $\R^{1+n}_{t,x}$, indeed this conjecture can be formulated  not just for hyperbolic problems, but for many dispersive and dissipative equations as well. We refer to \cite{duyckaerts_soliton_2022} and references therein for more discussion and current progress on this topic.
	
	There are two main approaches in the study of the soliton resolution conjecture. First, one can use the control provided by the boundedness of some norm to classify all solutions within a function space. This norm is usually given by the canonical energy, which sits at $H^1$ regularity. See \cite{duyckaerts_classification_2013,duyckaerts_soliton_2023,duyckaerts_soliton_2022} and references therein for more details. To the author's knowledge, this has been only used for energy critical and subcritical problems.

	A second approach, applied to energy supercritical problems as well, is to construct solutions of special type and understand the behaviour in a neighbourhood of these so as to sample the moduli space of solutions in small steps. One can study either global behaviour as in \cite{krieger_focusing_2007,donninger_nonscattering_2013}, or singularity formation as in \cite{krieger_renormalization_2008}. In this case, one regularly works in much stronger topology than $H^1$.
	 
	 Regarding \cref{main equation} around a single soliton, the above-discussed two approaches can be seen precisely. Perturbation around a single soliton and proof of asymptotic stability with all the dynamics traced via a bootstrap type argument is done in  \cite{krieger_focusing_2007,beceanu_center-stable_2014}, however, both of these works restrict to a symmetry class such that the soliton is not allowed to move its position. On the other hand, we have the resolution of the forward problem in \cite{krieger_center-stable_2015}, where the authors prove codimension 1 orbital stability of the soliton without any symmetry assumptions. In this work, codimension one stability is a corollary of a more general classification result.
	 
	The main reason to depart from the standard methods deployed for the study of \cref{main equation} is to get a more precise description of the solutions. For the current work, this is the conormal structure, i.e. commutation with weighted vectorfields.
	
	\subsection{Main theorem}\label{sec:intro_proof ideas}

	In this section, we will give a a precise version of the theorem, a sketch of the steps in its proof and their relation to other works in the literature that this paper builds upon.
	
	\begin{theorem}\label{main theorem}
		a)
		Let $N>6$ be an integer and $q>5$ a real number. Let $(r\psi)^{dat}\in\mathcal{C}^\infty(\scri)$ for $\scri=\R_u\times S^2$\footnote{this is a portion of infinity define by the endpoint of the curves $t-r=u$, $t\to\infty$, see \cref{eq:notation:hypersurfaces}} be a smooth function that satisfies the following bounds
		\begin{equation}
			\begin{gathered}
				\sum_{\abs{\alpha}\leq N}\int_{\scri}\d\omega \frac{\d u}{\jpns{u}}\Big(\jpns{u}^{q}\Gamma^\alpha(r\psi)^{dat}\Big)^2\leq 1,\quad \Gamma\in\{\partial_\omega,\jpns{u}\partial_u\}.
			\end{gathered}
		\end{equation}
		Then, there exists $\tau_1$ sufficiently large such that \cref{main equation} admits a solution $\phi=W+\psi$ in $t-r>\tau_1\gg 1$ such that
		\begin{equation}\label{eq:mainTheorem:trace}
			\begin{gathered}
				\lim_{\substack{r\to\infty, \\t-r=2u_\star}}r\partial_u\psi=\partial_u(r\psi)^{dat}(v_\star,\omega),\\ 
				\lim_{u_\star \to\infty}
				\lim_{\substack{r\to\infty, \\t-r=2u_\star}}\psi(u_\star,\omega)=0.
			\end{gathered}
		\end{equation}
	Moreover, $\psi$ is conormal, i.e. the following estimate holds for $\abs{\alpha}\leq N-3$
	\begin{equation}\label{eq:mainTheorem:decay}
		\begin{gathered}
			1_{r<t/2}(\Gamma^\alpha\psi),
			1_{r>t/2}(\tilde{\Gamma}^\alpha\psi)\lesssim\jpns{r}^{-1/2}\jpns{t-r}^{-q} ,\quad\Gamma\in\{r\partial_r,t\partial_t,\Omega\},\tilde{\Gamma}\in\{u\partial_u,v\partial_v,\Omega\}
		\end{gathered}
	\end{equation}
	
	b) Assume that $\psi^{dat}\in\mathcal{C}^{\infty}(\scri)$ has the following expansion
	\begin{equation}
		\begin{gathered}
			\psi^{dat}_{\mathfrak{E},M}=\psi^{dat}-\sum_{(i,j)\in\mathcal{E}, i\leq M} u^{-i}\log^j u \psi^{dat}_{i,j}(\omega),\quad \psi^{dat}_{i,j}\in\mathcal{C}^\infty(S^2)\\
			\sum_{\abs{\alpha}\leq N} \int_{\scri}d\omega \frac{d u}{\jpns{u}}\Big(\jpns{u}^{q}\Gamma^{\alpha}\psi^{dat}\Big)^2\leq 1,\quad \Gamma\in\{\partial_\omega,\jpns{u}\partial_u\},
		\end{gathered}
	\end{equation}
	for some index set $\mathcal{E}\subset\R_{\geq 3}\times\N$.  Then, $\psi$ has a polyhomogeneous expansion on an appropriate blow-up, e.g. there exists $\psi^F_{i,j}\in\mathcal{C}^\infty(\{\abs{x}\leq R\})$ for arbitrary $R$ such that
	\begin{equation}
		\begin{gathered}
			\psi=\sum_{(i,j)\in\mathcal{E}^\mathcal{F}, i\leq M} \psi^{\mathcal{F}}_{i,j}(x)t^{-i}\log^j t +\mathcal{O}(t^{-M-})
		\end{gathered}
	\end{equation}

	\end{theorem}
	
	A couple of remarks are in order.
	\begin{remark}[Decay rates]
		The assumed decay rate ($q$) is far from optimal, even with the method of the current paper. An important missing ingredient to lower the minimum value of $q$ more is integrated local energy decay (i.e. Morawetz estimate), see \cref{introdcution:forward} for more details. Furthermore, let us remark, that the $\jpns{r}^{-1/2}$ decay can easily be improved to $\jpns{r}^{-1}$ using the $r^p$ method (see \cite{schlue_decay_2013,moschidis_rp_2016}), but as this is not required for nonlinear application, we do not prove it in \textit{a)}.
	\end{remark}

	\begin{remark}[Conormal structure]
		Following the work of \cite{dafermos_quasilinear_2022}, we believe that the conormal condition, that is regularity with respect to $\jpns{u}\partial_u$ can be relaxed to $\partial_u$ derivatives to obtain a scattering solution. However, following recent progress on the late time behaviour of waves on asymptotically flat spacetimes (see \cite{schlue_decay_2013,moschidis_rp_2016,hintz_linear_2023}), we know that such a conormal assumption is well justified.
	\end{remark}
	
	\begin{remark}[Codimension]
		The unstable mode present for the linearised problem around the soliton means that a forward stability can be at most of finite codimension. However \cref{main theorem} contains an open neighbourhood of scattering data within a certain decay class. We interpret this as follows. The restriction that we are looking for a global solution prohibits any nontrivial dynamics in the unstable mode, because otherwise we would drift away from the soliton too widely. 
	\end{remark}
	
	\begin{remark}[Uniqueness]
		The theorem above does not yield uniqueness of the solution constructed. Indeed, we do not expect uniqueness in general, as already for the trivial data $\psi^{dat}=0$ we have solutions other than $\psi=0$, see \cite{duyckaerts_dynamics_2008} Theorem 1. These are to first order given by the unstable mode $e^{-t\lamed}Y(r)$.
	\end{remark}
	
	For better comparison with scattering results appearing in the literature concerning dispersive equations, we extend the solution from a neighbourhood of $i^+$ to $i_0$
	\begin{theorem}(see \cref{sec:cauchy} for precise statement)\label{thm:trivial_scattering}
		Given scattering data $\psi^{dat}\in\mathcal{C}^\infty(\scri\cap\{u\in(u_1,u_2)\})$, there exists $d$ sufficiently large such that \cref{main equation} admits a solution $\phi=W+\psi$ in $\{v>d, u<u_2+1\}$ such that
		\begin{equation}
			\begin{gathered}
				\lim_{\substack{r\to\infty, \\t-r=2u_\star}}r\psi=\psi^{dat}(v_\star,\omega).
			\end{gathered}
		\end{equation}
	\end{theorem}
	\begin{remark}
		There is no restriction on the size of any norm of $\psi^{dat}$. This is a simple consequence of working in the far exterior region and shows that constructing solutions around $i_0$ is much simpler in a certain sense.
	\end{remark}

	In particular, we may combine \cref{main theorem,thm:trivial_scattering} to get
	\begin{cor}\label{cor:intro:Cauchy}
		There exists scattering data $(u^\mu_\infty)\in\mathcal{C}^\infty$ (of arbitrary finite $\dot{H}^1$ norm) for $\mu\in\{0,1,2,3\}$, such that \cref{main equation} admits a solution for $t>T(u^\mu)$ scattering to this data and a soliton, i.e.
		\begin{equation}
			\begin{gathered}
				\norm{\partial_\mu(\phi-W)-u^\mu_\infty}_{L^2(\{t=c\})}\to0,\qquad \text{as }c\to\infty.
			\end{gathered}
		\end{equation}
		
	\end{cor}
	\subsubsection{Backwards problem} 
	
	The setup and the question posed in \cref{main theorem sketch} are those appearing in \cite{dafermos_scattering_2013}. As highlighted in that work, the scattering (i.e. backward) construction is detached from many of the difficulties present in the stability (i.e. forward) problem. In particular, one does not need to capture the (retarded) time decay of waves, a necessary ingredient in stability problems. However, the exponentially decaying scattering data posed in \cite{dafermos_scattering_2013} also allows one to get away without proving energy boundedness for the linearised operator. This is of crucial importance, as the presence of an ergo region intertwines boundedness and decay statement (see \cite{Dafermos2016}). For polynomially decaying data, energy boundedness appears crucial. Let us now present the set up of \cite{dafermos_scattering_2013}, which we will follow closely.
	
	In \cite{dafermos_scattering_2013}, the authors solve finite scattering problem in the region $\mathcal{R}$ shown in \cref{fig:two_hypersurface_BH}. They prescribe data on $\Sigma_2,\underline{\mathcal{C}}_{v_\infty},\mathcal{H}_{\tau_1,\tau_2}$ and construct the solution by a bootstrap argument.

	\begin{figure}[h]
		\centering
		\begin{subfigure}{.6\textwidth}
			\centering
			\resizebox{.9\textwidth}{!}{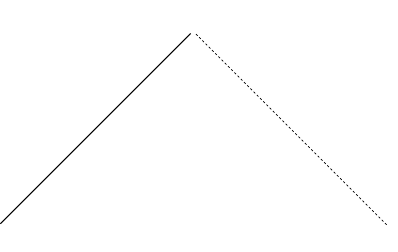}
			\caption{Black holes construction}
		\end{subfigure}%
		\begin{subfigure}{.5\textwidth}
			\centering
			\LARGE
			\resizebox{0.7\textwidth}{!}{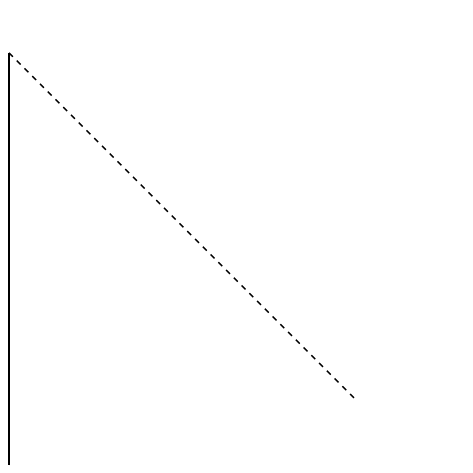}
			\normalsize
			\caption{Soliton construction}
		\end{subfigure}
		\caption{Penrose diagram of the region where the solution is constructed in \cite{dafermos_scattering_2013} (a)  with a limit taken as $\tau_2\to\infty$.  Scattering constructions in \cref{main theorem} ($I$), \cref{thm:trivial_scattering} ($II$) (b), $\tilde{\Sigma}$ denotes constant $t$ hypersurfaces.}
		\label{fig:two_hypersurface_BH}
	\end{figure}
	
	Then, they take the limit $v_\infty\to \infty$ thereby getting a scattering solution in bounded retarded time regions. Afterwards, they  take $u_\infty\to\infty$ and obtain uniform in $u_\infty$ estimates on the solution in $\mathcal{R}$. Existence of a scattering solutions follows by compactness argument.
	
	Let us briefly comment on this strategy. For equations that have quadratic nonlinearities the first step is highly non-trivial. Indeed, the famous example of Fritz John $\Box\phi=(\partial_t\phi)^2$ illustrates that this step might break down \footnote{though note that one may pose special, non generic, modified data on $\scri$ such that a scattering solution may exist \cite{yu_modified_2021,yu_nontrivial_2022}}. However, under very mild conditions --known as weak null structure-- this step should pose no difficulties if one uses the $r^p$ method of Dafermos--Rodnianski \cite{dafermos_new_2010}, see \cite{keir_weak_2018}. Indeed in \cite{dafermos_scattering_2013}, the authors introduce a structure, called $p-$indices to capture the behaviour of different parts of the system under study towards $\scri$. In our case, the quintic nonlinearities make this step trivial. 
	
	To proceed with the second limit, one uses energy estimates of the form
	\begin{equation}\label{introduction energy estimate}
		\begin{gathered}
			\mathcal{E}[\psi]_{\tau}+\int_{\mathcal{R}_{\tau,\tau_\infty}}\text{Bulk term}\lesssim\mathcal{E}[\psi]_{\tau_\infty}+\mathcal{E}_{\scri_{\tau,\tau_\infty}},
		\end{gathered}
	\end{equation}
	where $\mathcal{E}[\psi]_\tau$ is some $L^2$ based coercive norm of $\psi$ on $\Sigma_\tau$ and $\mathcal{E}_{\scri_{\tau,\tau_\infty}}$ is a similar quantity on $\scri$. In fact, it is sufficient to have such an estimate for the linearised problem. In most cases of interest, we have
	\begin{equation}
		\begin{gathered}
			\int_{\rdom}\text{Bulk term}\lesssim \int_{(\tau_1,\tau_2)}\mathcal{E}[\psi]_\tau.
		\end{gathered}
	\end{equation}
	Putting together these two estimates, and using the Gronwall inequality yields global control on $\mathcal{E}[\psi]_\tau$, provided that $\mathcal{E}_{\scri_{\tau_1,\tau_2}}$ decays exponentially and $\mathcal{E}[\psi]_{\tau_\infty}=0$. A similar issue also arises for \cref{main equation}. Even though there is no blueshift or ergo region present, the linearised problem does not admit a coercive energy because of a bad signed potential. It seems to be necessary to have energy boundedness for the linear problem to obtain global energy bounds.

	\subsubsection{Obstructions to boundedness: unstable and zero modes}\label{subsec:obstruction_to_decay}
	To understand the obstructions to boundedness, observe that the tangential part of the $\partial_t$ energy norm on $\tilde{\Sigma}_a=\{t=a\}$ slices
	\begin{equation}
		\begin{gathered}
			\mathcal{E}_{\tilde{\Sigma}_a,\parallel}^V[\psi]=\int_{\tilde{\Sigma}_a} \abs{\nabla\psi}^2-V\psi^2
		\end{gathered}
	\end{equation}
	of \cref{linearised eom} is not coercive. Indeed, we know (see \cite{duyckaerts_dynamics_2008}) that there is a unique eigenvalue $\lamed>0$ and corresponding eigenfunction $Y\in\dot{H}^1$ such that $(\Delta+V)Y=\lamed^2 Y$, moreover $Y\in\mathcal{C}^\infty, \abs{Y}\lesssim e^{-\lamed r}$.  This yields solutions $Y_\pm=e^{\pm\lamed t}Y$, called unstable modes, to \cref{linearised eom} with  $\mathcal{E}_{\tilde{\Sigma}_a,\parallel}^v[Y_\pm]<0$. To overcome this, one may try to capture the content of $Y_\pm$ present in a solution of \cref{linearised eom} $\psi$ by introducing
	\begin{equation}
		\begin{gathered}
			\alpha_{\pm}^\infty(s)=\int_{\tilde{\Sigma}} Y(\partial_t\pm\lamed)\psi\underbrace{\implies}_{\cref{linearised eom}} \partial_s\alpha_{\pm}^\infty(s)=\pm\lamed\alpha_{\pm}^\infty(s)
		\end{gathered}
	\end{equation}
	Indeed, controlling $\mathcal{E}^V_{\tilde{\Sigma},\parallel}$ and  the assumption $\alpha_{\pm}^\infty(s)=0$ implies that $\norm{\psi}_{\dot{H}^1}\geq 0$. For nonlinear applications, it is helpful if one is allowed to localise  $\alpha_{\pm}^\infty$ and introduce a cutoff function in its definition as done in \cite{luhrmann_stability_2022}. 
	
	However, non-negativity is not sufficient for nonlinear applications and it is immediate that all the symmetries of \cref{main equation} introduce zero mode solutions to \cref{linearised eom}. Indeed, note that $(\Box+V)Y^{\ker}=0$ for all $Y^{\ker}\in\mathcal{Z}:=\{\partial_i W,t\partial_i W,\Lambda W,t \Lambda W\}$. A crucial point is that $\mathcal{Z}$ contains all the obstruction for coercivity. In the terminology of Duyckaerts-Kenig-Merle \cite{duyckaerts_solutions_2016}, we require that $W$ is non-degenerate, i.e. $\ker(\Delta+5W^4)=\mathcal{Z}$.
	
	One could attempt to project these modes similar to the unstable ones, but due to the slow decay of $\Lambda W$, the corresponding quantity would not be well defined for $\dot{H}^1$ functions. If one localises via a cutoff function, that would introduce bad weights that cannot be controlled in a similar fashion as done in \cite{luhrmann_stability_2022}. In this paper, we instead capture the content of the zero modes (i.e. kernel elements) globally, similar to \cite{johnson_linear_2019,dafermos_linear_2019}. The exact procedure of this is the main content of the present work and it is included in \cref{sec:linear theory section}.
	
	\paragraph{Continuous vs discrete modulation}
	The presence of the kernel elements for the linearised operator is very well understood in many different settings and the treatment of these is usually referred to as modulation theory, see e.g. \cite{krieger_focusing_2007,luhrmann_stability_2022}. In this paper, we do this modulation in a discrete manner. We control such modulation parameters by applying appropriate projection operators to the solution of the linearised equation. In the scattering problem, we do not need to renormalise our solution, because we know that all of these parameters will converge to 0. For the stability problem, the final soliton might obtain some momentum and change in its scaling parameter compared to the perturbed one. To keep the solution perturbative, one needs to renormalise the soliton at timescales when the linear behaviour is modified by the nonlinear dynamics, see \cref{sec:intro_dyadic}.

	\subsubsection{Unstable mode and codimension}
	Let us discuss the data posed on $\Sigma_{\tau_\infty}$. Because of the unstable modes $Y_-$, it is easy to see that given generic scattering data and $(r\psi)|_{\Sigma_{\tau_\infty}}=0$ the solution may grow exponentially in retarded time ($\tau_\infty-t_\star$), even though the data on $\scri$ is only polynomial. However, by a topological argument, we can select $a_{\tau_\infty}\in\R$ depending on the scattering data, such that $\psi|_{\Sigma_{\tau_\infty}}=a_{\tau_\infty}Y_-$ yields a polynomially growing solution. In fact, $a_{\tau_\infty}\to 0$ as $\tau_\infty\to\infty$ with a rate prescribed by the decay of $(r\psi)|_{\scri}$.
	
	\subsubsection{Dyadic iteration}\label{sec:intro_dyadic}
	Finally, let us mention the structure of norms used in the construction of the scattering solution. We follow \cite{dafermos_quasilinear_2022}, and first consider the problem in bounded retarded time regions.  We prove an existence result and control of energy type quantities that holds on a timescale that  grows inverse polynomially with the size of the initial data. Then, we iterate this construction on dyadic slabs to get a global result.
	
	Even though the current problem is simple enough to not to use the dyadic technique, we hope to make further progress on the program outlined in \cref{gr motivation} in future works, where we believe that the dyadic approach has a more significant role.
	
	\subsection{Remarks on related problems}\label{sec:intro_related}
	
	\subsubsection{Supercritical problem}
	As remarked in \cref{introduction:motivation}, we develop methods that are applicable for supercritical equations as well. However, the detailed study of the soliton and its spectral properties are a necessary requirement for our methods to work. Consider the problem
	\begin{equation}\label{eq:supercritical:elliptic}
		\begin{gathered}
			-\Delta u=u^7-u^q,\quad u:\R^3\to\R\\
			u>0,\quad u(x)=u(\abs{x}),\quad u(r)=\mathcal{O}(r^{-1}) \text{ as } r\to\infty,\quad q\in\N_{>7}.
		\end{gathered}
	\end{equation}
	Then, we have
	\begin{theorem}[\cite{kwong_ground_1992,lewin_double-power_2020}]
		\cref{eq:supercritical:elliptic} admits a unique radially symmetric positive solution, $W_q$, and we have that the linearised operator $\mathcal{L}_{W_q}=-(\Delta+7W_q^6-qW^{q-1}_q)$ satisfies
		\begin{equation}
			\begin{gathered}
				\ker(\mathcal{L}_{W_q})=\{\partial W_q\}.
			\end{gathered}
		\end{equation}
	\end{theorem}
	Furthermore, we have
	\begin{lemma}
		a) The solution $W_q$ is conormal, i.e. $(r\partial_r)^m W_q=\mathcal{O}(r^{-1})$ for all $n$.
		
		b) Let $\mathcal{L}_{W_q}$ be as above, then $\dim\{Y\in H^1: \exists \lamed<0, \mathcal{L}_{W_q} Y=\lamed Y \}<\infty$.
	\end{lemma} 
	\begin{proof}
		a) This follows from the theory of regular singular points for ODEs.
		
		b) This is a simple consequence of the compact embedding $\dot{H}^1\to L^2(\R^3,\jpns{r}^{-4}dx)$, similar to the discussion after Claim 3.5 in \cite{duyckaerts_solutions_2016}.
	\end{proof}
	
	As these are all the properties that the we use in the proof of \cref{main theorem}, we can also show
	\begin{theorem}
		Given scattering data $\psi^{dat}$ as in \cref{main theorem}, there exist a scattering solution $\phi=W_q+\psi$ in a neighbourhood of $i_+$ for
		\begin{equation}
			\begin{gathered}
				\Box\phi = -\phi^7+\phi^q
			\end{gathered}
		\end{equation}
		such that $\psi$ satisfies \cref{eq:mainTheorem:decay,eq:mainTheorem:trace}.
	\end{theorem}

	\subsubsection{Forward problem and assumed decay rates}\label{introdcution:forward}
	A closely related problem to \cref{main theorem} is the following
	\begin{theorem}[orbital stability, \cite{krieger_center-stable_2015}]\label{forward:theorem}
		There exists a smooth codimension one manifold of initial data for
		\begin{equation}
			\begin{gathered}
				\Box\phi=-\phi^5\\ \phi|_{t=0}=\phi_0,\, \partial_t\phi|_{t=0}=\phi_1 
			\end{gathered}
		\end{equation}
		in a neighbourhood of $(\phi_0,\phi_1)=(W,0)$ with respect to the energy topology that leads to a future global solution that stays close to the family of boosted and scaled solitons.
	\end{theorem}

	Note, that this result is a consequence of a classification and is not an asymptotic stability result like \cite{krieger_focusing_2007}. In particular, the authors do not track the modulation parameters of the soliton or the energy and size of the error term. Indeed, it was shown by \cite{donninger_nonscattering_2013} that asymptotic stability cannot hold in the energy topology, for further details see \cref{sec:intro:blow_up_at_infinity}. To the best knowledge of the author, there is no asymptotic stability result for a single soliton for \cref{main equation} outside symmetry classes, and no result concerning the dichotomy of asymptotic and orbital stability. We pursue this problem in future work. Let us comment on the necessary ingredients required for such stability problems.
	
	Following the general strategy outlined in \cite{dafermos_quasilinear_2022}, one needs to acomplish the following steps
	\begin{enumerate}[noitemsep]
		\item Prove energy boundedness for the linear problem.\label{forward: boundedness}
		\item Prove integrated local energy decay (Morawetz estimate)\label{forward: morawetz} for the linear problem.
		\item Obtain decay via $r^p$ estimate for the linear problem.
		\item Close error terms\label{forward: close errors}.
		\item Modulate the equation at the end of each dyadic interval so that the soliton is unit size located at the origin with no momentum.\label{forward: zero mode}
	\end{enumerate}
	In this paper, we achieve \cref{forward: boundedness,forward: zero mode}. As the $r^p$ estimate is an exterior estimate, it is extremely robust, the missing ingredients are \cref{forward: morawetz,forward: close errors}. 
	
	\begin{remark}
		Note, that \cref{forward: boundedness} was already proved in \cite{krieger_focusing_2007} in the radial case. Their method relies on finding an explicit form of the resolvent, which involves a similar projection for the scaling content of the solution that we find in \cref{linear scale conservation with force}.
	\end{remark}

	We comment on, whether given \cref{forward: morawetz} one could achieve stability. To obtain control for $\mathcal{E}$, the decay of modulation parameters has to be at least as fast as $\mathcal{E}^V$. The respective error for these quantities in a dyadic slab are
	\begin{equation}
		\begin{gathered}
			\mathfrak{Err}_{\mathcal{E}}\sim\int_{\rdom}\partial_t\psi F,\quad \mathfrak{Err}_{\text{mod}}\sim\Big(\int_{\rdom}(\tau_2-\tau)F\Big)^2.
		\end{gathered}
	\end{equation}
	The first one can be bounded by bulk integrals, so no retarded time decay is necessary. For the second, we can still use bulk integrals to control $\mathfrak{Err}_{\text{mod}}\sim \mathcal{E}^2{\tau_1}^2$. We can check that $\mathfrak{Err}_{\text{mod}}\lesssim \mathcal{E}$ whenever $\mathcal{E}\sim{\tau_1}^{-p}$ for $p>2$. Note, that the original $r^p$ method would only work for $p\leq2$, but by the improvements in \cite{schlue_decay_2013,moschidis_rp_2016}, we except the $p>2$ regime to be achievable \footnote{see \cite{luhrmann_stability_2022} 1.8 for further discussion on decay when there are 0 modes present}. In the present paper, we use a much stronger decay, because we do not prove \cref{forward: morawetz}.
	
%
%
%
	
	\subsubsection{Blow up at infinity}\label{sec:intro:blow_up_at_infinity}
	
	In \cite{donninger_nonscattering_2013} Donninger-Krieger showed that there exist global in time solutions to \cref{main equation} that approach a soliton $W$ with a scale parameter $\lambda(t)$ that converges to 0 or blows up as $t\to\infty$. The method of proof is somewhat reminiscent to the construction that we perform in the polyhomogeneous part of the paper in the following sense. The authors first solve an elliptic problem to get an improved ansatz and then solve away the error. The first step is only performed once and so the ansatz still has an error term with relatively slow ($t^{-4}$) decay rate, but it is sufficient to close the estimates. Regarding the nonlinear part, Donninger-Krieger utilise Fourier transform in spatial variables with respect to the operator $\Delta+V$, called distorted Fourier transform, to prove decay estimates for the linearised problem reminiscent to \cite{krieger_focusing_2007}. They emphasise that the feasibility of this approach relies on the spherically symmetric nature of the problem.
	
	Even though the solutions constructed in \cite{donninger_nonscattering_2013} are not expected to be unique, just as for the solutions of \cref{main theorem}, the following is a natural conjecture regarding their behaviour
	\begin{conjecture}
		For $\abs{\alpha}$ sufficiently small the following holds.
		There exists a blow-up of the compactification of Minkowski space ($\mathfrak{M}_\alpha$), such that \cref{main equation} admits solution $\phi$ that is conormal on $\mathfrak{M}_\alpha$ and satisfies
		\begin{equation}
			\norm{\phi(t,\cdot)-t^{\alpha/2}W(t^{\alpha}\cdot)}_{L^\infty}\to0\, \text{ as }t\to\infty,\qquad \partial_u(r\phi)|_{\scri}=0.
		\end{equation}
	\end{conjecture}
	
	One could similarly investigate the  global regularity of solutions constructed in \cite{yuan_construction_2022}, or excited analogue of \cite{jendrej_construction_2020} in 4 dimensions.
	
	\subsection{Idea of the proof and outline}\label{sec:intro_ideas}
	
	In this subsection, we only discuss the \textit{linear} problem \cref{linearised eom}. The reason for this is the recent development of techniques by Dafermos-Holzegel-Rodnianski-Taylor in \cite{dafermos_quasilinear_2022}.  The authors develop a robust framework to study stability problems around localised perturbations asymptotically flat spacetimes by proving estimates in dyadic hyperboloidal regions. Global existence is in turn a consequence of dyadic iteration. More importantly, the stability in such dyadic regions is a direct consequence of robust linear estimates\footnote{ one actually needs to study the inhomogeneous problem, i.e. replacing 0 with some forcing $F$ in \cref{linearised eom}} plus some soft structural conditions on the nonlinear terms. Their approach is a guiding principle of our work. Let us turn to the linear problem.
	
	\subsubsection{Energy estimates}
	As mentioned previously we believe that energy boundedness is crucial to prove \cref{main theorem}, but the lack of coercivity in $\mathcal{E}^V$ renders a $\partial_t$ energy estimate useless on its own. The obstruction from unstable modes will be treated, as explained in \cref{subsec:obstruction_to_decay}, by a localised projection of the form
	\begin{equation}
		\begin{gathered}
			\alpha_\pm =\int \chi_R Y(\partial_t\pm\lamed)\psi
		\end{gathered}
	\end{equation}
	for a cutoff function, $\chi_R$, localising to $\{\abs{x}\leq R\}$. These will not be conserved in evolution. However, $\alpha_+$ decays exponentially in the backward direction, up to corrections by $e^{-R}\phi$ terms. In turn it is easy to see that $\alpha_+$ will stay about the same order as $\phi$. The unstable mode $\alpha_-$ will grow exponentially provided nonzero initial condition, plus correction of the form  $e^{-R}\phi$. Therefore, we must finely tune its initial value via a topological argument to get bounded solutions. This step is detailed in \cref{sec:unstable_mode}.
	
	Let us turn to the elements in $\{Y^{\ker}\}=\ker(\Box+V)$. Note, that these are not only obstruction to decay of the solution (e.g. $\partial_iW$), but also to boundedness (e.g. $t\partial_i W$). We intend to find projection operators $\Theta^{\bullet}$ for $\bullet\in\{mom,com,\Lambda\}$ --i.e. functionals from scattering solutions to $\R$-- such that for
	\begin{equation}
		\begin{gathered}
			\psi-\Theta^{com} Y^{\ker}_{com}-\Theta^{mom} Y^{\ker}_{mom}-\Theta^{\Lambda} Y^{\ker}_{\Lambda}
		\end{gathered}
	\end{equation}
	the energy functional $\mathcal{E}^V$ becomes coercive. We immediately note, that we do not need to project $t\Lambda W$ out. This is treated by insisting that $\psi$ decays toward $\scri$. 
	
	The momentum ($t\partial_i W$) and center of mass ($\partial_iW$) are projected similarly. Since \cref{main equation} is Poincaré invariant, it enjoys nonlinearly conserved quantities associated to translations (generated by $\partial_i$) and Lorentz boost (generated by $t\partial_i+x_i\partial_t$) called momentum and center of mas respectively. Linearising these quantities around the soliton $W$ yield functionals that are conserved for \cref{linearised eom}. Moreover we can evaluate $\Theta^{mom},\Theta^{com}$ on $t\partial_i W,\partial_iW$ and find that they yield non-vanishing terms. 
	
	In order to project out the scaling mode, $\Lambda W$, we make the observation that $t\Lambda W\sim1$ near null infinity. This in turn yields that $r\partial_v t\Lambda W$ decays much faster than $t\Lambda W$ and we can plug it in the linearised bilinear energy momentum tensor. Since $(\Box+V)t\Lambda W=0$ the associated $T$ energy current will be conserved. We simply set $\Theta^\Lambda[\psi](\tau)$ to be the value of this current through a hyperboloid $\Sigma_\tau$. We compute that $\Theta^\Lambda[\Lambda W]\neq0$. This concludes energy boundedness. The precise definition and the properties of $\Theta^\bullet$ are described in \cref{sec:zero_modes}.
		
	\subsubsection{Polyhomogeneity}
	To prove polyhomogeneity, we follow the general strategy employed by Hintz in \cite{hintz_gluing_2023} for gluing constructions in general relativity. For a comprehensive guide to this approach, see \cite{hintz_lectures_2023}.
	
	Let's assume that the non-linear solution $\phi$ to \cref{main equation} is polyhomogeneous on the compactification drawn in \cref{fig:soliton_blowup}. Here, $\mathcal{F}=\{\abs{x}<\infty,t=\infty\},I^+=\{\abs{x}/t\in(0,1),t=\infty\}$. In particular, we can write
	\begin{equation}
		\begin{gathered}
			\phi=\sum_{i,j}t^{-i}\log^j t\phi^+_{i,j}(x/t)+\mathcal{O}(t^{-N}),\qquad (x,t)\in\mathcal{N}_+.
		\end{gathered}
	\end{equation}
	For the sake of discussion let us focus on the terms with $j=0$. In the linearised operator $\Box+V$ we have, $\Box,V$ with scaling dimensions $-2$ and $-4$ respectively. Therefore, in this region $V$ is going to be a perturbation. We compute $\Box t^{-i}\phi^+_{i}=t^{-i-2}(N_i\phi^+_{i})(x/t)$ for some elliptic operator $N_i$\footnote{$N_i$ is just the Laplace-Beltrami operator of hyperbolic space in appropriate coordinates, see \cite{baskin_asymptotics_2015}}. In order to find $\phi^+_i$, we simply need to invert this elliptic operator, $N_i$, with boundary conditions and inhomogeneities coming from the radiation prescribed at $\scri$ and $\phi^5$ respectively. The $\phi^+_i$ found produces an error term $t^{-i}\phi^+_i(0)V(x)$. In particular, solving for an error $t^{-i-2}$ at $I^+$ we get and error $t^{-i}$ at $\mathcal{F}$. 
	
	Similarly, we can write
	\begin{equation}
		\begin{gathered}
			\phi=\sum_{i,j}t^{-i}\log^j t\phi^\mathcal{F}_{i,j}(x)+\mathcal{O}(t^{-N}),\qquad (x,t)\in\mathcal{N}_{\mathcal{F}}
		\end{gathered}
	\end{equation}
	We compute the action of the linear operator $(\Box+V)t^{-i}\phi^\mathcal{F}_i=t^{-i}(\Delta+V)\phi^\mathcal{F}_i+\mathcal{O}(t^{-i-1})$. Let's focus on the $l=0$ spherical harmonic part of the problem, for simplicity. Proceeding as before, we attempt to solve 
	\begin{equation}\label{eq:intro_pol_F}
		\begin{gathered}
			(\Delta+V)\phi^\mathcal{F}_i=f^\mathcal{F}_i
		\end{gathered}
	\end{equation}
	for some error term $f^\mathcal{F}_i$ coming from previous parts of the iteration. Since $(\Delta+V)$ is not invertible, this is not possible, unless $(f^F_i,\Lambda)_{L^2}=0$. We may attempt to enforce this orthogonality by adding $ct^{-i+2}\Lambda W$ to the solution, so that we get an extra $ct^{-i}\Lambda W$ in the error term. However, note that $\Lambda W\notin L^2$, and even worse, the error produced at $I^+$ is with decay $t^{-i-1}$ at $I^+$, creating even worse error on $\mathcal{F}$ when solved away.
	
	We instead correct our solution with $ct^{-i+1}\Lambda W$. Next, we solve away the error on $I^+$. The error thus produced on $\mathcal{F}$ is simply $ct^{-i}V(x)$. We can use this to correct $f^\mathcal{F}_i$ to ensure orthogonality, and solve \cref{eq:intro_pol_F}. The error produced on $I^+$ decays as $t^{-i-3}$. Iteratively solving away the above two errors yields the polyhomogeneous error. The detailed construction is in \cref{sec:polyhom}.
	
	\begin{figure}[h]
		\centering
		\resizebox{.5\textwidth}{!}{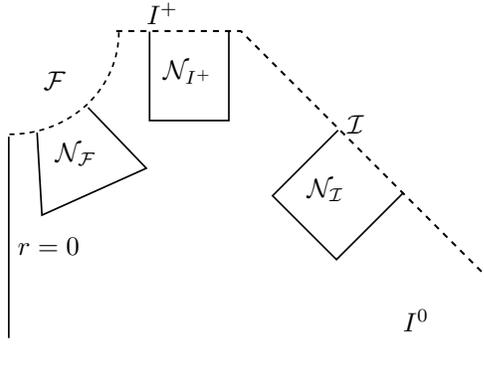}
		\caption{A compactification of Minkowski space on which scattering solution $\phi$ is conormal}
		\label{fig:soliton_blowup}
	\end{figure}
	
	\paragraph{Overview} The rest of the paper is divided into 3 parts. Until \cref{sec:nonlinear}, we are concerned with the proof of \textit{a)} of \cref{main theorem}. \cref{sec:polyhom,sec:cauchy} give a proof of \textit{b)} of \cref{main theorem} and \cref{thm:trivial_scattering} respectively, and are mostly unrelated to the rest of the paper. Let us give more detail on the first part. In \cref{sec:setup}, we give the definition of all geometric quantities used throughout the paper. Then, in \cref{sec:spectral}, we quote the necessary spectral information on the linearised operator. We also prove some weighted Sobolev embeddings for nonlinear applications. The next part, \cref{sec:linear theory section} is the main novelty of the current work. We introduce projection operators and prove a conditional energy boundedness for the linearised operator. We use this in \cref{sec:nonlinear} to construct the solution 
	
	\section{Geometric set up and notation}\label{sec:setup}
	In this section we are going to introduce the notation used in the introduction and the rest of the paper as well. Furthermore, we will introduce the necessary geometric framework.
	
	We will use the notation $A\lesssim_{a,b,c...} B$ to say that there exists a constant $C$ depending on $a,b,c...$ such that $A\leq C B$. We will abbreviate $\partial f\cdot\partial g=\eta^{\mu\nu}\partial_\mu f\partial_\nu g$ where $\eta$ is the Minkowski metric.
	
	We start with introducing geometric qunatities related to solutions of the linearised equation \cref{linearised eom}. We define the energy momentum tensor
	\begin{equation}\label{energy mom tensor}
		\begin{gathered}
			\T^w_{\mu\nu}[\psi]=\partial_\mu\psi\partial_\nu\psi-\frac{\eta_{\nu\mu}}{2}(\partial\psi\cdot\partial\psi-w\psi^2),\qquad \partial\psi\cdot\partial\psi=\partial_\sigma\psi\partial^\sigma\psi.
		\end{gathered}
	\end{equation}
	Note that for $\psi$ a solution of \cref{linearised eom}, $\text{div}(\T^V[\psi])=\text{grad}(V)(\psi)^2$. We also introduce the bilinear energy momentum tensor
	
	\begin{equation}\label{bilinear energy mom tensor}
		\begin{gathered}
			\T^w_{\mu\nu}[f,g]=\partial_{(\mu}f\partial_{\nu)}g-\frac{\eta_{\mu\nu}}{2}(\partial f\cdot\partial g-wfg).
		\end{gathered}
	\end{equation}
	
	Following \cite{holzegel_boundedness_2014} and \cite{dafermos_quasilinear_2022} we introduce the twisted energy momentum tensor
	\begin{equation}
		\begin{gathered}
			\tilde{\T}^w_{\mu\nu}[\psi]=\tilde{\partial}_\nu\psi\tilde{\partial}_\mu\psi-\frac{\eta_{\mu\nu}}{2}\big(\tilde{\partial}\psi\cdot\tilde{\partial}\psi-w\psi^2+w'\psi^2\big)\\
			\tilde{\partial}(\cdot)=\beta\partial(\beta^{-1}\cdot ),\quad w'=-\frac{\Box\beta}{\beta}
		\end{gathered}
	\end{equation}
	Although $\tilde{\T}$ is not divergence free, we have the following result from \cite{holzegel_boundedness_2014} Proposition 3
	\begin{equation}
		\begin{gathered}
			\partial^\mu (X^\nu \tilde{\T}^w_{\mu\nu}[\psi])=X^\nu S_\nu[\psi]+\psi^2 X^\nu \partial_\nu w\\
			S_\nu[\psi]=-\frac{\beta^{-1}\partial_\nu (\beta w')}{2\beta}\psi^2+\frac{\beta^{-1}\partial_\nu\beta^2}{2\beta}\tilde{\partial}_\sigma\psi\cdot\tilde{\partial}^\sigma\psi
		\end{gathered}
	\end{equation} 
	where $X$ is a Killing vector field. When using $\tilde{\T}$, we will take $X=T$, $\beta=\jpns{r}^{-1}$ which yields $X^\nu S_\nu=0$ and $w'=3\jpns{r}^{-4}$.
	
	As we are looking for a solution perturbatively around $W$, we will write $W+\psi$ for a solution to \cref{main equation} and think of the nonlinear in $\psi$ terms as error and write
	\begin{equation}
		\begin{gathered}
			\mathcal{N}[\psi]=(W+\psi)^{5}-W^5-5W^4\psi\implies (\Box+V)\psi=\mathcal{N}[\psi].
		\end{gathered}
	\end{equation}
	\subsection{Foliation}\label{sec:foliation}
	Our foliations are going to be composed of a flat and a null piece following the framework of \cite{dafermos_quasilinear_2022}. Furthermore, we will have 2 radii $R_1<R_2$ labeling different regions of the flat part. $R_1$ will denote the region where we detect unstable modes and where the potential coming from the soliton destroys coercivity, while $R_2$ is the transition region to the null part.
	
	We use the usual null coordinates
	\begin{equation*}
		\begin{gathered}
			v=\frac{t+r}{2}\qquad u=\frac{t-r}{2}.
		\end{gathered}
	\end{equation*}
	Let $\bar{\chi}$ be a cutoff function such that $\bar{\chi}|_{\{x<1\}}=1,\, \bar{\chi}_{\{x>3\}}=0,\, \bar{\chi}'\in(1,3),\bar\chi'|_{(1,3)}\neq0$. Set $\chi_R=\bar{\chi}(\frac{x}{R})$ and $\chi^c_R=1-\chi_R$. The function $h(x)=\int_{-\infty}^xdy \bar{\chi}^c(y)$ satisfies: $h|_{y\leq1}=0,\, h|_{y\geq3}=h(3)+y,\, h'\in[0,1]$. Define new time coordinate $t_\star=t-h(r-R_2)$. We fix the vector fields
	\begin{equation}
		\begin{gathered}
			T=\partial_t|_r=\partial_{t_\star}|_r,\quad \Omega_{ij}=x_i\partial_j-x_j\partial_i,\quad  S=t\partial_t|_r+r\partial_r|_t,\quad X=\rtStar,\quad X^\circ=\partial_r|_t.
		\end{gathered}
	\end{equation}
	We will also use $\Lambda=\frac{1}{2}+S$, but only as it acts on $W$. In $(t_\star,x)$ coordinates the wave operator and partial derivatives take the form 
	\begin{equation}\label{wave operator in good coordinates}
		\begin{gathered}
			\Box=-(1-h'^2)T^2+(X-2h'T)X+h''T+\frac{2}{r}(X-h'T)+\frac{\slashed{\Delta}}{r^2},\\
			X^\circ=X-h'T,\quad\partial_u|_v=T-X^\circ=(1+h')T-X,\quad \partial_v|_u=T+X^\circ=(1-h')T+X.
		\end{gathered}
	\end{equation}
	We fix a foliation and hypersurfaces
	\begin{equation}\label{eq:notation:hypersurfaces}
		\begin{gathered}
			\Sigma_\tau^{v_\infty}=\{t_\star=\tau\}\cap\{t+r<v_\infty\},\qquad \Sigma_\tau=\Sigma^\infty_\tau\\
			\scri^{v_\infty}_{\tau_1,\tau_2}=\{(t+r)=v_\infty,t_\star\in[\tau_1,\tau_2]\},\qquad \scri_{\tau_1,\tau_2}=\scri^\infty_{\tau_1,\tau_2}\\
			\mathcal{R}^{v_\infty}_{(\tau_1,\tau_2)}=\bigcup_{\tau\in[\tau_1,\tau_2]}\Sigma^{v_\infty}_\tau,\qquad \mathcal{R}_{(\tau_1,\tau_2)}=\mathcal{R}^\infty_{(\tau_1,\tau_2)},\quad \tau_1\leq\tau_2\\
			\underline{\mathcal{C}}_x=\{v=x\},\quad \tilde{\underline{\mathcal{C}}}_x=\underline{\mathcal{C}}_x\cap\{r>R_2\}\cap\mathcal{R}_{\tau_1,\tau_2}.
		\end{gathered}
	\end{equation}
	Note, that $\scri_{\tau_1,\tau_2}\simeq[\tau_1,\tau_2]\times S^2$ is not part of the spacetime, though it is a useful idealised boundary. For an alternative approach to $\scri$, see \cref{sec:polyhom:geometric}.  Here, $\tilde{\underline{\mathcal{C}}}$ contains $\tau_1,\tau_2$ implicitly and these boundaries will refer to the current region under discussion and should be clear from context, when not, we will write $\underline{\mathcal{C}}^{\tau_1,\tau_2}_x:=\tilde{\underline{\mathcal{C}}}_x$ . Also, note that for most of the paper $\tau_1\sim\tau_2$, but this will be stated clearly in all appearances. Similarly, $\tau_1\gg R_2$ everywhere outside \cref{sec:cauchy}. For a hypersurface $\Sigma$ and a current (1-form) $J$, let's denote
	\begin{equation}
		\begin{gathered}
			\Sigma[J]=\int_\Sigma J\cdot n
		\end{gathered}
	\end{equation}
	the flux of the current through $\Sigma$. A quick calculation (see \cref{energy calculation}) yields that 
	\begin{equation}\label{energy integral}
		\begin{gathered}
			\mathcal{E}^w[\psi]:=-\Sigma_\tau[T\cdot \T^w[\psi]]=\frac{1}{2}\int_{\Sigma_\tau\simeq\R^3}  \Big((1-(h')^2)(T\psi)^2+(X \psi)^2+\frac{\abs{\slashed{\nabla}\psi}^2}{r^2}-w\psi^2\Big)
		\end{gathered}
	\end{equation}
	In the integral above, the measure is simply the one from the isomorphism between $\Sigma_\tau$ and $\R^3$ given by $(t,x)\mapsto x$.
	We will use this convention for hypersurface integrals for the rest of the paper. 
	Note, that $\mathcal{E}^w[\psi]$ is not coercive, so we cannot use it to control the solution. Indeed, the main part of this work is concerned with obtaining estimates on $\mathcal{E}:=\mathcal{E}^0$ in terms of $\mathcal{E}^V$. 
	Let us also introduce the norm $\norm{u}_{\mathfrak{X}_1}=\norm{u(1-h'^2)^{1/2}}_{L^2}$ for $u\in\mathcal{C}^\infty(\R^3)$ corresponding to the $T\psi$ content of $\mathcal{E}^w[\psi]$.
	
	We are going to control the solution in energy spaces, and we introduce the following norm
	\begin{equation}
		\begin{gathered}
			\masterNum{0}[\psi]:=	\sup_{\tau\in(\tau_1,\tau_2)}-\Sigma_\tau[T\cdot (\T^0[\psi]+\tilde{\T}^0[\psi])]+ \sup_{x>R_1} -\tilde{\underline{\mathcal{C}}}_{(\tau_1+x)/2}[T\cdot(\T^0[\psi]+\tilde{\T}^0[\psi])].
		\end{gathered}
	\end{equation}
	The reason to introduce the second term is because $\mathcal{E}_\tau$ only controls tangential derivatives on the null part. We also introduce commuted analogue with $J[\psi]=-T\cdot(\T+\tilde{\T})$
	\begin{equation}
		\begin{multlined}
			\master[\psi]:=\sum_{\abs{\alpha}\leq k-1}	\sup_{\tau\in(\tau_1,\tau_2)}{\tau}^{2\alpha_0}\Sigma_\tau[J[\Gamma^\alpha\psi]]+ \sup_{x>\tau_2/2} \tilde{\underline{\mathcal{C}}}_{-x}[{\tau}^{2\alpha_0}J[\Gamma^\alpha\psi]]\\
			+\sup_{x>R_1} \tilde{\underline{\mathcal{C}}}_{(\tau_1+x)/2}[{\tau}^{2(k-1)}J[T^{k-1}\psi]] ,\quad \Gamma\in\{T,\Omega_{ij},S\}
		\end{multlined}
	\end{equation}
	\subsection{Symbols and extra notation}
	To keep the paper accessible and easily readable to a wider audience, we do no introduce too much extra notation to prove the existence part of \cref{main theorem}.	For the polyhomogeneity part, we will employ a significant amount of notation to keep the paper of reasonable length. These will be introduced exclusively in \cref{sec:polyhom}, sparing the reader from unnecessary complexity for the rest of the paper.
	
	For convenience of the reader, we collect all the symbols used in the paper below
	\begin{itemize}\setlength\itemsep{-0.4em}
		\item $\bar{\chi},h$ cutoff functions
		\item $u,v,t_\star$ coordinates on spacetime
		\item $W$ soliton at a particular scale, $V$ potential created by soliton
		\item $\T^w$ energy momentum tensor with potential $w$, $\T=\T^0$.
		\item $T,\Omega,S,\Lambda$ vector fields;
		\item $\Sigma_\tau,\underline{\mathcal{C}},\scri_{\tau_1,\tau_2}$ hypersurfaces, $B_r:=\{x\in\R^3:\abs{x}<r\}$ is the open ball of radius $r$
		\item $\mathcal{X},\mathcal{E}^w,\mathfrak{X}$ are energy quantities
		\item $\lamed$ eigenvalue of the linearised problem with $Y$ as eigenfunction
		\item $J,\tilde{J},\Gamma$ will be placeholders for currents and vector fields and will be defined locally within parts of the paper, such as a subsection, and reused with new definition at later point
		\item $P^{S^2}_l$ is the projection operator acting on $L^2$ functions projecting onto the $l-$th spherical harmonic
	\end{itemize}
	Let us note, that we are always going to suppress the dependence on $R_1,R_2$ in the notation $\lesssim$, whenever the estimate holds uniformly for $R_1,R_2\to\infty$.
	\section{Energy estimates and spectral properties of $W$}\label{sec:spectral}
	In this section, we will recall some of the known result on how to deal with the non-coercivity of $\mathcal{E}^V$, and prove some basic inclusions of the space $\master$ into weighted spaces. 
	
	\subsection{Symmetries}
	Let's introduce the space of obstruction to decay for the linearised equation coming from the symmetries of \cref{main equation}
	
	\begin{equation}
		\begin{gathered}
			\mathcal{Z}=\text{span}\{\partial_j W,\,\frac{1}{2} W+x\cdot \nabla W\,,-2x_jW+\abs{x}^2\partial_jW-2x_jx\cdot\nabla W\}
		\end{gathered}
	\end{equation}
	corresponding to translational, scaling and special conformal symmetries of the equation. As $W$ is invariant under Kelvin transform, the special conformal symmetries coincide with translations, in particular, we have
	\begin{equation}
		\begin{gathered}
			\mathcal{Z}=\text{span}\{\partial_j W,\, \Lambda W\},
		\end{gathered}
	\end{equation}
	where we used $\Lambda=1/2+S$.
	
	\subsection{Coercivity}

	In this section, we summarise some of the spectral results discussed in \cite{duyckaerts_dynamics_2008}.
	
	We want to find suitable energy boundedness statements for the linearised operator
	\begin{equation}
		\begin{gathered}
			\Box+V=-\partial_t^2+\Delta+5W^4=-\partial_t^2+H.
		\end{gathered}
	\end{equation}
	Boundedness is not imminent because of positive eigenvalues of $H$, $E=\{f\in \mathcal{C}^\infty(\R^3) | Hf=\lambda f, \lambda>0\}$, and kernel elements, $\ker H=\{f\in\dot{H}^1:Hf=0\}$. These solutions make the bilinear form $(u,v)\mapsto(Hu,v)_{L^2}$ non definite. Indeed, we have the following results from \cite{duyckaerts_dynamics_2008} and \cite{duyckaerts_solutions_2016}.
	
	\begin{lemma}\label{coercivity cited}
		a) 
		$\mathcal{Z}=\ker H$ , and $E=span\{Y\}$  with $Y\in\mathcal{C}^\infty$ and $\abs{Y(r)}\lesssim\exp(-\lamed r)$ for $\lamed>0$ the eigenvalue $HY=\lamed Y$.  (\cite{duyckaerts_dynamics_2008} section 5.5)
				
		b) For $E_i\in \mathcal{C}^\infty_0(\R^4)$ with 
		\begin{equation}
			\begin{gathered}
				(E_i,V_j)_{L^2}=\delta_{ij}\qquad (E_i,Y)=0,\qquad \text{ for } V_j\in\ker H
			\end{gathered}
		\end{equation}
		the following holds. There exists $\mu>0$ depending on $E_i$  such that (\cite{duyckaerts_solutions_2016} Proposition 3.6)
		\begin{equation}
			\begin{gathered}
				(-Hf,f)\geq \mu \norm{f}_{\dot{H}^1}^2-\frac{1}{\mu}\Big(\sum (E_i,f)_{L^2}^2+(Y,f)_{L^2}^2\Big).
			\end{gathered}
		\end{equation}
	\end{lemma}
	We will need to use a slight modification of the above lemma. 
	\begin{lemma}\label{coercivity lemma}
		Fix notation as in \cref{coercivity cited}. Given $\tilde{E}_i\in (\dot{H}^1)^\star$ with 
		\begin{equation}
			\begin{gathered}
				\jpns{\tilde{E}_i,V_j}=\delta_{ij},\qquad \jpns{\tilde{E}_i,Y}=0
			\end{gathered}
		\end{equation}
		and $R$ sufficiently large, there exists $\mu$ (depending on $\tilde{E}_i,R$) such that the following inequality holds
		\begin{equation}\label{coercivity equation r}
			\begin{gathered}
				\frac{1}{\mu}\Big(\sum_i \jpns{E_i,f}^2+(Y\chi_R(r),f)_{L^2}^2\Big)+(-Hf,f)\geq \mu \norm{f}_{\dot{H}^1}^2+\mu\norm{f/\jpns{r}}_{L^2}.
			\end{gathered}
		\end{equation}
	\end{lemma}
	\begin{proof}
		The proof without the cutoff and the $\jpns{r}^{-1}f \in L^2$ control is essentially the same and we refer the reader to \cite{duyckaerts_solutions_2016}.
		The zeroth order term is controlled via Hardy inequality.
		To see that a sufficiently large cutoff can be introduced, we use Hardy inequality to get
		\begin{equation*}
			\begin{gathered}
				(Y \chi^c_R(R),f)^2_{L^2}\leq c e^{-\lamed R/2} \norm{f}_{\dot{H}^1}\\
				\implies (-Hf,f)\geq \mu \norm{f}_{\dot{H}^1}^2-\frac{1}{\mu}\Big(\sum_i \jpns{E_i,f}^2+(Y,f)_{L^2}^2\Big)\\ \geq (\mu-ce^{-\lamed R/2})\norm{f}_{\dot{H}^1}^2-\frac{1}{\mu}\Big(\sum_i \jpns{E_i,f}^2+(Y\chi_R(r),f)_{L^2}^2\Big).
			\end{gathered}
		\end{equation*}
	\end{proof}

	\paragraph{Uniform coercivity estimates}
	In our proof, the foliation constructed in \cref{sec:foliation} depends on $R_2$, and so will the projector that we use.
	We need to keep $R_2$ of variable size, to close a topological argument, see \cref{sec:unstable_mode}.
	Therefore, we need to keep track of the $R_2$ dependence of the constants appearing in \cref{coercivity lemma}.

	\begin{lemma}\label{en:lemma:Hilbert1}
		Let $(\mathcal{H},(,))$ be a Hilbert space.
		Let $A$ be a positive semidefinite quadratic form with $\norm{A}\leq 1$ and a one dimensional kernel spanned by $v_0\in\mathcal{H}$, $\norm{v_0}=1$.
		Let $\Theta\in\mathcal{H}^\star$ with $\abs{\Theta(v_0)}/\norm{\Theta}=c_\Theta>0$.
		Then the following are equivalent
		\begin{enumerate}
			\item $A$ defines an equivalent norm to $(,)$ on $\Theta^\perp$, i.e. $A(v,v)\geq c_1(v,v)$ for $\Theta(v)=0$\label{en:item:Theta_perp}
			\item $A$ defines an equivalent norm to $(,)$ on $v_0^\perp$, i.e. $A(v,v)\geq c_2(v,v)$ for $(v,v_0)=0$\label{en:item:v_perp}.
		\end{enumerate}
		Moreover, $c_1\sim_{c_\Theta} c_2$ with implicit constant growing at most linearly in $c_\Theta^{-1}$.
	\end{lemma}
	\begin{proof}
		Let's consider the map
		\begin{nalign}\label{en:eq:Phi}
			\Phi:\Theta^\perp&\to v_0^\perp,& \Phi^{-1}:v_0^\perp&\to\Theta^\perp,\\
			w&\mapsto w-\lambda_w v_0,& w&\mapsto w-\lambda'_w v_0\\
			\lambda_w&=(w,v_0)&\lambda'_w&=\frac{\Theta(w)}{\Theta(v_0)}
		\end{nalign}
		We compute that the norms of these maps
		\begin{nalign}
			&\norm{\Phi}^2=\sup_{\norm{w}=1}(\Phi(w),\Phi(w))=\sup_{\norm{w}=1}1-2\lambda_w(w,v_0)+\lambda_w^2=\sup_{\norm{w}=1}1-(w,v_0)^2\leq 1\\
			&\norm{\Phi^{-1}}^2=\sup_{\norm{w}=1}(w,w)-\frac{2\Theta(w)}{\Theta(v)}(w,v_0)+\Big(\frac{\Theta(w)}{\Theta(v)}\Big)^2\leq (1+c_\Theta^{-1})^2
		\end{nalign}
		Therefore, $\Phi$ is a bounded map with bounded inverse.
		
		We also compute $A(\Phi(w),\Phi(w))=A(w-\lambda_w v_0,w-\lambda_w v_0)=A(w,w)$.
		
		Assume \cref{en:item:Theta_perp}.
		For $v\in v_0^\perp$ with $(v,v)=1$, we compute
		\begin{nalign}
			(\Phi^{-1}(v),\Phi^{-1}(v))\geq\norm{\Phi}\geq1\underbrace{\implies}_{\cref{en:item:Theta_perp}}A(\Phi^{-1}(v))\geq c_1\implies A(v)\geq c_1.
		\end{nalign}
		We conclude that we can take $c_2\leq c_1$.
		
		For the other direction, we take $v\in\Theta^\perp$ with $(v,v)=1$ to compute
		\begin{nalign}
			(\Phi(v),\Phi(v))\geq\norm{\Phi^{-1}}\geq1+c_\Theta^{-1}\underbrace{\implies}_{\cref{en:item:v_perp}}A(\Phi^{-1}(v))\geq c_2 (1+c_\Theta^{-1})\implies A(v)\geq c_2(1+c_\Theta^{-1}).
		\end{nalign}
		We conclude that we can take $c_1\leq(1+c_{\Theta}^{-1})c_2$.
	\end{proof}

	\begin{cor}\label{en:cor:Hilbert2}
		Let $(\mathcal{H},(,)),A,v_0,\Theta,c_1$ be as in \cref{en:lemma:Hilbert1} and assume \cref{en:item:Theta_perp} holds.
		Then for $w\in\mathcal{H}$
		\begin{equation}\label{en:eq:abstract_coercivity}
			(w,w)\leq c_1^{-1}A(w,w)+(\Theta(w))^2(\Theta(v_0))^{-2}.
		\end{equation}
	\end{cor}
	\begin{proof}
		We define the projection $\Phi:\mathcal{H}\to \Theta^\perp$ with $w\mapsto w-\lambda'_w v_0$ for $\lambda'_w$ as in \cref{en:eq:Phi}.
		The estimate follows from the simple computation
		\begin{eqnarray}
			(w,w)\leq(w-\lambda v_0,w-\lambda v_0)+\lambda^2\leq c_1^{-1}A(w-\lambda)+(\Theta(w))^2(\Theta(v_0))^{-2}.
		\end{eqnarray}
	\end{proof}

	\begin{cor}\label{en:cor:R2_dependence_in_Theta}
		Let $(\mathcal{H},(,)),A,v_0,\Theta$ be as in \cref{en:lemma:Hilbert1}.
		Let $\tilde{\Theta}$ be another functional on $\mathcal{H}$ with corresponding constant $c_{\tilde{\Theta}}$.
		Assume \cref{en:eq:abstract_coercivity} holds.
		Then, there exists $c_2\leq_{c_{\Theta},c_{\tilde{\Theta}}} c_1$ such that 
		\begin{equation}
			(w,w)\leq c_2^{-1}A(w,w)+(\tilde{\Theta}(w))^2(\tilde{\Theta}(v_0))^{-2}.
		\end{equation}
		Moreover $c_2^{-1}$ grows at most inverse linearly with $c_{\tilde{\Theta}}^{-1}$.
	\end{cor}
	\begin{proof}
		We apply \cref{en:lemma:Hilbert1} twice, to connect the implicit constant to that in $v_0^{\perp}$, to get a statement like \cref{en:item:Theta_perp} for $\tilde{\Theta}$.
		Then we apply \cref{en:cor:Hilbert2} to conclude the result.
	\end{proof}

	\subsection{Weighted Sobolev embedding}
	In this section we prove some embeddings that are useful for non-linear applications
	\begin{lemma}\label{weighted estimates}
		Let $\psi\in\mathcal{C}^\infty_0(\rdom)$ with $\master[\psi]\leq 1$ for $k\geq 3$. Then the following holds for $\Gamma\in\{t_\star T,\Omega_{ij},S\}$
		\begin{subequations}
			\begin{gather}
				\abs{\psi}\lesssim \jpns{r}^{-1/2}\label{r-1/2 decay}\\
				\sum_{\abs{\alpha}\leq k}\rint\Big( \frac{(\Gamma^\alpha\psi)^2}{\jpns{r}^{2}\jpns{\tau_2-\tau}^{3/2}}+\frac{(T\Gamma^\alpha\psi)^2}{\jpns{r+\tau_2-\tau}^{3/2}}\Big)\lesssim 1\label{integrability}
			\end{gather}
		\end{subequations}
	\end{lemma}
	\begin{proof}
		We first write
			\begin{equation*}
			\begin{gathered}
				S=tT+r\partial_r|_t=(t-rh')T+rX=(t_\star+h-rh')T+rX
			\end{gathered}
		\end{equation*}
		where $\abs{h-rh'}\lesssim1$. Therefore, we get that
		\begin{equation*}
			\begin{gathered}
				\mathcal{E}[\psi]+{\tau}\mathcal{E}_\tau[T\psi]+\mathcal{E}_\tau[S\psi]\gtrsim \int_{\Sigma_\tau}\Big((XrX\psi)^2+\frac{1}{r^2}(\slashed{\nabla}rX\psi)^2\Big).
			\end{gathered}
		\end{equation*}
		Similarly, we get higher order estimate
			\begin{equation*}
			\begin{gathered}
				\sum_{n+m\leq k}{\tau}^n\mathcal{E}[T^nS^m\psi]\gtrsim_k\int_{\Sigma_\tau}\Big((X(rX)^k\psi)^2+\frac{1}{r^2}(\slashed{\nabla}(rX)^k\psi)^2\Big).
			\end{gathered}
		\end{equation*}
		Combining with $\Omega$ commutator bounds, by standard weighted Sobolev estimates we get \cref{r-1/2 decay}.
		
		Using that the twisted and untwisted current control undifferentiated quantities, we immedeatly have
		\begin{equation*}
			\begin{gathered}
				\int_{\Sigma_\tau}\frac{(\Gamma^\alpha\psi)^2}{\jpns{r}^2}\lesssim 1				
			\end{gathered}
		\end{equation*}
		which implies the first part of \cref{integrability}. For the second part, we split $T=\frac{1}{1+h'}(\partial_u|_v+X)$, which implies 
		\begin{equation*}
			\begin{gathered}
				\rint\frac{(T\psi)^2}{\jpns{r+\tau_2-\tau}^{3/2}}\leq 2\rint\frac{(\partial_u\psi)^2+(X\psi)^2}{\jpns{r+\tau_2-\tau}^{3/2}}\lesssim \mathcal{X}[\psi].
			\end{gathered}
		\end{equation*}
		For higher commutated quantities, we split the integration domain into $\rdom\cap \{r<\tau_1\}, \rdom\cap\{r>\tau_1\}$. In the first, we use the control from
		\begin{equation*}
			\begin{gathered}
				\sup_{x>\tau_2/2} \tilde{\underline{\mathcal{C}}}_{\tau_1+x}[{\tau}^{2\alpha_0}T\cdot\T[\Gamma^\alpha\psi]]
			\end{gathered}
		\end{equation*}
		to bound the commuted quantity similar to the uncommuted. In the second domain, we split cases. For $\abs{\alpha}=\alpha_0$, we may use the control provided by $\master$ to bound the integral similar to the uncommuted case. When $\Gamma^\alpha$ contain a term that is not $T$, we simply trade the $r$ weight with ${\tau}$ using that ${\tau}\gtrsim r$ in the domain of interest. Then, we use the control provided by $\mathcal{E}_\tau$ term for the estimate. For instance, when $\Gamma^\alpha=T^n\Omega$ we write
		\begin{equation*}
			\begin{gathered}
				\int_{\rdom\cap\{u>\tau_2\}}{\tau}^{2n}\frac{(T T^n\Omega\psi)^2}{\jpns{r+\tau_2-\tau}^{3/2}}\lesssim\int_{\rdom\cap\{u>\tau_2\}}{\tau}^{2n+2}\frac{(\partial_iT^{n+1}\psi)^2}{\jpns{r+\tau_2-\tau}^{3/2}}\lesssim\sup_\tau{\tau}^{n+1}\mathcal{E}_\tau[T^{n+1}\psi].
			\end{gathered}
		\end{equation*}
	
	\end{proof}

	\begin{lemma}\label{nonlinear bounds lemma}
		Let $\psi$ be a smooth function in $\rdom$ for $\tau_1\sim\tau_2$ with $\master[\psi]\leq1$ for $k\geq6$.
		Then we have the following bulk estimates
		
		\begin{subequations}
			\begin{gather}
				\sum_{\abs{\alpha}\leq k-1}\rint{t_\star}^{2\alpha_0}(T\Gamma^\alpha\psi)(\Gamma^\alpha \mathcal{N}[\psi])\lesssim \jpns{\tau_2-\tau_1}^{3/2},\label{nonlinear bounds 1st term}\\
				\int_{\rdom} \mathcal{N}[\psi]\frac{1}{\jpns{r}}\lesssim \jpns{\tau_2-\tau_1}^{3/2},\quad 				\int_{\rdom} \mathcal{N}[\psi]\frac{(t_\star-\tau_1)}{\jpns{r}^2}\lesssim \jpns{\tau_2-\tau_1}^{5/2}.
			\end{gather}
		\end{subequations}
		as well as boundary term estimates		
		\begin{subequations}
			\begin{gather}
				\sum_{l\leq k-1}\int_{\Sigma_\tau} \abs{t_\star}^l \abs{T^{l}\mathcal{N}[\psi]}\frac{1}{\jpns{r}}\lesssim 1\\
				\sum_{\abs{\alpha}\leq k-1}\int_{\Sigma_\tau}(\Gamma^\alpha \jpns{r}\mathcal{N}[\psi])^2\lesssim1\\
				\sum_{j\leq k}\int \abs{T^j\mathcal{N}[\psi]} e^{-r\lamed/2}\chi^c_{2R_1}\lesssim 1
			\end{gather}
		\end{subequations}	
	\end{lemma}
	\begin{proof}
		 We start with \eqref{nonlinear bounds 1st term} and use Cauchy-Schwartz with \cref{weighted estimates}
		\begin{equation*}
			\begin{gathered}
				\rint{\tau}^{2\alpha_0}(T\Gamma^\alpha)(\Gamma^\alpha\mathcal{N}[\psi])\leq \Big(\rint \frac{{\tau}^{2\alpha_0}(T\Gamma^\alpha\psi)^2}{\jpns{r+\tau-\tau_1}^{3/2}}\Big)^{1/2}\Big(\rint {\tau}^{2\alpha_0}\jpns{r+\tau-\tau_1}^{3/2}(\Gamma^\alpha\mathcal{N}[\psi])^2\Big)^{1/2}
			\end{gathered}
		\end{equation*}
		Consider each term in $\mathcal{N}[\psi]$ separately. Let $\alpha=\beta+\gamma$ with $\abs{\beta}\geq k/2,\abs{\gamma}\leq k/2$. The quadratic term is bounded by
		\begin{equation*}
			\begin{multlined}
				\rint {\tau}^{2\alpha_0}\jpns{r+\tau-\tau_1}^{3/2}(\jpns{r}^{-3}\Gamma^{\beta}\psi\Gamma^\gamma\psi)^2\\
				\lesssim \sup_{\rdom}\jpns{r}^{2-6}\jpns{\tau-\tau_1}^{3/2}\jpns{r+\tau-\tau_1}^{3/2}{\tau}^{2\gamma_0}(\Gamma^\gamma\psi)^2\rint {\tau}^{2\beta_0}\frac{(\Gamma^\beta\psi)^2}{\jpns{r}^{2}\jpns{\tau-\tau_1}^{3/2}}\\
				\lesssim\jpns{\tau-\tau_1}^3
			\end{multlined}
		\end{equation*}
		While the quintic is
				\begin{equation*}
			\begin{multlined}
				\rint {\tau}^{2\alpha_0}\jpns{r+\tau-\tau_1}^{3/2}(\Gamma^{\beta}\psi\Gamma^\gamma\psi^4)^2\\
				\lesssim \sup_{\rdom}\jpns{r}^{2}\jpns{\tau-\tau_1}^{3/2}\jpns{r+\tau-\tau_1}^{3/2}\underbrace{{\tau}^{2\gamma_0}(\Gamma^\gamma\psi)^8}_{\jpns{r}^{-4}}\rint {\tau}^{2\beta_0}\frac{(\Gamma^\beta\psi)^2}{\jpns{r}^{2}\jpns{\tau-\tau_1}^{3/2}}\\
				\lesssim\jpns{\tau-\tau_1}^3
			\end{multlined}
		\end{equation*}
		The intermediate terms are treated similarly. Putting the estimates together yields the first inequality. The proof for the rest of the bulk terms proceeds similarly with Cauchy-Schwartz inequality and \cref{weighted estimates}.
		
		For the boundary terms, we again will only prove the first inequality, the other follows similarly. 
		Using the pointwise estimate from \eqref{weighted estimates} and undifferentiated term control we get for all $l\leq k-1$
		\begin{equation*}
			\begin{gathered}
				\int_{\Sigma_\tau}\frac{\abs{t_\star}^l}{\jpns{r}}T^l\mathcal{N}[\psi]\lesssim\sum_{j\leq l}\int_{\Sigma_\tau}\abs{t_\star}^j(T^j\psi)\frac{1}{\jpns{r}^3}\lesssim\sum_{j\leq l}\Big(\int_{\Sigma_\tau}\frac{\abs{t_\star}^{2j}(T^j\psi)^2}{\jpns{r}^2}\Big)^{1/2}\lesssim1.
			\end{gathered}
		\end{equation*}
	\end{proof}

	\section{Scattering solutions}\label{sec:scattering}
	In this section, we set up the framework in which we are going to work and recall some standard existence results. For this section, there are some factors of $r^{1/2}$ and similar weights appearing in $L^2$ norms. This is a consequence of the relation between $L^2$ and $L^2_b$ spaces, where the latter is more closely related to $L^\infty$ estimates, see \cref{sec:polyhom}. 
	\begin{definition}\label{def:scattering definition}
		a) Fix $N\geq 0,\epsilon>0$ and $q\in\R^+,d\in[-1/2,1/2)$. We say that $\psi_k^\Sigma\in \mathcal{C}^\infty(\R^3), 0\leq k\leq N+1$ and $(r\psi)^{\scri},\in \mathcal{C}^\infty(\R^+_u\times S^2)$ are \textit{scattering data} for \eqref{main equation} respectively for \eqref{linearised eom2} in $\rdom$ if it holds that
		\begin{equation}
			\begin{gathered}
					\sum_{\abs{\alpha}\leq N+1}\norm{\jpns{u}^{-1/2}\jpns{u}^{q}\Gamma^\alpha (r\psi)^\scri}_{L^2(\scri)}<\epsilon,\quad\Gamma\in\{u\partial_u,\Omega_{ij}\}\\
				\sum_{k+\abs{\alpha}\leq N+1}\norm{r^{-2}\Gamma^\alpha r\psi^\Sigma_k }_{L^2(\Sigma)}+\sum_{\substack{k+\abs{\alpha}+i\leq N+1\\
						i\neq0}}\norm{\jpns{r}^{d-3/2}\Gamma^\alpha (\jpns{r}X)^i r\psi^\Sigma_k }_{L^2(\Sigma)}<\epsilon \jpns{\tau_2}^{-q},\quad \Gamma\in\{\Omega_{ij},\chi_1(r)\partial_i\}\\
			\end{gathered}
		\end{equation}
		Furthermore, we require the constraints
		\begin{subequations}\label{constraint non-linear}
			\begin{gather}
					(1-h'^2)\psi_k^\Sigma=(h''-2h'X-\frac{2h'}{r})\psi^\Sigma_{k-1}+(X^2+\frac{2}{r}X+\frac{\slashed{\Delta}}{r^2}+V)\psi^\Sigma_{k-2}+T^{k-2}F\quad 2\leq k\leq N+1
			\end{gather}
		\end{subequations}
		are satisfied. Here $F=\mathcal{N}[\psi]=(W+\psi)^5-W^5-5W^4\psi$ and $T^j\psi=\psi_j$ for \eqref{main equation}. We abbreviate $\underline{\psi}=(\psi^\Sigma_k,(r\psi)^\scri)$ and call the sum of the above two norms $\norm{\underline{\psi}}_{\mathcal{X}_{data}}$. Furthermore, for \eqref{linearised eom2} $F$ is part of the scattering data and we require 
		\begin{equation}
			\begin{gathered}
				\sum_{\abs{\alpha}+i_1\leq N}\norm{t^{1/2}\jpns{r}^{d-1}\Gamma^\alpha (\jpns{r}X)^{i_1} F}_{L^2(\rdom)}<\epsilon,\quad \Gamma\in\{\Omega_i,t_\star T\}
			\end{gathered}
		\end{equation}
		We say that the scattering data has order $N$, decay $q$ and size $\epsilon$.
		
		b) We say that $\psi\in\mathcal{C}^1(\mathcal{R}_{\tau_1,\tau_2})$ is a \textit{scattering solution} associated to the scattering data $\underline{\psi}$ (and $F$) if $\psi$ solves \eqref{main equation} (or \eqref{linearised eom2}) and 
		\begin{equation}
			\begin{gathered}
				\lim_{r\to\infty} r\psi(r,\tau,\omega)= (r\psi)^\scri(\tau,\omega),\quad \text{in }L^2\\
				\psi|_{\Sigma_{\tau_1}}=\psi_0^\Sigma,\quad \partial_t\psi|_{\Sigma_{\tau_1}}=\psi_1^\Sigma
			\end{gathered}
		\end{equation}
	\end{definition}
	
	\begin{remark}
		In the above definition, $\psi_k$ represents the $k$-th time derivative of the spacetime function $\psi$ that we construct. We prescribe these higher derivatives because otherwise we would face loss of regularity when trying to recover them from transport equations along null surfaces. This can be seen from \cref{constraint non-linear} where on the null part, we get a transport equation for $\psi_{k-1}$ in terms of second derivatives of $\psi_{k}$. Let us also note that for a \textit{scattering solution} $\psi\in \mathcal{C}^k(\mathcal{R}_{\tau_1,\tau_2})$, $T^j\psi|_{\Sigma_\tau}$ give valid scattering data for any  $\tau\in(\tau_1,\tau_2)$.
	\end{remark}
	
	\begin{remark}
		Note, that at $\scri$ there are no constraint equations.
	\end{remark}
	The following is a standard local scattering result
	
	\begin{lemma}\label{existence of scattering solution}
		a) (Linear) Given scattering data $\underline{\psi},F$ for \eqref{linearised eom2}, there exists a unique scattering solution $\psi$ such that for all $\tau_1<\tau_2$ the following holds
		\begin{equation}\label{scattering boundedness}
			\begin{gathered}
				\sum_{\abs{\alpha}\leq N}\norm{r^{-2}t_\star^{-1/2}\Gamma^\alpha r\psi}_{L^2(\rdom)},\sum_{\substack{\abs{\alpha}+\abs{i}\leq N\\ i\neq0}}\norm{\jpns{r}^{d-3/2}t_\star^{-1/2}\Gamma^\alpha (\jpns{r}X)^{i}\jpns{r}\psi}_{L^2(\rdom)}<\infty,
				\quad \Gamma\in\{\Omega_i,\partial_\tau\}
			\end{gathered}
		\end{equation}
		furthermore
		\begin{equation}\label{scattering tracelessness}
			\begin{gathered}
				\norm{\partial_\tau^k\Gamma^\alpha (r\psi)|_{\scri_{\tau_1,\tau_2}}}_{L^2(\scri_{\tau_1,\tau_2})}:=\lim_{r\to\infty} \norm{\partial_\tau^k\Gamma^\alpha (r\psi)|_{\scri_{\tau_1,\tau_2}}}_{L^2(C_r)}<\infty,\quad \abs{\alpha}\leq N,\, k\in\{0,1\} \\
				\norm{\partial_\tau\Gamma^\alpha (\jpns{r}X)^{i}(r\psi)|_{C_r}}_{L^2(\scri_{\tau_1,\tau_2})}=0,\quad \abs{\alpha}\leq N \land i\neq 0
			\end{gathered}
		\end{equation}
		
		b) (Non-linear) Given scattering data $\underline{\psi}$ for \eqref{main equation} with regularity $N\geq 6$, and $\tau_2-\tau_1>0$ sufficiently small, there exists unique scattering solution such that \eqref{scattering boundedness} and \eqref{scattering tracelessness} hold.
	\end{lemma}
	
	\begin{remark}[$r^p$ estimates]
		The extended control over $rXr\psi$ is a standard form of an $r^p$ estimate, see \cref{polyhomogeneity theorem}. The maximum value of $d<1/2$ can be increased to $d_{max}=3$ after commuting with $r^2\partial_v$, see \cite{angelopoulos_asymptotics_2018}. The value of $d_{max}$ follows from the following simple computation. Given that $d>0$, we know that $\psi$ admits a radiation field, so $\mathcal{N}[\psi]\sim \jpns{r}^{-5}$. Furthermore, to leading order close to $\scri$ we have $\partial_u \partial_v (r\psi)=r\psi^5$. Therefore, if $r\partial_v(r\psi)\sim \mathcal{O}(r^{-d})$, than this property is conserved. Indeed, there is a conservation law --with associated quantity called Newman-Penrose charge-- at $\scri$ whenever the data is such that it can be peeled. For further information and the use of these quantities we refer to \cite{angelopoulos_late-time_2018,gajic_relation_2022}.
	\end{remark}
	
	\section{Linear theory}\label{sec:linear theory section}

	In this section, we study the behaviour of solutions to
	\begin{equation}\label{linearised eom2}
		\begin{gathered}
			(\Box+5W^4)\psi=F,\qquad F,\psi:\mathcal{R}_{(\tau_1,\tau_2)}\to \R\\
			(r\psi)|_{\scri_{(\tau_1,\tau_2)}}=\psi^{dat}, \, \psi|_{\Sigma_{\tau_1}}=\psi_0,\, \partial_t\psi|_{\Sigma^f_{\tau_1}}=\psi_1.
		\end{gathered}
	\end{equation}
	We introduce projection operators to the unstable mode associated to \cref{linearised eom2}
	\begin{equation}
		\begin{gathered}
			\alpha_\pm[\psi](\tau)=\int_{\Sigma_{\tau}} \chi_{R_1} Y(\partial_t\pm\lamed)\psi,
		\end{gathered}
	\end{equation}
	and write $\alpha_\pm=\alpha_\pm[\psi](\tau)$ whenever $\psi,\tau$ are clear from context. For $\psi$ a solution of \eqref{linearised eom2}, we have
	\begin{equation}\label{ode unstable mode}
		\begin{gathered}
			\partial_\tau \alpha_\pm=\pm \lamed\alpha_\pm+\mathfrak{Err}_{\pm}\\
			\mathfrak{Err}_{\pm}=\int_{\Sigma_\tau} \Big\{-\Big(Y\Delta\chi_{R_1}+ Y'\cdot\chi_{R_1}^{\prime}\Big)\psi-\chi_{R_1}^c Y F\Big\}.
		\end{gathered}
	\end{equation}
	The main result of the current section is the boundedness for the linear problem as captured by the existence of a master current in \cref{prop:master_current}.
	
	\subsection{Zero modes}\label{sec:zero_modes}
	In this section we introduce a method to detect the obstructions to the coercivity of $\mathcal{E}^V$. The idea here is to do a projection similar to \cite{dafermos_linear_2019}. In that work, all non-gauge zero mode solutions are supported on $l=0,1$ spherical harmonic. Similarly, in the current work, we know that on $l\geq2$ modes $\mathcal{E}^V$ is coercive, but unlike in \cite{dafermos_linear_2019} there are still radiative physical degrees of freedom left in the $l=0,1$ sector. To detect these in compact region as done in \cite{luhrmann_stability_2022} seems to be difficult for the current problem, therefore we opt to detect them globally as done in \cite{dafermos_linear_2019}. We will introduce $\Theta^{mom},\Theta^{com},\Theta^\Lambda$ functionals that precisely determine the momentum, center of mass and scale content of the solution on a hypersurface $\Sigma_\tau$.
	
	The main point of this subsection is to introduce the notation and prove the following lemma
		
	\begin{lemma}
		a) Let $\psi^{dat},\psi_0,\psi_1$ be scattering initial data for \eqref{linearised eom2} and $\psi$ the corresponding solution.
		Then the following conservation laws and inequality hold for $\psi$
		\begin{subequations}
			\begin{gather}
				\Theta^{mom}(\tau_2)=\Theta^{mom}(\tau_1)+\mathcal{B}^{mom}\\ \Theta^{com}(\tau_2)=\Theta^{com}(\tau_1)+(\tau_2-\tau_1)\Theta^{mom}(\tau_1)+\mathcal{B}^{com}\\ \Theta^{\Lambda}(\tau_2)=\Theta^{\Lambda}(\tau_1)+\mathfrak{R}^{\Lambda}_{\tau_1,\tau_2}+\mathcal{B}^{\Lambda}\\
				\mathcal{E}_{\tau_2}^V=\mathcal{E}^V_{\tau_1}+\mathcal{E}_{\scri_{\tau_1,\tau_2}}+\mathcal{B}_{\mathcal{E}},\quad \tilde{C}_x[J]\leq\mathcal{E}_{\tau_1}+\mathcal{E}_{\scri_{\tau_1,\tau_2}}+\mathcal{B}_{\mathcal{E}},\quad J=T\cdot\T^V[\psi].
			\end{gather}
		\end{subequations}
		
		where
		\begin{equation}\label{inhomogeneous errors}
			\begin{gathered}
				\abs{\mathcal{B}^{mom}}\leq c\int_{\rdom}\abs{F}\frac{1}{\jpns{r}^2};\quad
				\abs{\mathcal{B}^{com}}\leq c\int_{\rdom}\abs{F}\frac{\tau_2-t_\star}{\jpns{r}^2};\\
				\abs{\mathcal{B}^{\Lambda}}\leq c\int_{\rdom}\abs{F}\frac{1}{\jpns{r}};\quad 
				\abs{\mathcal{B}_{\mathcal{E}}}\leq \int_{\rdom}\abs{F} (\partial_t\psi).
			\end{gathered}
		\end{equation}		
	\end{lemma}
	\begin{proof}
		This is a summary of \cref{momentum content} \cref{com content} \cref{scaling content} and divergence theorem applied to $T\cdot\T^V[\psi]$.
	\end{proof}

	\subsubsection{Momentum}
	The full non-linear equation \eqref{main equation} has a nonlinear conserved quantity associated to translation invariance. Indeed, the corresponding divergence free Noether current is 
	\begin{equation}
		\begin{gathered}
			(\partial_i)\cdot \T^{\phi^4/3}[\phi]
		\end{gathered}
	\end{equation}
	where $\T$ is given in \eqref{energy mom tensor}. Linearising around $W$, we may write $\phi=W+\psi$ for small $\psi$. Evaluating the linear contribution to the above quantity yields
	\begin{equation}\label{linearised energy mom tensor}
		\begin{gathered}
			\Tlin_{\mu\nu}[\psi]:=2\partial_{(\mu} W\partial_{\nu)}\psi-\eta_{\mu\nu}(\partial W\cdot\partial\psi-\Vlin\psi),\qquad
			\Vlin=W^5.
		\end{gathered}
	\end{equation}
	We introduced $\Vlin$ to highlight that it is different from $V=5\Vlin/W$. A simple computation shows that $\Tlin$ is divergence free in the linear theory, i.e. if $\psi$ solves \eqref{linearised eom} then $\partial^\mu \Tlin_{\mu\nu}[\psi]=0$. The corresponding current on an incoming null hypersurface $\underline{\mathcal{C}}$ is
	\begin{equation}
		\begin{gathered}
			(\partial_t-\partial_r)^\mu(\partial_i)^\nu \Tlin_{\mu\nu}[\psi]=\partial_t\psi\partial_iW-\partial_r W\partial_i\psi-\partial_r\psi\partial_iW+\hat{x}_i(\partial_rW\partial_r\psi-\Vlin \psi)
			\\=\partial_t\psi\partial_iW-\partial_rW\partial_i\psi-\Vlin\psi.
		\end{gathered}
	\end{equation}
	Evaluating this flux on an outgoing null cone is
	\begin{equation}
		\begin{gathered}
			\int_{\underline{\mathcal{C}}^{v_1,v_2}_{u}}\d\omega \d v r^2\Big(\partial_r W(\hat{x}_i\partial_t-\partial_i)\psi+\Vlin\psi\Big).
		\end{gathered}
	\end{equation}
	Using $\partial_rW\underset{r\to\infty}{\sim} r^{-2}$ and the boundedness of the incoming energy from $\scri$ for a scattering solution $\psi$ we get that the corresponding flux through $\scri$ vanishes.
	\begin{lemma}
		For $\psi$ a scattering solution as described in \cref{existence of scattering solution} 
		\begin{equation}
			\begin{gathered}
				\mathfrak{R}^{mom}:=\int_{\scri_{v_1,v_2}}(\partial_t-\partial_r)^\mu(\partial_i)^\nu \Tlin_{\mu\nu}[\psi]=0.
			\end{gathered}
		\end{equation} 	
	\end{lemma}
	This means that linear momentum is not radiated through $\scri$. Using the divergence theorem on the spacetime region $\rdom$ implies
	\begin{equation}\label{linear momentum conservation}
		\begin{gathered}
			\Theta_i^{mom}(\tau_1):=\Sigma_{\tau_1}[J_i]=\Sigma_{\tau_2}[J_i],\qquad J_i=(\partial_i)\cdot \Tlin[\psi]
		\end{gathered}
	\end{equation}
	for a solution $\psi$ of \eqref{linearised eom}.
	
	For the convenience of the reader, we derive an explicit form of $\Theta^{mom}_i$ in the coordinates we use.
	We can calculate $\Theta^{mom}$ similar to the energy \cref{energy calculation},
	\begin{equation}\label{eq:app:mom}
		\begin{multlined}
			\Theta^{mom}(\tau)[\psi]=\int_{\R^3}\Tlin[\psi](\partial_i,dt_\star^\#)=-\int_{\R^3}h'\Big((X^\circ W) (X^\circ_i\psi)+\hat{x}_i \Vlin \psi\Big)+X^\circ_iWT \psi\\
			=-\int_{\R^3}h'\Big((XW)(X_i-\hat{x}_ih'T)\psi+\hat{x}_i\Vlin \psi\Big)+X_iWT\psi\\
			=-\int_{\R^3}h'\Big((XW)(X_i\psi)+\hat{x}_i\Vlin \psi\Big)+X_iWT\psi(1-h'^2).
		\end{multlined}
	\end{equation}
	Note that $\Theta^{mom}(\tau):(\psi|_{\Sigma_\tau},T\psi|_{\Sigma_\tau})\to\R$ is a bounded function on $\dot{H}^1\times\mathfrak{X}_1$.
	The same quantity evaluated on $\tilde{\Sigma}=\{t=0\}$ would give
	\begin{equation}
		\begin{gathered}
			\tilde{\Sigma}[\partial_i\cdot\Tlin[\psi]]=-\int_{\tilde{\Sigma}}\Tlin[\psi][\partial_i,T]=-\int_{\tilde{\Sigma}}T\psi\partial_i W.
		\end{gathered}
	\end{equation}
	
	Whenever it is convenient, we suppress the subscript on $\Theta^{mom}$ and treat it as a vector. We call this the momentum content of the solution. Furthermore, note the following explicit computations for solutions of \eqref{linearised eom}:
	\begin{equation}
		\begin{gathered}
			\Theta^{mom}_i[\partial_jW]=\Theta^{mom}_i[\Lambda W]=\Theta^{mom}_i[t\partial_kW]=0,\quad k\neq i\\
			\Theta^{mom}_i[t\partial_iW]=\Sigma_{\tau}[(\partial_i)\cdot \Tlin[t\partial_iW]]=\tilde{\Sigma}[(\partial_i)\cdot \Tlin[t\partial_iW]]=\int (\partial_iW)^2=:C^{mom},\quad \tilde{\Sigma}=\{t=0\}\\
		\end{gathered}
	\end{equation}
	which means that $\psi-\frac{\Theta_i^{mom}[\psi]}{C^{mom}}t\partial_iW$ has zero momentum.
	
	Introducing a forcing term makes the energy momentum tensor non-divergence free 
	\begin{equation}
		\begin{gathered}
			\nabla^\mu \Tlin_{\mu\nu}[\psi]=F\partial_{\nu} W
		\end{gathered}
	\end{equation}
	for $\psi$ a solution to \eqref{linearised eom2}. In particular \eqref{linear momentum conservation} changes to
	\begin{equation}\label{linear momentum conservation with force}
		\begin{gathered}
			\Sigma_{\tau_1}[J]=\Sigma_{\tau_2}[J]+\mathcal{B}^{mom}, \qquad
			\mathcal{B}=\int_{\mathcal{R}_{\tau_1,\tau_2}} F\partial_i W.
		\end{gathered}
	\end{equation}

	\begin{lemma}[Momentum content]\label{momentum content}
		Let $\psi$ be a scattering solution to \eqref{linearised eom2} in $\mathcal{R}_{\tau_1,\tau_2}$. Then \eqref{linear momentum conservation with force} holds.
	\end{lemma}
	
	\subsubsection{Translation}
	We also have the nonlinearly conserved quantity associated to the Lorentz boosts $\Gamma^i=(t\partial_i+x_i\partial_t)$. The current on incoming null cone $\underline{\mathcal{C}}$ is
	\begin{equation}
		\begin{gathered}
			(\Gamma^i)^\nu(\partial_t-\partial_r)^\mu \Tlin_{\mu\nu}[\psi]=x_i(\partial_rW\partial_r\psi-\Vlin\psi)+t\partial_iW\partial_t\psi-x_i\partial_rW\partial_t\psi\\
			-t\hat{x}_j(\partial_jW\partial_i\psi+\cancel{\partial_iW\partial_j\psi})+\hat{x}_it(\cancel{\partial_rW\partial_r\psi}-\Vlin\psi)\\
			=\partial_rW\Big(x_i\partial_r+\hat{x}_it\partial_t-x_i\partial_t-t\partial_i\Big)\psi-\Vlin\psi(x_i+\hat{x}_it)\\
			=-\partial_rW\Big(\hat{x}_i\cdot\slashed{\nabla}\Big)\psi+2\hat{x}_iu\partial_rW\partial_u\psi-\Vlin\psi(x_i+\hat{x}_it)
		\end{gathered}
	\end{equation}
	where in the last equality we used that $v=\frac{t+r}{2}$ and $\partial_i=\hat{x}_i\partial_r+\frac{1}{r}\hat{x}_i\cdot\slashed{\nabla}$. Similarly as for the momentum, we have no radiation
	\begin{lemma}
		For a scattering solution $\psi$ we have
		\begin{equation}\label{com flux through scri}
			\begin{gathered}
				\mathfrak{R}^{com}:=\int_{\scri_{v_1,v_2}}(\Gamma^i)^\nu(\partial_t-\partial_r)^\mu \Tlin_{\mu\nu}[\psi]=0.
			\end{gathered}
		\end{equation}		
	\end{lemma}

	Also, note that we may split the vector field 
	\begin{equation}
		\begin{gathered}
			\Gamma^i=t_\star\partial_i+x_i\partial_t+R_2h\partial_i=:t_\star \partial_i+\tilde{\Gamma}^i.
		\end{gathered}
	\end{equation}
	Using the divergence theorem, we get
	\begin{equation}\label{linear com conservation}
		\begin{gathered}
			\Theta^{com}_i(\tau_1)+\tau_1\Theta^{mom}_i(\tau_1):=\Sigma_{\tau_1}[\tilde{J}_i+t_\star J_i]=\Sigma_{\tau_2}[\tilde{J}_i+t_\star J_i]\\
			\tilde{J}_i[\psi]=\tilde{\Gamma}^i\cdot \Tlin[\psi],\, J_i=(\partial_i)\cdot \Tlin[\psi]
		\end{gathered}
	\end{equation}
	Note, that using \eqref{linear momentum conservation} the $\Sigma_{\tau_1}[J]=0$ case gives boundedness (in fact, no $\tau$ dependence) for the $\tilde{J}$ current. We call this the center of mass content of the solution.
	
	Because we are using a weighted vector field, namely a boost, to define the center of mass, it is not a priori clear that $\Theta^{com}_i$ is well defined for scattering solutions.
	We calculate
	\begin{equation}
		\begin{multlined}
			-\Tlin[x_iT+R_2hX_i^\circ,\d t_\star^\#]=\Tlin[x_iT+R_2hX_i^\circ,T+h'X^\circ]\\
			=\Tlin[x_iT,T]+\Tlin[R_2hX_i^\circ,T]+\Tlin[x_iT,h'X^\circ]+\Tlin[hX_i^\circ,h'X^\circ]
			\\=x_i(\partial W\cdot \partial\psi-\Vlin\psi)+R_2hX^\circ_i WT\psi+x_iT\psi h'X^\circ W+h'R_2h(X_i^\circ\psi X^\circ W+\hat{x}_i\Vlin\psi)\\			
			=(rX_iWX\psi+R_2hh'X_i\psi XW)+X_iWT\psi\underbrace{R_2h(1-h'^2)}_{=0, r>2R_2}+\Vlin\psi (h'R_2h\hat{x}_i-x_i).
		\end{multlined}
	\end{equation}
	Integrating on a sphere removes all but the $l=1$ mode of $\psi$, so without loss of generality, write $\psi=\frac{x_k \bar{\psi}(r,t_\star)}{r^2}+\text{other modes}$ to get
	\begin{equation}
		\begin{gathered}
			\int_{S^2}	(rX_iWX\psi+R_2hh'X_i\psi XW)=\int_{S^2}r\hat{x}_iW'\frac{x_k\bar{\psi}'-\hat{x}_k\bar{\psi}}{r^2}+R_2hh'W'\frac{\delta_{ik}\bar{\psi}+\hat{x}_ix_k\bar{\psi}'-2\hat{x}_i\hat{x}_k\bar{\psi}}{r^2}\\
			=\delta_{ik}\frac{4\pi}{3}W'\Big(\bar{\psi}'\frac{r+R_2hh'}{r}+\bar{\psi}\frac{R_2hh'-r}{r^2}\Big).
		\end{gathered}
	\end{equation}
	In conclusion we get
	\begin{equation}
		\begin{multlined}
			\Theta_i^{com}[\psi](\tau)=\int_{\Sigma_\tau}
			\hat{x}_i\Big(-\frac{r+R_2hh'}{r}W'X(r\psi)+R_2hW'T\psi(1-h'^2)\\+\psi(\Vlin r-h'R_2h\Vlin+W'\frac{R_2hh'-r}{r})\Big).
		\end{multlined}
	\end{equation}
	We compute the first term as follows
	\begin{equation}
		\begin{gathered}
			\int_{\Sigma_\tau}\hat{x}_i\frac{r+R_2hh'}{r}W' X(r\psi)=\lim_{R\to\infty}\Big( \int_{S^2,\Sigma_\tau\cap\{r=R\}}\hat{x}_i2W'r^2(r\psi)\\
			-\int_{\Sigma_\tau\cap\{r\leq R\}}r\psi\frac{r+R_2hh'}{r}\Delta W+r\psi W'X\frac{R_2hh'}{r}\Big)\\
			=-\int_{\Sigma\tau\cap\scri}2\sqrt{3}\hat{x}_i(r\psi)+\int_{\Sigma_\tau}\psi\Big((r+R_2hh')\Vlin-W'X\frac{R_2hh'}{r})\Big).
		\end{gathered}
	\end{equation}
	We conclude that
	\begin{equation}\label{eq:app:com}
		\begin{gathered}
			\Theta^{com}_i[\psi]=\int_{\Sigma_\tau\cap\scri}2\sqrt{3}x_i\psi+\int_{\Sigma_\tau}\hat{x}_i\bigg(R_2hW' T\psi (1-h'^2)
			+\psi\underbrace{\Big(W'\frac{R_2hh'-r}{r}+W'X\frac{R_2hh'}{r}-2R_2hh'\Vlin\Big)}_{\mathcal{O}(r^{-3})}\bigg)
		\end{gathered}
	\end{equation}
	In particular the second part, which we call $\tilde{\Theta}^{com}$ is in $(\mathfrak{X}_1\times\dot{H}^1)^\star$.
	Moreover, its size is bounded by a constant times $R_2$.
	Let us also note the center of mass content on $\tilde{\Sigma}$ for vanishing momentum 
	\begin{equation}
		\begin{gathered}
			\tilde{\Sigma}[(x_iT+tX_i)\cdot\Tlin[\psi]]=	\tilde{\Sigma}[(x_iT)\cdot\Tlin[\psi]]=-\int_{\tilde{\Sigma}}x_i(\partial W\cdot \partial\psi-\Vlin\psi)=\int_{\tilde{\Sigma}}\psi\partial_iW\\
			\text{given } \Theta^{mom}[\psi]=0.
		\end{gathered}
	\end{equation}

	We now calculate the $\tilde{J}$ current of the solutions in $\ker H$. As for the momentum, we may perform this computation on $\tilde{\Sigma}=\{t=0\}$. A simple calculation then  gives
	\begin{equation}\label{eq:com_detection}
		\begin{gathered}
			\Theta^{com}_i[\Lambda W]=\Theta^{com}_i[\partial_jW]=0,\quad j\neq i\\
			C^{com}:=\tilde{\Sigma}[\tilde{J}_i[\partial_i W]]=\int x_i\Big(\partial_rW\partial_r\partial_iW-\Vlin\partial_iW\Big)=-\int\hat{x}_i\hat{x}_iW\partial_r^2W-\int x_iW\cancel{(\Delta+W^4)\partial_iW}\\
			=2\pi\int_0^\infty dr r^2W(W^5+\frac{2}{r}\partial_rW)=2\pi\int_0^\infty dr W^2-r^2W^6\neq0.
		\end{gathered}
	\end{equation} 
	Therefore, we know that for $\psi$ with $\Theta^{mom}[\psi]=0$ we have that $\psi-\frac{\Theta_i^{com}[\psi]}{C^{com}}\partial_iW$ has no center of mass content. Furthermore, using that $\partial_i W\sim r^{-2}$, we get that $\Theta^{com}[\partial_iW]=\tilde{\Theta}^{com}[\partial_i W]$.
	
	As for the momentum content, the presence of inhomogeneity ruins the conservation and instead for $\psi$ a solution of \eqref{linearised eom2} we have 
	\begin{equation}\label{linear com conservation with force}
		\begin{gathered}
			\Theta^{com}_i(\tau_2)=\Theta^{com}_i(\tau_1)+(\tau_1-\tau_2)\Theta^{mom}_i(\tau_1)+\mathcal{B}^{com}\\
			\mathcal{B}^{com}=\int_{\rdom} F(t-\tau_2)\partial_iW
		\end{gathered}
	\end{equation}
	
	\begin{lemma}[Linear com]\label{com content}
		Let $\psi$ be a scattering solution in $\rdom$ to \eqref{linearised eom2}. Then \eqref{linear com conservation with force} holds.
	\end{lemma}
	
	\subsubsection{Scaling}
	For the scaling part of the kernel of \eqref{linearised eom} we have no killing field at our disposal. Nonetheless, we know that both $\Lambda W$ and $t\Lambda W$ solve \eqref{linearised eom}. A quick computation shows that $\Lambda W$ has no radiation field on $\scri$ and has 0 $\mathcal{E}^V$ energy flux through $\{t=\text{const}\}$ and $\Sigma_\tau$. Similarly, one can calculate that the $\mathcal{E}^V$ energy flux of $t\Lambda W$ is infinite on $\{t=\text{conts}\}$ slices, but finite through $\Sigma_\tau$ because it has $\mathcal{O}(1)$ behaviour towards $\scri$. Indeed, by the divergence theorem, we have that for any scattering solution $\psi$ solving \eqref{linearised eom}
	
	\begin{equation}\label{divergence with tLW}
		\begin{gathered}
			\scri_{v_1,v_2}[J_\psi]+\Sigma_{\tau_1}[J_\psi]=\Sigma_{\tau_2}[J_\psi], \qquad J_\psi=(\partial_t)\cdot \T^V[\psi,t\Lambda W].
		\end{gathered}
	\end{equation}
	More explicitly, we may calculate the radiation field of $t\Lambda W$
	\begin{equation}
		\begin{gathered}
			\lim_{r\to \infty} r\partial_u t\Lambda W|_{\Sigma_\tau}=\lim_{r\to\infty} r(1-tX^{\circ})\Lambda W=\lim_{r\to\infty}r(1-tX^{\circ})\frac{-1}{2}W=
			\lim_{r\to\infty}-rW=-\sqrt{3}.
		\end{gathered}
	\end{equation}
	Using this, we evaluate the radiative term
	\begin{equation}
		\begin{gathered}
			\mathfrak{R}_{\tau_1,\tau_2}^\Lambda[\psi]:=\scri_{v_1,v_2}[J_\psi]=\lim_{r\to \infty}\int_{\tau_1}^{\tau_2}\int_{S^2}-\sqrt{3}r\partial_\tau\psi \d\tau\d\omega=-\sqrt{3}\lim_{r\to \infty}\int_{S^2}\d\omega(r\psi|_{\Sigma_{\tau_2}}-r\psi|_{\Sigma_{\tau_1}}) 
		\end{gathered}
	\end{equation}
	which is simply the difference of the $l=0$ mode of the radiation fields on the two corresponding spheres at infinity. This transforms \eqref{divergence with tLW} to
	\begin{equation}\label{scaling conservation law}
		\begin{gathered}
			\Theta^{\Lambda}[\psi](\tau_2):=\Sigma_{\tau_1}[J_\psi]+\sqrt{3}\lim_{r\to\infty}\int_{S^2}\d\omega r\psi|_{\Sigma_{\tau_1}}=
			\Sigma_{\tau_2}[J_\psi]+\sqrt{3}\lim_{r\to\infty}\int_{S^2}\d\omega r\psi|_{\Sigma_{\tau_2}}.
		\end{gathered}
	\end{equation}
	We also define $\tilde{\Theta}^{\Lambda}[\psi]=\Sigma_\tau[J_\psi]$ which as before satisfies $\tilde{\Theta}^\Lambda\in(\dot{H}^1,\mathfrak{X}_1)$ with norm growing linearly with $R_2$.
	We compute $\Theta^{\Lambda}$ using the bilinear energy calculation from \cref{energy calculation}
	\begin{equation}\label{eq:app:Lambda}
		\begin{gathered}
			\tilde{\Theta}^\Lambda[\psi](\tau)=\int_{\Sigma_\tau}(1-h'^2)T\psi\Lambda W+X\psi X(t\Lambda W)-V\psi t\Lambda W\\
			=\int_{\Sigma_\tau}(1-h'^2)T\psi\Lambda W-X\psi \underbrace{X(h\Lambda W)}_{\sim r^{-3}}+V\psi h\Lambda W+\tau \cancel{(X\psi X\Lambda W-V\psi\Lambda W)}
		\end{gathered}
	\end{equation}
	We can also evaluate it on $\tilde{\Sigma}=\{t=0\}$ to get
	\begin{equation}
		\begin{gathered}
			\tilde{\Sigma}[\partial_t\cdot\T[\psi,t\Lambda W]]=-\int_{\tilde{\Sigma}}\Lambda W T\psi.
		\end{gathered}
	\end{equation}

	To capture the stationary part of the linear solution ($\psi$) that comes from the scaling ($\Lambda W$), we must evaluate this conserved quantity on $\Lambda W$

	\begin{equation}\label{eq:scale_detection}
	\begin{gathered}
			\Sigma_{\tau}[J_{\Lambda w}]+4\pi\sqrt{3}\lim_{r\to \infty}r\Lambda W=-6\pi+\int_{\Sigma_\tau} (1-h'^2)(T\Lambda W) T(t\Lambda W)+(X\Lambda W)(X t\Lambda W)-V\Lambda W t\Lambda W\\
			=-6\pi+\int_{\Sigma_\tau}\tau\cancel{\Big((X\Lambda W)^2-V(\Lambda W)^2\Big)}+(X\Lambda W)(Xh\Lambda W)-hV (\Lambda W)^2\\
			=-6\pi+\lim_{R\to\infty}\int_{\Sigma_\tau\cap\{r\leq R\}}(X\Lambda W)(Xh\Lambda W)-hV (\Lambda W)^2\\
			=-6\pi+\lim_{R\to\infty}[4\pi r^2(X\Lambda W)(h\Lambda W)]|_R-\int_{\Sigma_\tau\cap\{r\leq R\}}+h (\Lambda W)\cancel{\big(\Delta+V\big)\Lambda W}=-9\pi:=C^\Lambda\neq0.
	\end{gathered}
\end{equation}
	Importantly, this quantity does not vanish, so we may use it to detect the $\Lambda W$ content of $\psi$. Hence, for $\psi$ a solution of \eqref{linearised eom} with $\mathfrak{R}^\Lambda_{\tau_1,\tau_2}[\psi]=0$ for all $\tau_2>\tau_1$ we have $\psi-\frac{\Theta^\Lambda[\psi]}{C^\Lambda}\Lambda W$  is scale content free. The above computation also yields that $\tilde{\Theta}^\Lambda[\Lambda W]\neq0$.
	
	Following the theme of the previous sections, we get the inhomogeneous case for $\psi$ a solution of \eqref{linearised eom2}
	\begin{equation}\label{linear scale conservation with force}
		\begin{gathered}
			\Sigma_{\tau_2}[J_\psi]=\Theta^\Lambda[\psi](\tau_1)+\mathfrak{R}^\Lambda_{\tau_2,\tau_1}[\psi]+\mathfrak{B}^\Lambda\\
			\mathfrak{B}^\Lambda=\int_{\rdom} F\Lambda W
		\end{gathered}
	\end{equation}
	
	\begin{lemma}[Scaling detection]\label{scaling content}
		Let $\psi$ be a scattering solution in $\rdom$ to \eqref{linearised eom2}. Then \eqref{linear scale conservation with force} holds.
	\end{lemma}
	
	\subsubsection{Unstable modes}
	Throughout the paper, we will use localised projection operator for the unstable modes. In this section, we give a brief alternative description on how to detect these globally.
	
	We recall from Proposition 3.9 in \cite{duyckaerts_solutions_2016}, that $Y=e^{-\lamed r}(\jpns{r}^{-1}+\mathcal{O}(r^{-1-1/2}))$. In particular, we get that both $Y_\pm$ are valid scattering solutions. We start with the detection of the $Y_+=Ye^{\lamed t}$ content of $\psi$. We note, that as $Y_-$ decays exponentially towards $\scri$, the current $J_+=T\cdot\T^V[\psi,Y_-]$ has trivial flux through $\scri$. Therefore $\Sigma_\tau[J_+]$ is conserved. Evaluating the flux of $T\cdot\T^V[Y_+,Y_-]$ through $\{t=0\}$ we find that it is non-vanishing, therefore $J_+$ can be used to detect the $Y_+$ content. 
	
	Proceeding similarly for $Y_-$, we observe the following problem. $Y_+$ has a nontrivial radiation field, indeed $(r\partial_uY)|_{\scri}=e^u$. Therefore the current $J_-[\psi,Y_+]$ has a non zero flux through $\scri$. Using coordinates $t_\star,r$ we can write the above as
	\begin{equation}
		\begin{gathered}
			\partial_\tau\Big(\Sigma_{\tau}[T\cdot\T^V[e^{\lamed(t_\star+R_2h)}Y(r),\psi]]\Big)=ce^{\lamed \tau}\int_{\Sigma_\tau\cap\scri}\partial_ur\psi\\
			\implies \partial_\tau\Big(e^{\tau\lamed}\underbrace{\Sigma_{\tau}[T\cdot \T^V[e^{\lamed R_2h}Y(r),\psi]]}_{=:\tilde{\alpha}_-}\Big)=ce^{\lamed \tau}\int_{\Sigma_\tau\cap\scri}\partial_ur\psi\\
			\implies \tilde{\alpha}_-(\tau)+\tilde{\alpha}_-'(\tau)=c\int_{\Sigma_\tau\cap\scri}\partial_ur\psi.
		\end{gathered}
	\end{equation}
	Therefore, even though $\tilde{\alpha}_-$, i.e. the $Y_-$ content, is not conserved, it can be bounded by the radiation field of $\psi$. Given that all the above constructed projection operators indeed capture the obstruction to boundedness and decay it is natural to formulate the following conjecture.
	\begin{conjecture}
		Let $\psi^{dat}$ be data as in \cref{def:scattering definition} with $d>0$ and $N>5$ for \cref{linearised eom} on $\Sigma_{1}$. Provided that $\Theta^{mom}[\psi]=\Theta^{com}[\psi]=\Theta^{\Lambda}[\psi]=\Sigma_{1}[J_+]=0$, the forward solution to \cref{linearised eom} with $\psi^{dat}$ initial data decay polynomially as $t\to\infty$.
	\end{conjecture}
	
	For compactly supported initial data $\psi$ on a slice $\{t=0\}$ the requirement of the conjecture translates to 
	\begin{equation}
		\begin{gathered}
			\int \Lambda WT\psi=\int \psi \partial_i W=\int \partial_iW T\psi=\int Y(\lamed+T)\psi=0.
		\end{gathered}
	\end{equation}
	\subsection{Uncommuted estimates}\label{uncommuted estimate section}
	As discussed in the introduction, we deal with 0 modes and unstable modes separately. In fact, we will only care about the unstable mode at the very end of the scattering construction. Until then, we will simply work with the following bootstrap assumption whenever we need it:
	\begin{equation}\label{unstable bootstrap assumption}
		\begin{gathered}
			\abs{\alpha_-}|_{\tau}\leq e^{-\lamed R_1/2}\epsilon
		\end{gathered}
	\end{equation}

	\begin{lemma}
		Let $\psi$ satisfy \eqref{unstable bootstrap assumption} in $\mathcal{R}_{\tau_1,\tau_2}$ with $\epsilon=1$ and
		\begin{equation}
			\begin{gathered}
				e^{R_1\lamed/2}\abs{\alpha_-(\tau_1)},\sup_\tau \mathcal{E}^V(\tau),\mathcal{E}_{\scri_{\tau_1,\tau_2}},\abs{\tilde{\Theta}^\bullet(\tau)}\leq 1,\quad \bullet\in\{com,\Lambda\}\\
				\int_{\rdom}\abs{F}T\psi,e^{\lamed R_1}\sum_{\tau\in\{\tau_1,\tau_2\}}\int_{\Sigma_\tau}\abs{F}e^{-\lamed r}\leq1 .
			\end{gathered}
		\end{equation}
		There exists $c>0$ such that 
		\begin{subequations}\label{hom bounds}
			\begin{gather}
				\masterNum{0}[\psi]
				\leq c\label{energy control}\\
				\sup_\tau\abs{\alpha_\pm(\tau)}e^{\lamed R_1/2}\leq c \label{unstable mode hom}.
			\end{gather}
		\end{subequations}
	\end{lemma}
	\begin{proof}
%
		
		\textit{Step 1: stable mode.}
		Using the equation for $\alpha_+$, \eqref{ode unstable mode} we get
		\begin{equation}\label{stable mode uncommuted bound}
			\begin{gathered}
				\begin{multlined}
					\int_{\Sigma_{\tau}}\mathfrak{Err}_\pm[\psi]\leq c\int_{\mathcal{D}}e^{-\lamed r}\Big(\frac{\abs{\psi}+\abs{\partial\psi}}{R_1^2}+\abs{F}\Big)\leq c e^{-\lamed R_1}+ c e^{-\lamed R_1}R_1\Big(\int_{\mathcal{D}} \frac{\psi^2}{r^2}+\abs{\partial\psi}^2\Big)^{1/2}\\ \leq  ce^{-\lamed R_1}R_1 \mathcal{E}(\tau)+c e^{-\lamed R_1}
				\end{multlined}\\
				\alpha_+(\tau_1)=e^{-\lamed(\tau_2-\tau_1)}\alpha_+(\tau_2)+\int_{\tau\in(\tau_1,\tau_2)} e^{-\lamed(\tau_2-\tau)}\mathfrak{Err}_-(\tau)\leq  e^{-R_1\lamed/2}+ ce^{-\lamed R_1/2}R_1\sup_\tau\mathcal{E}(\tau)
			\end{gathered}
		\end{equation}
		for $\mathcal{D}=(B_{2R_1}\backslash B_{R_1})\cap _{\Sigma_{\tau}}$.
		 
		\textit{Step 2: coercivity.}
		Let us note that $\mathfrak{E}^V$ is coercive in the $T\psi$ part. In particular, we can write
		\begin{equation}
			\begin{gathered}
				(\tilde{\Theta}_{T\psi}^{\bullet})^2\lesssim \mathcal{E}^V_{T\psi}
			\end{gathered}
		\end{equation}
		for $\bullet\in\{com,\Lambda\}$, where we use the subscript $T\psi$ to refer to the part of a quantity depending on $T\psi$. We will similarly write $\mathcal{E}^V_{\parallel}=\mathcal{E}^V-\mathcal{E}^V_{T\psi}$. The results from \cref{eq:com_detection,eq:scale_detection} yields that we can use $\tilde{\Theta}_{\parallel}^\bullet\in\dot{H}^1$ for detecting the kernel elements $\{\Lambda W,\partial_i W\}$. Using coercivity from \cref{coercivity lemma} we conclude
		\begin{equation}
			\begin{gathered}
				\mathcal{E}(\tau)\leq \mu \mathcal{E}^V(\tau)-\frac{1}{\mu}\Big(\abs{\alpha_{\pm}}^2+\sum_\bullet\abs{\tilde{q}^\bullet(\tau)}^2\Big)\leq \mu-\frac{1}{\mu}(1+c e^{-\lamed R_1/2}R_1\mathcal{E}(\tau))\\
				\implies \mathcal{E}(\tau)\leq c				
			\end{gathered}
		\end{equation}	
		where we used that $R_1$ is sufficiently large, and $\mu$ depends on $R_2$ linearly.
		To control the integral on outgoing cones, we just use the divergence theorem.
	\end{proof}

	\subsection{Higher order estimates}
	In this section, we extend the previous results to higher order estimates. We commute with vector field $\{\tau T,S,\Omega\}$ to get that the solution is conormal. Since we lack good commutation properties for $\tau T$, we will simply include extra $\tau$ weight in the definition of the current. Since we need stronger decay in $\tau$ for $T$ commuted quantities, we need to make a new unstable mode bootstrap assumption\footnote{Note, that if we did not want to get extra $\tau$ factor for $T$ commutator, i.e. we did not prove conormality, then we could recover the corresponding higher order bootstrap from the lower order one and elliptic estimates.}:
	\begin{equation}\label{higher unstable bootstrap}
		\begin{gathered}
			\abs{\lamed\alpha_-[{T^N\psi}]}\leq e^{-\lamed R_1/2}{\tau}^{-N}\epsilon
		\end{gathered}
	\end{equation}

	In order to get boundedness with $\tau T$ commutators, we will need to perform a 0-mode analysis for each of these quantities.
	
	\begin{lemma}[Higher order zero modes]\label{higher order zero modes}
		
		a) Let $\psi$ be a solution to \cref{linearised eom2} with $F=0$. Than $T\psi$ also solves \cref{linearised eom2} with 
		\begin{equation}
			\begin{gathered}
				\Theta^{mom}[T\psi]=\Theta^\Lambda[T\psi]=0,\qquad 
				\Theta^{com}[T\psi]=\Theta^{mom}[\psi].
			\end{gathered}
		\end{equation}
		
		b) For $F$ nonzero, smooth we have the following corrections
		\begin{equation}
			\begin{gathered} 
				\Theta^{mom}[T\psi]=\int_{\Sigma_\tau}FX_i W,\quad \Theta^\Lambda[T\psi]=\int_{\Sigma_\tau}F\Lambda W,\\
				\Theta^{com}[T\psi]=\Theta^{mom}[\psi].
			\end{gathered}
		\end{equation}
	\end{lemma}
	\begin{proof}
		This is straightforward computation. Let's start with $F=0$ case for the for the momentum:
		\begin{equation}
			\begin{gathered}
				\Theta^{mom}[T\psi]=\Sigma_{\tau}[J]=\tilde{\Sigma}[J]=\int_{\tilde{\Sigma}} \partial_iWTT\psi=
				\int_{\tilde{\Sigma}} \partial_iW(\Delta+V)\psi=\int_{t=0}\psi(\Delta+V)\partial_iW=0\\
				\text{ where }J=X_i^\circ\cdot\Tlin[T\psi], \tilde{\Sigma}=\{t=0\}
			\end{gathered}
		\end{equation}
		In order to consider the case with $F\neq0$, let's introduce $\tilde{\psi}_\epsilon$ for $\epsilon>0$ that has the same data as $\psi$ on $\Sigma_\tau$, but the inhomogeneity is $\tilde{F}_\epsilon=F\bar{\chi}(2-\frac{\tau-t_\star}{\epsilon})$. Since $F=\tilde{F}_\epsilon$ in a neighbourhood of $\Sigma_\tau$, $\Theta^{mom}[\psi](\tau)=\Theta^{mom}[\tilde{\psi}_\epsilon](\tau)$, but we can calculate
		\begin{equation*}
			\begin{gathered}
				\Theta^{mom}[T\tilde{\psi}_\epsilon](\tau)=\Theta^{mom}[T\tilde{\psi}_\epsilon](\tau-4\epsilon)+\mathcal{B}=+\mathcal{B}\\
				\begin{multlined}
					\mathcal{B}=\int_{t_\star>\tau} \partial_i WT\tilde{F}_\epsilon=\int_{t_\star\in(\tau,\tau+4\epsilon)} \partial_i W\Big(\bar{\chi}(2-\frac{\tau-t_\star}{\epsilon})TF+F\bar{\chi}'(2-\frac{\tau-t_\star}{\epsilon})\frac{1}{\epsilon}\Big)\\
					\underset{\epsilon\to0}{\to}+\int_{\Sigma_\tau}\partial_i W F
				\end{multlined}
			\end{gathered}
		\end{equation*}
		
		We proceed similarly for the centre of mass. For $F=0$, we have
			\begin{equation}
			\begin{gathered}
				\Theta^{com}[T\psi](\tau)=\tau\Theta^{mom}[T\psi](\tau)+\Theta^{com}[T\psi](\tau)=\Sigma_\tau[\tilde{J}+t_\star J]=\tilde{\Sigma}[\tilde{J}+t_\star J]\\
				=-\int_{\tilde{\Sigma}}x_i(\partial W\cdot\partial T\psi-\Vlin T\psi)=\int x_iT\psi\cancel{(\Delta+W^4)W}+T\psi\partial W\cdot\partial x_i=\tilde{\Sigma}[J]=\Theta^{mom}[\psi](\tau)\\
				\text{ where }J=X_i\cdot\Tlin[T\psi],\,\tilde{J}=(x_iT+hX_i)\cdot\Tlin[T\psi],\, \tilde{\Sigma}=\{t=0\}
			\end{gathered}
		\end{equation}
		For $F\neq 0$, we again introduce $\tilde{\psi}_\epsilon$ and note that we have
		\begin{equation}
			\begin{gathered}
				\Theta^{com}[\tilde{\psi}_\epsilon](\tau)=\Theta^{com}[\tilde{\psi}_\epsilon](\tau+-4\epsilon)+\mathcal{B}=\Theta^{mom}[\tilde{\psi}](\tau-4\epsilon)+\mathcal{B}\\
				\mathcal{B}=\int_{\tau_\star\in(\tau,\tau+4\epsilon)}F(t-\tau_2)\partial_i W	\underset{\epsilon\to0}{\to}0.
			\end{gathered}
		\end{equation}
		Finally, a similar vanishing limit implies that $\Theta^{mom}[\tilde{\psi}](\tau-4\epsilon)	\underset{\epsilon\to0}{\to}\Theta^{mom}[\psi](\tau)$.
		
		For the scaling we again start with the $F=0$ case. Let $\tilde{\psi}$ be a scattering solution with same data on $\Sigma_\tau$ as $\psi$ and $\partial_ur\tilde{\psi}|_{\scri\cap t_\star>\tau+1}=0$.   We use the conservation law from \cref{scaling conservation law} between the hypersurface $\Sigma_\tau$ and $\tilde{\Sigma}_x=\{t+h(r-R_2-x)=\tau\}$ for $x>2$ to get
		\begin{equation}
			\begin{gathered}
				\Theta^\Lambda[\tilde{\psi}]=\Sigma_\tau[J]+\Big(\int_{S^2}rT\tilde{\psi}\Big)\Big|_{\Sigma_\tau\cap\scri}=\tilde{\Sigma}_x[J]+\cancel{\Big(\int_{S^2}rT\tilde{\psi}\Big)\Big|_{\tilde{\Sigma}_x\cap\scri}}\\
				=\lim_{x\to\infty} \tilde{\Sigma}_x=\int_{\{t=\tau\}}\Lambda WT^2\tilde{\psi}=\int_{\{t=\tau\}}\Lambda W(\Delta+V)\tilde{\psi}=0\\
				\text{ where }J=T\cdot\T[t\Lambda W,T\tilde{\psi}].
			\end{gathered}
		\end{equation}
		In the case $F\neq 0$ we argue the same way as for the momentum.	
	\end{proof}
	
	We may use this and \cref{coercivity lemma} to control the coercive energy on $\Sigma_\tau$ for higher derivatives.
	\begin{cor}
		Let $\psi$ be a solution to \eqref{linearised eom2}. Then
		\begin{equation}
			\begin{aligned}
				&\mathcal{E}_\tau[\psi]\leq c\mathcal{E}^V_\tau[\psi]-\frac{1}{c}\Big(\abs{\alpha_{\pm}[\psi]}+\abs{\Theta^{com}[\psi](\tau)}+\abs{\Theta^\Lambda[\psi](\tau)}+\abs{P^0_{S^2}r\psi|_{\scri\cap\Sigma_\tau}}+\norm{P^1_{S^2}r\psi|_{\scri\cap\Sigma_\tau}}_{L^2(S^2)}\Big)^2\\
				&\begin{multlined}
					\mathcal{E}_\tau[T\psi]\leq c\mathcal{E}^V_\tau[T\psi]-\frac{1}{c}\Big(\abs{\alpha_{\pm}[T\psi]}+\abs{\Theta^{mom}[\psi](\tau)}+\int_{\Sigma_\tau}\abs{F}
					\jpns{r}^{-1}\\
					+\abs{P^0_{S^2}\partial_ur\psi|_{\scri\cap\Sigma_\tau}}+\norm{P^1_{S^2}\partial_ur\psi|_{\scri\cap\Sigma_\tau}}_{L^2(S^2)}\Big)^2
				\end{multlined}\\
				&\begin{multlined}
					\mathcal{E}_\tau[T^k\psi]\leq c\mathcal{E}^V_\tau[T^k\psi]-\frac{1}{c}\Big(\abs{\alpha_{\pm}[T^k\psi]}+\int_{\Sigma_\tau}\abs{T^{k-2}F}\jpns{r}^{-2}+\int_{\Sigma_\tau}T^{k-1}F\jpns{r}^{-1}\\
					+\abs{P^0_{S^2}\partial_u^{k-1}r\psi|_{\scri\cap\Sigma_\tau}}+\norm{P^1_{S^2}\partial_u^{k-1}r\psi|_{\scri\cap\Sigma_\tau}}_{L^2(S^2)}\Big)^2,\qquad k\geq2
				\end{multlined}
			\end{aligned}
		\end{equation}
	\end{cor}
	
	Controlling lower order unstable modes is also possible using assumption about the top order one:	
	\begin{lemma}[Recovering lower order bootstraps]
		Let $\psi$ be a solution to \eqref{linearised eom2}. Then for all $k< N$
		\begin{equation}
			\begin{gathered}
				{\tau}^{k}\abs{\alpha_-[T^k\psi]}\lesssim {\tau}^{k-N}\Big({\tau}^N\abs{\alpha_-[T^k\psi]}\Big)+  e^{-\lamed R_1/2}\sum_{j=k}^N {\tau}^{j}\Big(\mathcal{E}_\tau[T^j\psi]+\int_{\Sigma_\tau}\abs{T^jF}e^{-r\lamed}\chi_{2R_1}\Big)
			\end{gathered}
		\end{equation}
	\end{lemma}
	\begin{proof}
		Note, that from the equation for $\alpha_-$,  \eqref{ode unstable mode}, it follows that
		\begin{equation}
			\begin{gathered}
				\alpha_-[T^{k+1}\psi]=\partial_\tau \alpha_-[T^k\psi]=-\lamed\alpha_-[T^k\psi]+\mathfrak{Err}_+[T^k\psi]\\
				\implies \abs{\alpha_-[T^k\psi]}\lesssim \abs{\alpha_-[T^{k+1}\psi]}+\abs{\mathfrak{Err}_+[T^k\psi]}\\
				\lesssim \abs{\alpha_-[T^{k+1}\psi]}+			e^{-\lamed R_1/2}\mathcal{E}_\tau[T^k\psi]+\int_{\Sigma_\tau}\abs{T^kF}e^{-r\lamed}\chi_{2R_1}
			\end{gathered}
		\end{equation}
		where we used the control on $\mathfrak{Err}_+$ derived in \cref{stable mode uncommuted bound}. The claim follows by induction.
	\end{proof}
	
	We also need to control some error terms from $\Theta-\tilde{\Theta}$ that arise on $\scri$ from the data that we prescribe $(r\psi)^{\scri}$.
	\begin{lemma}\label{lemma:theta_scri_control}
		Let $(r\psi)^\scri$ be scattering data in $\mathcal{R}_{\tau_1,\tau_2}$ with decay $q$, size $\epsilon$ and regularity $N$ as in \cref{def:scattering definition}. Furthermore, assume that $(r\psi)^{\scri}(\tau_2)=0$. Then
		\begin{equation}
			\begin{gathered}
				\norm{P^l\partial_u^k(r\psi)^\scri(\tau)}_{L^2(S^2)}\lesssim_{\epsilon} \tau^{-q-k},\qquad l\in\{0,1\},\, k\leq N-1
			\end{gathered}
		\end{equation}
	\end{lemma}
	\begin{proof}
		This is a simple consequence of fundamental theorem of calculus and Holder inequality.
	\end{proof}
	
	The $T$ commuted estimates are already sufficient to control the solution in the region $t_\star\gtrsim r$. This is a standard approach to treat the region close to the localised stationary perturbation. 
	
	\begin{lemma}[Elliptic estimate]\label{lemma:elliptic_estimate}
		For a smooth solution of \eqref{linearised eom2} $\psi$ in $\mathcal{R}_{\tau_1,\tau_2}$ with $10{\tau}>R>R_2$, we have
		\begin{equation}
			\begin{gathered}
				\norm{\Gamma^\beta\psi}_{\dot{H}^1[\Sigma_\tau\cap\{r<R\}]}^2\lesssim_\beta \sum_{k\leq \abs{\beta}}{\tau}^{2k'}\mathcal{E}[T^{k'}\psi]+\sum_{\abs{\alpha}+m\leq k-1}{\tau}^{2m}\norm{\Gamma^\alpha T^m\jpns{r}F}_{L^2[\Sigma_\tau\cap\{r<R\}]}^2
			\end{gathered}
		\end{equation}
		for $\Gamma\in\{S,\Omega\}$. 
	\end{lemma}
	\begin{proof}
		For the proof, we are going to solely work on a single slice $\Sigma_\tau$ and write $\dot{H}^1$ to be the norm with respect to standard measure on $\R^3\simeq\Sigma_\tau$ and differential operators $\Delta,\nabla$ acting on $\R^3$.
		
		We only prove the estimate when $\Gamma=S$ and write $\beta=k$, the $\Omega$ case follows similarly with less error terms to deal with. We apply standard elliptic estimate with Hardy inequality to get	
		\begin{equation}
			\begin{gathered}
				\norm{S\psi}_{\dot{H}^{1}[\{r<R\}]}^2\lesssim \norm{r\Delta \psi}_{L^2[\{r<R\}]}^2+\norm{({\nabla},\frac{1}{R})\psi}^2_{L^2[r<R]}\lesssim\norm{r\Delta \psi}^2_{L^2[\{r<R\}]} +\norm{\psi}^2_{\dot{H}^1[\{r\leq R\}]}
			\end{gathered}
		\end{equation}
		with constant independent of $R$. Using the form of $\Box$ in $t_\star,r$ coordinates \eqref{wave operator in good coordinates} we claim that
		\begin{equation*}
			\begin{gathered}
				\norm{r\Delta \psi}_{L^2[\{r<cR\}]}^2=	\norm{r\Big(V-(1-h'^2)\partial_{t_\star}^2-2h'\partial_{t_\star}\partial_r-h''\partial_{t_\star}-\frac{h'}{r}\partial_{t_\star}\Big) \psi+rF}_{L^2[\{r<R\}]}^2
				\\\lesssim \mathcal{E}[\psi]+ {\tau}\mathcal{E}[T\psi]+\norm{rF}_{L^2[\{r<R\}]}^2,
			\end{gathered}
		\end{equation*}
		which is equivalent to the $k=1$ case. 	We bound the terms in order. By l'Hopital rule, we get that $\frac{(h'')^2}{(1-(h')^2)}\lesssim_h 1$. Using Hardy inequality and the form of $\mathcal{E}$ \eqref{energy integral} we get
		\begin{equation*}
			\begin{gathered}
				\norm{rV\psi}_{L^2[\{r\leq R\}]}+\norm{rh''T\psi}_{L^2[\{r\leq R\}]}\lesssim \mathcal{E}[\psi]^{1/2}.
			\end{gathered}
		\end{equation*}
		Using $1-(h')^2,h'\leq 1$ and Hardy inequality we get
		\begin{equation*}
			\begin{gathered}
				\norm{(1-h'^2)T^2\psi}_{L^2[\{r\leq R\}]}+\norm{2h'\partial_r T\psi}_{L^2[\{r\leq R\}]}+\norm{\frac{h'}{r}T\psi}_{L^2[\{r\leq R\}]}\lesssim\mathcal{E}[T\psi]^{1/2},
			\end{gathered}
		\end{equation*}
		and the claim follows. Higher order estimates follow after commutation with $S$. 
		\begin{equation}
			\begin{gathered}
				[S,r^2(\Box-V)]=-r(r^2V)'\\
				[S^k,r^2(\Box-V)]=-\sum_{\substack{k_1+k_2=k\\ k_2\neq k}}(S^{k_1}r^2V)S^{k_2}.
			\end{gathered}
		\end{equation}	
		Therefore, we get
		\begin{equation}
			\begin{gathered}
				\norm{S^k\psi}_{\dot{H}^{1}[\{r<R\}]}\lesssim_c \norm{r\Delta S^{k-1}\psi}_{L^2[\{r<R\}]} +\norm{S^{k-1}\psi}_{\dot{H}^1[r\leq R]}
			\end{gathered}
		\end{equation}
		where we control the first term with
		\begin{equation}
			\begin{gathered}
				\norm{r\Delta S^{k-1}\psi}_{L^2[r< R]}= \norm{r^{-1}S^{k-1}(r^2\Delta\psi)}_{L^2[r< R]}\\
				\lesssim\norm{r^{-1}S^{k-1}r^2F}_{L^2[r< R]}+\norm{r^{-1}S^{k-1}(V\psi)}_{L^2[r< R]}+\norm{r^{-1}S^{k-1}r^2(1-h'^2)T^2\psi}_{L^2[r< R]}\\
				+\norm{r^{-1}S^{k-1}r^2\partial_r T\psi}_{L^2[r< R]}+\norm{r^{-1}S^kr^2h''T\psi}_{L^2[r< R]}+\norm{r^{-1}S^{k-1}\frac{h'}{r}T\psi}_{L^2[r< R]}.
			\end{gathered}
		\end{equation}
		We bound each term separately
		\begin{equation*}
			\begin{gathered}
				\norm{r^{-1}S^{k-1}r^2F}_{L^2[r< R]}\lesssim_{k}\sum_{k'\leq k-1}\norm{rS^{k'}F}_{L^2[r< R]}\\
				\norm{r^{-1}S^{k-1}(V\psi)}_{L^2[r< R]}\lesssim_{k}\sum_{k'\leq k-1}\norm{S^{k'}\psi}_{\dot{H}^1}\\
				\norm{r^{-1}S^{k-1}r^2\partial_r T\psi}_{L^2[r< R]}\lesssim_{k}\sum_{k'\leq k-1}\norm{rS^{k'}T\psi}_{\dot{H}^1}\\
				\norm{r^{-1}S^kr^2h''T\psi}_{L^2[r< R]},\norm{r^{-1}S^{k-1}rh'T\psi}_{L^2[r< R]}\lesssim_{k}\sum_{k'\leq k-1}\norm{S^{k'}T\psi}_{L^2}\underbrace{\lesssim}_{\text{Hardy}}\sum_{k'\leq k-1}R\norm{S^{k'}T\psi}_{\dot{H}^1}\\
				\begin{multlined}
					\norm{r^{-1}S^{k-1}r^2(1-h'^2)T^2\psi}_{L^2[r< R]}\lesssim_{k}{\tau}^{k-1}R\norm{(1-h'^2)T^{k+1}\psi}_{L^2}+\sum_{k'\leq k-2}R^2\norm{S^{k'}T^2\psi}_{\dot{H}^1}\\	
					\lesssim {\tau}^k\mathcal{E}[T^k\psi]^{1/2}+\sum_{k'\leq k-2}R^2\norm{S^{k'}T^2\psi}_{\dot{H}^1}
				\end{multlined}
			\end{gathered}
		\end{equation*}
		By induction we get the lemma.
	\end{proof}
	
%
%

	\begin{lemma}[Far commutation]\label{far commutation}
		Let $\tau_2>\tau_1>R/10$ with $\tau_2-\tau_1<R$, let $\psi$ be a solution to \eqref{linearised eom2} and $\chi=1-\bar{\chi}(\frac{r}{t\delta})$ . Then there exists $c_\beta>0$ such that
		\begin{equation}
			\begin{gathered}
				\mathcal{E}^V_{\tau_1}[\Gamma^\beta\psi]\leq \mathcal{E}^V_{\tau_2}[\Gamma^\beta\psi]+\mathcal{E}_{\scri_{\tau_1,\tau_2}}[\Gamma^\beta\psi]+c_\beta\sum_{k'\leq \abs{\beta}}{\tau}^{2k'} \mathcal{X}[T^{k'}\psi]\\
				+c_\beta\sum_{\abs{\alpha}+m\leq k-1}{\tau}^{2m}\norm{\Gamma^\alpha T^mF}^2_{L^2[r<2R]}+c_\beta\rint \abs{\partial_t\Gamma^\beta\psi} \abs{\Gamma^\alpha F}
			\end{gathered}
		\end{equation}
		where $\Gamma\in\{\chi S,\chi \Omega\}$.
	\end{lemma}
	\begin{proof}
		We will only write the proof when $\Gamma=Z:=\chi S$, the mixed case follows similarly with less error terms to control.
		Let's denote by $\mathfrak{D}^k$ any finite sum of arbitrary up to $k$ fold product of vector fields $S,r\partial_t$ times a conormal function $\mathcal{O}(1)$ supported on $\supp\chi'$.  Commuting $Z$ with $\Box-V$ we get
		\begin{equation}
			\begin{gathered}
				[Z,r^2(\Box-V)]=-r(r^2V)'+r^{-2}\mathfrak{D}^2\\
				[Z^k,r^2(\Box-V)]=-\sum_{\substack{k_1+k_2=k\\ k_2\neq k}}(S^{k_1}r^2V)S^{k_2} + r^{-2}\mathfrak{D}^{k+1}.
			\end{gathered}
		\end{equation}
		We use a $\partial_t$ multiplier to get an energy estimate
		\begin{equation}
			\begin{gathered}
				\mathcal{E}^V_{\tau_2}[Z^k\psi]\leq \mathcal{E}^V_{\tau_1}[Z^k\psi]+\mathcal{E}_{\scri_{\tau_1,\tau_2}}[S^k\psi]+\rint \frac{(\mathfrak{D}^{k+1}\psi)^2}{r^3}+(\partial_tZ^k\psi) (r^{-2}S^kr^2F)
			\end{gathered}
		\end{equation}
		Using \cref{lemma:elliptic_estimate} and $\tau_2-\tau_1<R$ we get
		\begin{equation}
			\begin{gathered}
				\rint \frac{(\mathfrak{D}^{k+1}\psi)^2}{r^3}\lesssim\rint\frac{(\partial\mathfrak{D}^k\psi)^2}{r}\\
				\lesssim \sum_{k'\leq k}{\tau}^{2k'} \mathcal{X}[T^{k'}\psi]+\sum_{\abs{\alpha}+m\leq k-1}{\tau}^{2m}\norm{\Gamma^\alpha T^mF}^2_{L^2[r<2R]}.
			\end{gathered}
		\end{equation}
		
	\end{proof}
	
	We conclude with the following lemma
	
	\begin{prop}[Master current]\label{prop:master_current}
		 For sufficiently small constants $c,C$ the following holds. Let $\tau_1<\tau_2<2\tau_1$, let  $\psi$ be a solution to \eqref{linearised eom2} in $\mathcal{R}_{\tau_1,\tau_2}$ that satisfies \eqref{higher unstable bootstrap} and $\abs{\alpha_-[T^k\psi](\tau)}\leq e^{-R_1\lamed/2}\mathcal{E}_\tau[T^k\psi]$ for $\tau\in[\tau_1,\tau_2]$. For any $\delta>0$ there exist $\kappa,R_1,R_2$ constants and currents $\masterJ=\masterJ[\psi]$ and  $\masterK=\masterK[\psi]$ with the following properties
		\begin{subequations}
			\begin{gather}
					\Sigma_{\tau_1}[\masterJ]\leq \Sigma_{\tau_2}[\masterJ]+\scri_{\tau_1,\tau_2}[\mathfrak{J}]+\rint\abs{\masterK}+\mathfrak{B}\\
					\tilde{\underline{\mathcal{C}}}_{\tau_1-x}[\masterJ]\leq\sum_{\abs{\alpha}\leq  k}{\tau}^{2\alpha_0}\mathcal{E}_\tau[\Gamma^\alpha \psi]+\scri_{\tau_1,\tau_2}[\mathfrak{J}]+\int_{\rdom}\abs{\masterK}+\mathcal{B}\\
				\Sigma_{\tau}[\masterJ]+\sum_{k'\leq k}\kappa^{-2k'}{\tau}^{2k'}\Big(\abs{\tilde{\Theta}^{com}_\tau[T^{k'}\psi]}^2+\abs{\tilde{\Theta}^{\Lambda}_\tau[T^{k'}\psi ]}^2\Big)\geq C \sum_{\abs{\alpha}\leq  k}{\tau}^{2\alpha_0}\mathcal{E}_\tau[\Gamma^\alpha \psi]\\
				\tilde{\underline{\mathcal{C}}}_{(\tau_2-x)/2}[\masterJ]\geq C \sum_{\abs{\alpha}\leq k}{\tau}^{2\alpha_0}\underline{\mathcal{C}}_x[J^0[\psi]], \quad  x>R_2\\
				\rint\abs{\masterK}\leq \delta \master^\kappa,\quad
				\mathfrak{B}\lesssim \sum_{\abs{\alpha}\leq k} \rint {\tau}^{2\alpha_0}T(\Gamma^\alpha\psi) \Gamma^\alpha F \\
				\Gamma\in\kappa^{-1}\{T,\kappa^{-1}\Omega_i,\kappa^{-1}S\}\\
				\begin{multlined}
					\master[\phi]:=\sum_{\abs{\alpha}\leq k-1}	\sup_{\tau\in(\tau_1,\tau_2)}{\tau}^{2\alpha_0}\Sigma_\tau[J[\Gamma^\alpha\phi]]+ \sup_{x>\tau_2/2} \tilde{\underline{\mathcal{C}}}_{-x}[{\tau}^{2\alpha_0}J[\Gamma^\alpha\phi]]\\
					+\sup_{x>R_1} \tilde{\underline{\mathcal{C}}}_{(\tau_2-x)/2}[{\tau}^{2(k-1)}J[T^{k-1}\phi]]
				\end{multlined}
			\end{gather}
		\end{subequations}
		In particular we can take
		\begin{equation}
			\begin{gathered}
				\masterJ=\sum_{\abs{\alpha}\leq k}{\tau_1}^{\alpha_0}J^V[\Gamma^\alpha\psi]\kappa^{-\abs{\alpha}},\quad 
				\Gamma\in\kappa^{-1}\{T,\kappa^{-1}\Omega_i,\kappa^{-1}S\}\\
				-J^V[\psi]=T\cdot \T^V[\psi]+cT\cdot \tilde{\T}^V[\psi]\\ \tilde{\T}^V[\psi]_{\mu\nu}=r^{-2}\partial_{\mu}(r\psi)\partial_{\nu}(r\psi)-\frac{1}{2}\eta_{\mu\nu}\Big(r^{-2}\partial_\sigma(r\psi)\partial^\sigma(r\psi)-V\psi^2\Big)
			\end{gathered}
		\end{equation}
	\end{prop}	
	
	\begin{proof}
		Let's first pick the constants $c,C$. Let $c_1$ be the constant $\mu$ from \cref{coercivity lemma}. Using Hardy inequality and the boundedness of $Vr^2$ we get that there exists $c_2$ such that
		\begin{equation*}
			\begin{gathered}
				\mathcal{E}[\psi]+c_2\Sigma_\tau[T\cdot \tilde{\T}^V[\psi]]\geq0.
			\end{gathered}
		\end{equation*}
		We fix $C=\min (\frac{1}{100},c_1/2),c=\min(\frac{c_1c_2}{2},\frac{1}{100})$, to get
		\begin{equation*}
			\begin{gathered}
				2C^{-1}(\abs{\tilde{\Theta}^{com}[\psi]}^2+\abs{\tilde{\Theta}^\Lambda[\psi]}^2)+\Sigma_\tau[J^V[\psi]]\geq C\mathcal{E}[\psi].
			\end{gathered}
		\end{equation*} 
		Furthermore, note that 
		\begin{equation*}
			\begin{multlined}
				2J^V[\psi]\cdot(\partial_u)\geq (1+c)(\partial_u\psi)^2-2c\frac{\psi}{r}\partial_u\psi+cr^{-2}\psi^2-(1+c)V\psi^2\\
				\geq (1-c)(\partial_v\psi)^2+\frac{c}{2r^2}\psi^2-(1+c)V\psi^2
			\end{multlined}
		\end{equation*}
		and so we get that for $R_2^2>10c^{-1}$
			\begin{equation*}
			\begin{gathered}	
				\underline{\mathcal{C}}_x[J^V]\geq \frac{1}{2}\underline{\mathcal{C}}_x[J^0]
			\end{gathered}
		\end{equation*}
		For this choice, $\abs{\alpha}=0$ estimate follows from \cref{uncommuted estimate section}. For higher commuted estimates, we pick $\kappa$ to be the largest constant that appears in \cref{far commutation} and in \cref{lemma:elliptic_estimate} times $\delta^{-1}$.
	\end{proof}

	\section{Nonlinear part}\label{sec:nonlinear}
	We will apply the linear theory developed in the previous section as a black box. In particular, we will consider $\psi$ a solution of \eqref{main equation} to be a the solution of \eqref{linearised eom2} with $F=\mathcal{N}[\psi]$.
	
	Using \cref{existence of scattering solution}, we know that it is sufficient to do a bootstrap argument to get a solution in $\rdom$, since local solution is known to exists. Let's introduce the quantity
	\begin{equation}
		\begin{gathered}
			\Theta_\tau[\psi]:=(\kappa^{-1}\tau\Theta^{mom}(\tau),\Theta^{com}(\tau)+\tau\Theta^{mom}(\tau),\Theta^\Lambda(\tau))
		\end{gathered}
	\end{equation}
	where $\kappa$ will be the same one as in $\masterJ$.
	
	\begin{lemma}\label{lemma:dyadic nonlinear iteration}
		Fix $\gamma\in(0,1)$. There exists $C_1,\epsilon_0$ such that the following holds.
		Let $\underline{\psi}$ be scattering data for \eqref{main equation} of order $6$, decay $q=0$ and size $\epsilon_0$ at $\scri$ and $\Sigma_{\tau_2}$. Assume that there exist $\psi$ a scattering solution in $\rdom$ that satisfies \eqref{higher unstable bootstrap}
		and 
		\begin{equation}
			\begin{gathered}
				\abs{\Theta_{\tau_2}}^2,\abs{\mathcal{R}^\Lambda_\scri}^2,\Sigma_{\tau_2}[\masterJ]+\scri_{\tau_1,\tau_2}[\masterJ]\leq \epsilon
			\end{gathered}
		\end{equation}
		Then, for all $\epsilon<\epsilon_0<C_1^{-1}$ and $2\tau_1>\tau_2>\tau_1\gg1$ it holds that
		\begin{equation}
			\begin{gathered}
				\master[\psi]\leq C_1\epsilon,\quad \epsilon(\tau_2-\tau_1)^6<1,\quad
				\implies 	\master[\psi]\leq C_1\epsilon/2.
			\end{gathered}
		\end{equation}
		In particular for $\tau_2-\tau_1<\epsilon^{-1/6}$ we have $\mathcal{X}_{\rdom}[\psi]\leq C_1\epsilon$. Furthermore, for $\masterJ$ introduced in \cref{prop:master_current} we have
		\begin{equation}
			\begin{gathered}
				\gamma \Big(\abs{\Sigma_{\tau_1}[\masterJ]}+\abs{\Theta_{\tau_1}}^2\Big)\leq \Big(\abs{\Sigma_{\tau_2}[\masterJ]}+\abs{\Theta_{\tau_2}}^2\Big)+\abs{\scri_{(\tau_1,\tau_2)}[\masterJ]}+\abs{\mathcal{R}^\Lambda_\scri}^2+C_2\epsilon^{1.5}\jpns{\tau_2-\tau_1}^{5/2}
			\end{gathered}
		\end{equation}
	\end{lemma}
	\begin{remark}
		This lemma in particular implies that for scattering data with incoming $\masterJ$ energy smaller than $\epsilon$, $\psi$ can only fail to exist on $\rdom$ if the unstable bootstrap assumption \eqref{higher unstable bootstrap} fails.
	\end{remark}
	\begin{proof}[Proof of \cref{lemma:dyadic nonlinear iteration}]
		Throughout this proof, we can use \cref{lemma:theta_scri_control} to bound $\abs{\tilde{\Theta}^\bullet[T^k\psi]}\leq\abs{\Theta^\bullet[T^k\psi]}+\epsilon$ for $k\leq N,\bullet\in\{com,\Lambda\}$. Therefore, without loss of generality, we drop the tilde whenever we use the coercivity from \cref{prop:master_current}.
		
		Let $\Tau=\tau_2-\tau_1$ denote the maximum existence time of the scattering solution $\psi$ such that
		\begin{equation}\label{eq:dyadic_bootstrap}
			\begin{gathered}
				\mathcal{X}[\psi]_{\rdom}\leq C_1\epsilon
			\end{gathered}
		\end{equation}
		for some $C_1\gg1$ to be determined.	From \cref{existence of scattering solution}, we know that there exists $C>0$ such that $\Tau>0$ for $C_1>C$. Let's pick $\masterJ$ current from \cref{prop:master_current}. Introduce notation
		\begin{equation}
			\begin{gathered}
				\theta_\tau:=\sum_{k'\leq k}\kappa^{-2k'}{\tau}^{2k'}\Big(\abs{\Theta^{com}_\tau[T^{k'}\psi]}^2+\abs{\Theta^\Lambda_\tau[T^{k'}\psi ]}^2\Big).
			\end{gathered}
		\end{equation}
		Because of its coercivity property, we may use the current $\masterJ$ to control the $\master^\kappa$ norm of the solution
		\begin{equation}\label{control of master with masterJ}
			\begin{gathered}
				\sup_{x>R_2}\tilde{\underline{\mathcal{C}}}_{(\tau_2-x)/2}[\masterJ]+\sup_{\tau\in[\tau_1,\tau_2]}	(\Sigma_{\tau}[\masterJ]+\theta_\tau)\geq C\master^\kappa
			\end{gathered}
		\end{equation}
		for some absolute $C$. We now pick $\delta(\gamma,C)$ sufficiently small to be determined later. \cref{prop:master_current} says
		\begin{equation}\label{nonlinear bound on current}
			\begin{gathered}
				\sup_\tau\Sigma_{\tau}[\masterJ]\leq  \Sigma_{\tau_2}[\masterJ]+\scri_{\tau_1,\tau_2}[\masterJ]+\delta\master^\kappa+\mathcal{B}\leq 2\epsilon +\delta\master^\kappa +\mathcal{B}\\
				\sup_{x>R_2}\tilde{\underline{\mathcal{C}}}_{(\tau_2-x)/2}[\masterJ]\leq C^{-1}(\Sigma_{\tau_2}[\masterJ]+\theta_{\tau_2})+\scri_{\tau_1,\tau_2}[\masterJ]+\delta\master^\kappa+\mathcal{B}\leq C^{-1}2(\epsilon+\theta_{\tau_1})+\delta\master^\kappa+\mathcal{B}
			\end{gathered}
		\end{equation}
		We trivially have that $\master^\kappa\sim_\kappa\master$. Using the scaling of $\mathcal{B}$ with $\psi$, \cref{nonlinear bounds lemma,eq:dyadic_bootstrap} we get $
		\mathcal{B}\lesssim_\kappa C_1^{1.5}\epsilon^{1.5}\jpns{\tau_2-\tau_1}^{3/2}$.	
		For the 0 modes, using \cref{higher order zero modes,nonlinear bounds lemma} together with the scaling of $F$ with $\psi$ we get
		\begin{equation}\label{nonlinear bound on modes}
			\begin{gathered}
				\theta_\tau\leq\abs{\Theta^{com}_\tau}^2+\kappa^{-2}\tau^2\abs{\Theta^{mom}_\tau}^2+\abs{\Theta^\Lambda_\tau}^2+C_1^2\epsilon^{1.5}\\ \leq \abs{\Theta^{com}_{\tau_2}}^2+\kappa^{-2}{\tau_2}^2\abs{\Theta^{mom}_{\tau_2}}^2+\abs{\Theta^\Lambda_{\tau_2}}^2+\abs{\mathcal{R}^\Lambda_\scri}^2+C_1^2\epsilon^{1.5}+C_1^2\jpns{\tau_2-\tau_1}^{5/2}\epsilon^{1.5}\lesssim \epsilon+ C_1^2\epsilon^{1.5}\jpns{\tau_2-\tau_1}^{5/2}
			\end{gathered}
		\end{equation}
		Putting all of these estimates together, we conclude that
		\begin{equation}
			\begin{gathered}
				\master^\kappa\leq (C^{-1}+3)\epsilon+\delta\master^\kappa+C_1^2\epsilon^{1.5}\jpns{\tau_2-\tau_1}^{5/2}\\
				\implies \master\lesssim_\kappa \epsilon+ C_1^2\epsilon^{1.5}\jpns{\tau_2-\tau_1}^{5/2},
			\end{gathered}
		\end{equation}
		where the implication follows by choosing $\delta<1$. Choosing $C_1$ larger than twice the implicit constant in the second equation and letting $\epsilon_0$ be sufficiently small based on the implicit constant and $C_1$ we get the first part of the lemma.

		For the second part, we simply use \cref{control of master with masterJ} in \cref{nonlinear bound on current,nonlinear bound on modes} together with the nonlinear bounds to get
		\begin{equation}
			\begin{gathered}
				\sup_\tau (\theta_\tau+\Sigma_{\tau}[\masterJ])\leq \abs{\Theta_{\tau_2}}+\Sigma_{\tau_2}[\masterJ]+\delta 3\jpns{C^{-2}}	\sup_\tau (\theta_\tau+\Sigma_{\tau}[\masterJ])+C_\kappa\epsilon^{1.5}\jpns{\tau_2-\tau_1}^{5/2}
			\end{gathered}
		\end{equation}
		Picking $\delta$ in this equation small enough yields the result.
		
	\end{proof}
		
	\subsection{Iteration}
	Let's define $\mathcal{R}_{m}=\mathcal{R}_{2^{m-1},2^{m}}$, and 
	\begin{equation}
		\begin{gathered}
			\mathcal{X}^{q}_{\rdom}[\psi]=\sum_{m\in (\floor{\log{\tau_1}},\ceil{\log{\tau_2}})} 2^{q m}\mathcal{X}_{R_m\cap\rdom}[\psi]\\
			\mathfrak{I}_{\rdom}^{q}=\sum_{m\in (\floor{\log{\tau_1}},\ceil{\log{\tau_2}})} 2^{q m}\scri_{-2^{m+1},2^{-m}}[\masterJ]\\
			\mathfrak{R}^{q}_{\rdom}=\sum_{m\in (\floor{\log{\tau_1}},\ceil{\log{\tau_2}})} 2^{q m/2}\mathcal{R}^\Lambda_{-2^{m+1},2^{-m}}[\psi]\\
			\theta_{\rdom}^{q}=\sup_{\substack{m\in (\log{\tau_1},\log{\tau_2}) \\\tau\in\mathcal{R}_m\cap\rdom}} 2^{qm/2}\abs{\Theta(\tau)}.
		\end{gathered}
	\end{equation}
	We also introduce an unstable mode bootstrap assumption
	\begin{equation}\label{dyadic unstable bootstrap assumption}
		\begin{gathered}
			\sup_{\tau\in(\tau_1,\tau_2)}\abs{\alpha_-[{T^N\psi}]}{\tau}^{q/2+N}< e^{-\lamed R_1/2}
		\end{gathered}
	\end{equation}
	\begin{lemma}\label{iteration lemma}
		Let $\tau_2=2^{n}$ for $n\gg1$. Let $\underline{\psi}$ be scattering data for \eqref{main equation} of order $6$ size $1$ and decay $q>5$ at $\scri$ and $\Sigma_{\tau_2}$. Assume that there exist $\psi$ a scattering solution in $\rdom$ that satisfies \eqref{dyadic unstable bootstrap assumption} and 
		\begin{equation}
			\begin{gathered}
				{\tau_2}^{q}(\abs{\Theta(\tau_2)}^2,\Sigma_{\tau_2}[\masterJ]),\mathfrak{R}^{q}_{{0,\tau_2}},\mathfrak{I}_{{0,\tau_2}}^{q}\leq 1
			\end{gathered}
		\end{equation}
		Then, there exists $C_1>0$ and $\tau_1$ such that for any $\tau_2>\tau_1$
		\begin{equation}\label{iterative bootstrap}
			\begin{gathered}
				\mathcal{X}^{q}_{\rdom}[\psi]\leq C_1
			\end{gathered}
		\end{equation}
	\end{lemma}
	\begin{proof}
	Similarly as in the proof of \cref{lemma:dyadic nonlinear iteration}, we will treat $\tilde{\Theta}$ and $\Theta$ exchangeable for coercivity estimates.	
		
	Let's fix $\masterJ$ as in \cref{lemma:dyadic nonlinear iteration}, in particular, choose a fixed $\kappa,\delta$ such that we can pick $\gamma=1-\frac{1}{100}$. We will not indicate the dependence of the constant on $\delta,\kappa,\gamma$, as we treat them as fix parameters from now on.
	
	For ease of notation, we will only prove the case when $
	\tau_1=2^{M_1},\,\tau_2=2^{M_2}$, the other case follows similarly.
	
	We will proceed by induction. Let's introduce
	\begin{equation}
		\begin{gathered}
			a_\tau:=\abs{\Theta(\tau)}^2+\Sigma_{\tau}[\masterJ]\\
			b_{\tau}:=\mathcal{R}^\Lambda_{\tau/2,\tau}+\scri_{\tau/2,\tau}[\masterJ].			
		\end{gathered}
	\end{equation}		
	The assumption of the lemma in these variable implies
	\begin{equation*}
		\begin{gathered}
			2^{qM_2}a_{2^{M_2}},2^{qm}b_{2^{m}}\leq 1,\quad \forall m
		\end{gathered}
	\end{equation*}
	We now prove the following
	\begin{claim}
		There exists $c$ such that $a_{2^m}\leq c2^{-mq}$ whenever $m>M_1$ with $M_1\gg 1$.
	\end{claim}
	Pick $M_2$, depending on $c$, such that $(c+1)2^{-M_2q}\leq\epsilon_0$, where $\epsilon_0$ is the constant from \cref{lemma:dyadic nonlinear iteration}. Note, that for $m=M_1$ the claim holds for $c>1$. Using \cref{lemma:dyadic nonlinear iteration} and induction on $m$, we get 
	\begin{equation*}
		\begin{gathered}
			a_{2^m}\leq\gamma^{-1}\big(a_{2^{m+1}}+b_{2^{m+1}}\big)+C\big(a_{2^{m+1}}+b_{2^{m+1}}\big)^{1.5}2^{5m/2}\\
			\leq \frac{\gamma^{-1}}{2}2^{-qm}\Big((\frac{c}{2^{q}}+1)+C(\frac{c}{2^{q}}+1)^{1.5}2^{m(5/2-q/2)}\Big)
		\end{gathered}
	\end{equation*}
	Fixing $c>2C+10$ closes the induction, which proves the claim.
	
	Using the boundedness of the fluxes, the bound on $\master$ follows from \cref{lemma:dyadic nonlinear iteration}.

	\end{proof}
	
	\subsection{Unstable mode}\label{sec:unstable_mode}
	To treat the unstable mode contribution, we introduce the solutions to the modified problem
	\begin{equation}\label{mlinear eom}
		\begin{gathered}
			(\Box+V)\psia=F,\qquad F,\psia:\mathcal{R}_{(\tau_1,\tau_2)}\to \R\\
			(r\psia)|_{\scri_{(\tau_1,\tau_2)}}=\psia^{dat}, \quad \psia|_{\Sigma_{\tau_1}}=a\lamed Y\chi\big(\frac{r}{3R_1}\big)+\psi_0,\quad \partial_t\psia|_{\Sigma^f_{\tau_1}}=-aY\chi\big(\frac{r}{3R_1}\big)+\psi_1.
		\end{gathered}
	\end{equation}
	We also introduce $\alphaa_\pm$ similar to $\alpha_\pm$. 
	\begin{lemma}\label{unstable lemma}
		Let $\underline{\psi}$ be scattering data for \eqref{main equation} with $\psi_k|_{\Sigma_{\tau_2}}=0$. Let $\psia$ be the modified scattering solution to \cref{mlinear eom}. There exists $\abs{a}\lesssim{\tau_2}^{-q}$ such that the corresponding scattering solution $\psia$ exists in $\rdom$ and satisfies \cref{dyadic unstable bootstrap assumption} as well as \cref{iterative bootstrap}.
	\end{lemma}
	\begin{proof}
		Let $\tau_1$ be from \cref{iteration lemma}. First, let us introduce $\mathcal{T}:[-C_1,C_1]\to\R$ 
		\begin{equation}
			\begin{gathered}
				\mathcal{T}(a):=\inf\{\max\{\tau_1,\tau\}: \exists \psia\text{ in }\mathcal{R}_{\tau,\tau_2} \text{ solution of } \cref{mlinear eom} \text{ satisfying \cref{dyadic unstable bootstrap assumption} and \cref{iterative bootstrap}} \}
			\end{gathered}
		\end{equation}

		\textit{Boundary values:} Let's pick $C_1={\tau_2}^{-q/2}C_2$ such that $\tau^{N+q/2}\abs{\alpha_-[{T^N\prescript{}{{\tiny{\pm C_2}}}{\psi}}](\tau_2)}=e^{-\lamed R_1/2}$. It's easy to see that in this case $\im(\mathcal{T})\subset[\tau_1,\tau_2]$.
		
		\textit{Continuity:} We prove that $\mathcal{T}$ is continuous. Let $a\in[-C_1,C_1]$ be such that the maximum in the definition of $\mathcal{T}$ is only attained at $\tau_1$. Then by Cauchy stability $\mathcal{T}=\tau_1$ in a neighbourhood of $a$. 
		
		Pick $a$ otherwise. Then from \cref{iteration lemma} it follows that
		 \begin{equation}
		 	\begin{gathered}
		 		\abs{\alpha_-[{T^N\psia}](\mathcal{T}(a))}\abs{\mathcal{T}(a)}^{q/2+N}=e^{-\lamed R_1/2}.
		 	\end{gathered}
		 \end{equation}
	 	Furthermore at $\tau=\tilde{\tau}:=\mathcal{T}(a)$, we have
	 	\begin{equation}
	 		\begin{gathered}
	 			\partial_\tau \Big(\alpha_-[{T^N\psia}](\tau){\tau}^{q/2+N}\Big)|_{\tau=\tilde{\tau}}=\Big(-\lamed\alpha_-[{T^N\psia}](\tilde{\tau})+\mathfrak{Err}_+[T^N\psia](\tilde{\tau})\Big){\tilde{\tau}}^{q/2+N}\\
	 			+(\frac{q}{2}+N)\alpha_-[{T^N\psia}](\tilde{\tau}){\tilde{\tau}}^{q/2+N-1}
	 			=-e^{-\lamed R_1/2}(\lamed-\frac{q+2N}{2{\tilde{\tau}}})+{\tilde{\tau}}^{q/2}\mathfrak{Err}_+[T^N\psia](\tau).
	 		\end{gathered}
	 	\end{equation}
 	Using \cref{iterative bootstrap} we get the following bound for the error term
 	\begin{equation}
 		\begin{gathered}
 			{\tilde{\tau}}^{q/2+N}\int_{\Sigma_{\tilde{\tau}}}\mathfrak{Err}_-[T^N\psi]\lesssim {\tilde{\tau}}^{q/2+N}\int_{\Sigma_{\tilde{\tau}}\cap\{r\geq R_1\}}e^{-\lamed r}\Big(\frac{\abs{T^N\psi}+\abs{T^N\partial\psi}}{R_1^2}+\abs{T^NF}\Big)
 			\\\lesssim  \tilde{\tau}^{q/2+N}\Big(e^{-3\lamed R_1/4}\mathcal{E}_\tau[T^N\psi]^{1/2}+ \mathcal{E}_\tau[T^N\psi]\Big) \leq  e^{-3\lamed R_1/4}+{\tau}^{-q/2}.
 		\end{gathered}
 	\end{equation}	
 	Therefore, we conclude that for $R_1$ sufficiently large
 	\begin{equation}
 		\begin{gathered}
 			\partial_\tau \Big(\alpha_-[{T^N\psia}](\tilde{\tau}){\tilde{\tau}}^{q/2+N}\Big)\leq -\frac{1}{2}e^{-\lamed R_1/2}.
 		\end{gathered}
 	\end{equation} 
 	Using Cauchy stability and the negativity of this quantity we may use the implicit function theorem to conclude that $\mathcal{T}$ is continuous.

		\textit{Topological argument:} We claim that there exists $a\in[-C_1,C_1]$ such that $\mathcal{T}(a)=\tau_1$. Assume this is not the case. Let's define $\mathcal{S}:[-C_1,C_1]\to\{\pm1\}$ by $\mathcal{S}(a)=\text{sign}\big(\alpha_-[T^N\psia](\mathcal{T}(a))\big)$. By the assumption that $\mathcal{T}(a)>\tau_1$, $\forall a$, we get that  $\alpha_-[T^N\psia](\mathcal{T}(a))\neq0$, so $\mathcal{S}$ is continuous. Also, we know that $\mathcal{S}(\pm C_1)=\pm1$. Using the mean value theorem (or Brouwer's fixed point theorem in higher dimensions)  we get a contradiction.
		
	\end{proof}
	
	\begin{remark}[Decay of unstable mode]
		The basic idea of the proof relies on the fact that $\partial_\tau\alpha_-$ is close to $\lamed\alpha_-$ when $\alpha_-$ \cref{dyadic unstable bootstrap assumption} fails, and $\mathfrak{Err}_-$ is perturbative at this time. Because $\mathfrak{Err}_-$ contains terms that are linear in $\psi$, this is only possible if $\alpha_-$ does not decay faster than $\mathcal{E}[\psi]$.
	\end{remark}
	
	We conclude with the proof of the main theorem
	\begin{proof}[Proof of \cref{main theorem}]\label{main theorem proof}
		Let $\psi_n$ denote the solutions constructed in \cref{unstable lemma} with $\tau_2=2^n$. Extending each $\psi_n$ to the future smoothly and keeping their norms bounded in $\mathcal{X}^{q}_{\mathcal{\tau_1,\infty}}$ by $2C_1\epsilon$, we get a bounded sequence of functions in $\mathcal{X}^{q,k}_{\mathcal{R}_{\tau_1,\infty}}$, where we inserted the corresponding regularity and decay of the function space. This sequence is compact in $\mathcal{X}^{q-,k-}_{\mathcal{R}_{\tau_1,\infty}}$, so there is a a subsequence converging to a limit $\psi_\infty\in \mathcal{X}^{q-,k-}_{\mathcal{R}_{\tau_1,\infty}}$.
	\end{proof}
	
	\section{Polyhomogeneity}\label{sec:polyhom}
	In this section, we improve \cref{main theorem} provided the data is polyhomogeneous.
	\begin{theorem}[Rough version of \cref{polyhomogeneity theorem}]
		Let $\psi^{dat}$ be as in \cref{main theorem}. Provided that $\psi^{dat}$ is polyhomogeneous, the solution $\psi$ is also polyhomogeneous on a blow up of Minkowski space ($\tilde{\mathcal{R}}$).
	\end{theorem}
	
	In order to prove this theorem, we will first introduce some extra notation.
	
	\subsection{Notation}
	
	\subsubsection{Analytic}

	In this subsection, we are going to introduce the function spaces and other analytic tools used in the rest of the paper. We restrict ourselves to the minimum use of sophisticated spaces that suffice for our problem so as to ease readability, but see general setup in \cite{grieser_basics_2001,hintz_stability_2020}.
	
	Let's fix a manifold with corners $X=[0,1)_{x_1}\times...\times[0,1)_{x_n}\times Y$ for some smooth manifold $Y$. First we define the vector fields with respect to which we measure smoothness.
	\begin{definition}[b vector fields]
		Let 		
		\begin{equation}
			\begin{gathered}
				V=\{x_i \partial_{x_i},Y_i\}.
			\end{gathered}
		\end{equation}
		where $Y_i$ are smooth vector fields on $Y$ spanning the tangent space at each point. Furthermore, let's define
		\begin{equation}
			\begin{gathered}
				\Diff_b^1=\sum_i f(x,y)V_i, \quad V_i\in V
			\end{gathered}
		\end{equation}	
		with $f_i\in\mathcal{C}^\infty(X)$. Also, let $\Diff^k_b$ denote a $k$-fold product of elements in $\Diff^1_b$.
	\end{definition}
	
	In application, we will mostly take $Y=S^2$ (or $Y=S^2\times (0,1)$) with $Y_i$ to be the usual spherical derivatives and $n\leq2$.
	
	\begin{definition}[$H_b$ norm]\label{def:Hb_norm}
		Given a measure $dy$ on $Y$, we define a naturally weighted norm, higher order variants and higher order variants with extra weights
		\begin{equation}
			\begin{gathered}
				\norm{f}_{L^2_b(X)}:=\int f^2 \frac{d x_1}{x_1}...\frac{d x_n}{x_n}d y\\
				\norm{f}_{H^{k}_b(X)}=\norm{f}_{H^{;k}_b(X)}:=\sum_{\abs{\alpha}\leq s}\norm{\Gamma^\alpha f}_{L^2_b}\\	
				\norm{f}_{H^{a_1,...,a_n;s}_b}:=\norm{{\prod_i x_i^{-a_i}f}}_{H^{;k}}
			\end{gathered}
		\end{equation}
		for $\Gamma^i\in V$. We will use the notation $x_1^{a_1}...x_n^{a_n}H^s_b=H^{a_1,...,a_n;s}_b$ as well.
	\end{definition}
	
	\begin{remark}
		Away from the boundary, the vector fields span the tangent space of $X$. Indeed, these spaces agree with the usual $L^2$ and $H^k$ spaces on compact subsets of $X$. The normalisation for $L^2_b$ is motivated by the observation that $\jpns{x_1}^{a_1+\epsilon}\mathcal{O}^{0}\subset H^{a_1;\infty}_b\subset \jpns{x_1}^{a_1}\mathcal{O}^{0}$, where $\mathcal{O}^0$ contains bounded conormal functions (see \cref{def:errors}).
	\end{remark}
	
	We introduce index sets, that are used to measure polyhomogeneity
	
	\begin{definition}
		A discrete subset of $\mathcal{E}\subset\R\times\N$ is called an index set if
		\begin{itemize}
			\item $(z,k)\in\mathcal{E}$ and $k\geq1$ implies $(z,k-1),(z+1,k)\in\mathcal{E}$
			\item $\mathcal{E}_c:=\{(z,k)\in\mathcal{E}| z<c\}$ is finite for all $c\in\R$.
		\end{itemize}
		Furthermore, let's introduce the following notations
		
		\begin{itemize}
			\item $\mathcal{E}_{\leq c}=\{(z,k)\in\mathcal{E}| z\leq c\}$.
			\item $(z,k)\leq(z',k')$ if $z<z'$ or $z=z'$ and $k\geq k'$. $(z,k)\leq z'$ if $(z,k)\leq (z',\infty)$ and $(z,k)\geq z'$ if $(z,k)\geq(z,0)$. 
 			\item 
			$\min(\mathcal{E})=(z,k)\in\mathcal{E}\neq\emptyset$ such that $(z,k)\leq(z',k')$ for all $(z',k')\in\mathcal{E}$. Also $\min(\emptyset)=\infty$.
		\end{itemize}
	\end{definition}

	\begin{definition}[Polyhomogeneity]
		Let $X=[0,1)_{x_1}\times Y$ be a manifold with boundary. Given an index set $\mathcal{E}$, we define the corresponding polyhomogeneous space $\mathcal{A}_{phg}^{\mathcal{E}}(X)$. For $u\in x_1^{-\infty}H^{\infty}_b(X)$ we say $u\in\mathcal{A}_{phg}^{\mathcal{E}}(X)$ if there exist $a_{z,k}\in H^{\infty}(Y)$ for $(z,k)\in\mathcal{E}$ such that for all $c\in \R$ 
		\begin{equation}\label{eq:notation:polyhom def 1}
			\begin{gathered}
				u-\sum_{(z,k)\in\mathcal{E}_c} a_{z,k}x_1^{z}\log^kx_1\in x_1^c H^{\infty}_b(X).
			\end{gathered}
		\end{equation}
		
		At a corner, we define mixed $b-$ polyhomogeneous spaces as follows. Let $X=[0,1)_{x_1}\times[0,1)_{x_2}\times Y$ be a manifold with corners, $b_2\in\R$ and $\mathcal{E}_1$ an index set. For $u\in x_1^{-\infty}x_2^{-\infty}H^{\infty}_b(X)$ we say $u\in \mathcal{A}^{\mathcal{E}_1,b_2}_{phg,b}(X)$ if there exists $a_{z,k}\in x_2^{b_2}H^{\infty}_b([0,1)_{x_2}\times Y)$ such that
		\begin{equation}
			\begin{gathered}
				u-\sum_{(z,k)\in\mathcal{E}_c} a_{z,k}x_1^{z}\log^kx_1\in x_1^c x_2^{b_2}H^{\infty}_b(X).
			\end{gathered}
		\end{equation}
		
		We define polyhomogeneous space at the corner for a manifold with corners $X=[0,1)_{x_1}\times[0,1)_{x_2}\times Y$. For $u\in x_1^{-\infty}x_2^{-\infty}H^{\infty}_b(X)$ we say $u\in\mathcal{A}_{phg}^{\mathcal{E}_1,\mathcal{E}_2}(X)$ if there exists $a_{(z,k)}\in\mathcal{A}_{phg,b}^{\mathcal{E}_2}([0,1)_{x_2}\times Y)$ such that
		\begin{equation}
			\begin{gathered}
				u-\sum_{(z,k)\in(\mathcal{E}_1)_c} a_{z,k}x_1^{z}\log^kx_1\in \mathcal{A}_{b,phg}^{c,\mathcal{E}_2}(X).
			\end{gathered}
		\end{equation}
		
	\end{definition}
	
	\begin{remark}
		Note that we may give an alternative, more geometric characterisation of a polyhomogenous function $u\in\mathcal{A}_{phg}^{\mathcal{E}}([0,1)_{x}\times Y)$ as
		\begin{equation}\label{eq:notation:polyhom def 2}
			\begin{gathered}
				\Big(\prod_{(z,k)\in\mathcal{E}_c}(x\partial_x+z) \Big)u\in x^c H^\infty_b([0,1)_{x}\times Y).
			\end{gathered}
		\end{equation}
		From this, it is easy to see that the definition of $\mathcal{E}$ only depends on the smooth structure of $[0,1)_{x}\times Y$, that is given any coordinate change $x\to z$ such that $z'(x)$ is bounded away from 0 on $[0,1)$, the index set with respect to $x$ and $z$ coincide. We will use this freedom to detect polyhomogeneity with respect to different choices of coordinates to suit our need. For further details, we refer to \cite{grieser_basics_2001,hintz_stability_2020}.
	\end{remark}

	\begin{definition}
		For index sets $\mathcal{E}_1,\mathcal{E}_2$, we define the index sets
		\begin{equation}\label{index set operations}
			\begin{gathered}
				\mathcal{E}_1\bar{\cup}\mathcal{E}_2:=\{(z,k+1)| \exists(z,k_i)\in\mathcal{E}_i,\, k_1+k_2\geq k\}\cup\mathcal{E}_1\cup\mathcal{E}_2\\\
				\mathcal{E}_1+\mathcal{E}_2:=\{(z,k)|\exists (z_i,k_i)\in\mathcal{E}_i,\, z_1+z_2=z,\,k_1+k_2=k\}\\
				\mathcal{E}_1-(z_2,k_2):=\{(z,k)|\exists (z_1,k_1)\in\mathcal{E}_1,\, z_1-z_2=z,\,k_1-k_2=k\}
			\end{gathered}
		\end{equation}
	\end{definition}
	
	We will encounter many index sets in the construction, so to ease notation we also introduce the following shorthand
	\begin{definition}
		For a discrete subset of $X\subset\R\times\N$ we write
		\begin{equation}
			\begin{gathered}
				\overline{X}=\cap_{X\subset\mathcal{E}}\mathcal{E}
			\end{gathered}
		\end{equation}
		 for the smallest index set containing $X$. When $X$ has a single element we use the shorthand $\overline{(a,b)}:=\overline{\{(a,b)\}}=\{(a+n,l):n\in\N,l\leq b\}$. We also set $\overline{(a,-1)}=\emptyset$ for $a\in\R$.
	\end{definition}
	 Finally, let us introduce a notation for error terms
	 \begin{definition}[Errors]\label{def:errors}
	 	For a function $f\in\mathcal{A}^{\vec{\mathcal{E}}}$ for some $\vec{\mathcal{E}}$, we write
	 	\begin{equation}
	 			f\in\mathcal{O}^{\overrightarrow{(p,k)}}
	 	\end{equation}
	 	whenever $\min(\mathcal{E}_\bullet)\geq(p_\bullet,k_\bullet)$. 
	 	We furthermore write $p=(p,\infty)$, so that $f\in\mathcal{O}^{\vec{p}}$ implies $\min(\mathcal{E}_\bullet)\geq p_\bullet$. 
	 \end{definition}

	\subsubsection{Geometric}\label{sec:polyhom:geometric}
	The geometric space on which we will study properties of the solution $\psi$ of \cref{main theorem} is a blow up of the compactification of $\mathcal{R}_{0,\infty}$. 
	\begin{definition}
		Let's introduce global coordinates $\rho_\scri=\frac{t_\star}{t},\rho_+=\frac{t}{t_\star\jpns{r}},\rho_F=\frac{\jpns{r}}{t}$, that serve as defining functions of null infinity, timelike infinity and the front face respectively. These satisfy $\rho_\bullet(\mathcal{R}_{-\infty,\tau_2})\subset(0,1]$ for $\bullet\in\{\scri,+,F\}$. Let's write $\tilde{\mathcal{R}}$ for the compactification of $\mathcal{R}_{0,\infty}$ with $\rho_\bullet$ smoothly extending to $0$. See \cref{fig:soliton_blowup}.
		
		On this space, we will denote by $H^{\vec{a};k}_b(\Rtild),\mathcal{A}_{phg}^{\vec{\mathcal{E}}}$ with $\vec{a}=(a_F,a_+,a_\scri),\vec{\mathcal{E}}=(\mathcal{E}_F,\mathcal{E}_+,\mathcal{E}_\scri)$ the corresponding conormal and polyhomogeneous spaces where the subscript denotes the behaviour close the given boundary.
	\end{definition}

	\begin{lemma}[Projection operators]\label{lemma: projection operators}
		Let's fix a function $f\in\mathcal{A}^{\vec{\mathcal{E}}}$ with $(p_\bullet,k_\bullet)=\min(\mathcal{E}_\bullet)$ for $\bullet\in\{F,+\}$. We can write 
		\begin{equation}
			\begin{gathered}
				f=\tilde{\chi}^c(r/R)t^{-p_+}\log^{k_+}(\rho_+) (P^+_{p_+,k_+} f)(x/t)+f',\quad f'\in \mathcal{A}^{\vec{\mathcal{E}'}},P^+_{p_+,k_+}f\in\mathcal{A}^{\mathcal{E}^B}(\dot{B})\\
				f=\tilde{\chi}(Rr/t)t^{-p_F}\log^{k_F}(\rho_F)(P^F_{p_F,k_F} f)(x)+f'',\quad f''\in \mathcal{A}^{\vec{\mathcal{E}''}},P^F_{p_F,k_F}f\in\mathcal{A}^{\mathcal{E}^\R}(\R^3).
			\end{gathered}
		\end{equation}
		where we have
		\begin{equation}
			\begin{gathered}
				\mathcal{E}'_+=\mathcal{E}_+\backslash\{(p_+,k_+)\}, \quad \mathcal{E}'_F=\mathcal{E}_F,\quad \mathcal{E}'_\scri=\mathcal{E}_{\scri}\\
				\mathcal{E}''_+=\mathcal{E}_+,\quad  \mathcal{E}''_\mathcal{F}=\mathcal{E}_\mathcal{F}\backslash\{(p_F,k_F)\},\quad \mathcal{E}''_\scri=\mathcal{E}_\scri \\
				\mathcal{E}^{\dot{B}}_{y=0}=\mathcal{E}_F-(p_+,0),\quad \mathcal{E}^{\dot{B}}_{y=1}=\mathcal{E}_\scri-(p_+,0)\\
				\mathcal{E}^{\R}=\mathcal{E}_+-(p_F,0).
			\end{gathered}
		\end{equation}
	\end{lemma}

	\subsection{Propagation of polyhomogeneity}
	In this section, we will not care about the regularity of the solution, and the derivative losses as we peel the leading order terms for the solution $\psi$. Let us note that already for $\psi$ as in \cref{main theorem} and $\psi^{dat}\in H^{q,\infty}_{b}([0,1)_{1/v}\times S^2)$ the proof in \cref{main theorem proof} in particular implies $\psi\in H^{q,q,1/2;\infty}_b(\Rtild)$. 
	
	The proof of polyhomogeneity is going to proceed in essentially 3 steps. The first is proving it at $\scri$. This step is rather different from the next to --and also more well known in the literature (see \cite{angelopoulos_late-time_2018})-- so we are only going to comment on it in the proof of \cref{polyhomogeneity theorem}. The second and third steps are performed iteratively and it involves solving model problems on the faces $F,i_+$ (similar to the methods detailed in \cite{hintz_lectures_2023}). We will now introduce these model operators.

	When restricted to the region $\frac{r}{t}\in (c_1,c_2)$ with $0<c_1<c_2<1$ the operator $\Box$ acts homogeneously on functions of the form $t^{-\sigma} f(y)$ where $y=\frac{x}{t}$ and $\rho=\abs{y}$ (see \cite{baskin_explicit_2016,baskin_asymptotics_2018} for a alternative forms)
	\begin{equation}\label{eq:normal operator}
		\begin{gathered}
			L=t^3\Box t^{-1}\\
			N_\sigma f:=t^\sigma L t^{-\sigma} f(y)=-\Big(\sigma^2+3\sigma+2+2\big((\sigma+2)\rho-\frac{1}{\rho}\big)\partial_\rho+(\rho^2-1)\partial_\rho^2-\frac{\slashed{\Delta}}{\rho^2}\Big)f
		\end{gathered}
	\end{equation}
	
	\begin{lemma}[Resolvent on hyperbolic space]
		Setting $\tilde{\rho}=\frac{1-\sqrt{1-\rho^2}}{\rho}$ and $h_\sigma(\tilde{\rho})=(\frac{1+\tilde{\rho}^2}{1-\tilde{\rho}^2})^{1+\sigma}$ we get
		\begin{equation}\label{eq:I0normal_operator}
			\begin{gathered}
			 \frac{1}{h_\sigma(\tilde{\rho})}	N_\sigma h_\sigma(\tilde{\rho})f(\tilde{\rho})=\frac{(\tilde{\rho}^2+1)^2}{4}\Big(\Delta_{\mathbb{H}}-4(\sigma^2-1)\Big)\\
			 \Delta_{\H}=(1-\tilde{\rho}^2)\partial^2_{\tilde{\rho}}+\frac{2(1-\tilde{\rho}^2)}{\tilde{\rho}}\partial_{\tilde{\rho}}+\frac{\slashed{\Delta}}{\tilde{\rho}^2}.
			\end{gathered}
		\end{equation}
	\end{lemma}
	\begin{proof}
		This is a straightforward computation, we provide some steps for completeness. First, we change to coordinates such that the top order derivative terms match with the right hand side of \cref{eq:I0normal_operator}. This yields the condition $(1-\rho^2)(\frac{\partial\tilde{\rho}}{\partial\rho})^2=\frac{\tilde{\rho}^2}{\rho^2}$. Solving the ODE we get $\tilde{\rho}=\frac{1-\sqrt{1-\rho^2}}{\rho}$ as stated in the lemma and also a prefactor $\frac{\tilde{\rho}^2}{\rho^2}=\frac{(\tilde{\rho}^2+1)^2}{4}$. Next, we find a function $h$ such that upon conjugation the first order term from $\Delta_{\mathbb{H}}$ is also recovered. This is performed via the $h$ given. An explicit computation then gives the remainder 0th order term.
		\end{proof}
	
	\begin{lemma}[Radiative solutions]
		Let's fix a functions on the unit ball
		\begin{equation}
			\begin{gathered}
				f_{rad}(z)=f_r\big(\frac{z}{\abs{z}}\big)(1-\abs{z})	^{-\sigma}+H_b^{-\sigma+}(B_1),\quad f_r\in\mathcal{C}^{\infty}(S^2).
			\end{gathered}
		\end{equation}
		Then for $g=t^{-1-\sigma}f_{rad}(x/t)$
		\begin{equation}
			\begin{gathered}
				u\partial_u(rg)_{\scri}=-\sigma 2^{-\sigma}f_r(\omega) u^{-\sigma}
			\end{gathered}
		\end{equation}
	\end{lemma}
	\begin{proof}
		First, let us note that for the error term we have vanishing radiation field, as $t^{-1-\sigma}f_e(x/t)\in\mathcal{O}^{1+\sigma,1+\sigma,1+}$ for $f_e\in H^{-\sigma+}(B_1)$. We simply compute
		\begin{equation}
			\begin{gathered}
				\lim_{\substack{x=(t-u)\omega\\ r\to\infty}}\partial_u(rg)=\lim_{\substack{x=(t-u)\omega\\r\to\infty}}\partial_u(t^{-\sigma}(1-\frac{\abs{x}}{t})f_r(\omega))=-2\sigma(2u)^{-\sigma-1}f_r(\omega)
			\end{gathered}
		\end{equation}
	\end{proof}

	\begin{remark}[Choice of $N_\sigma$]
		The reason for choosing $t$ instead of $t_\star$ as a coordinate to measure decay is so that $t^{-p}\Lambda W,\, t^{-p}\partial_i W$ are almost conserved. We could have used different normal coordinate near $F$ and $I_+$, but we found this choice to give a cleaner presentation.
		
		$N_\sigma$ with $\sigma>0$ has two indicial roots at $\rho=1$, namely $-\sigma,0$ corresponding to solutions with and without incoming radiation respectively.  In this paper we will only use $N_{>0}$, but we note that 
		$N_0$ has only $\sigma=0$ as indicial root, and so the asymptotic behaviour towards $\partial B$ are $\log\rho,\rho^0$ corresponding to solutions with and without incoming radiation. These translate to $\sigma+1,-(\sigma-1)$ for $\Delta_{\H}-4(\sigma^2-1)$.
	\end{remark}
	
	\begin{remark}(Green's function)
		Following \cite{mazzeo_meromorphic_1987} --see equation (6.8) for the same expression in hyperbolic plane coordinates-- we note that the Greens function for $\Delta_{\H}-4(\sigma^2-1)$ is
		\begin{equation}\label{eq:greens_function}
			\begin{gathered}
				g_\sigma(x)=\text{const}\cdot\frac{1-r^2}{2r}\Big(\frac{1-r}{1+r}\Big)^\sigma,\qquad \sigma\geq0.
			\end{gathered}
		\end{equation}
	\end{remark}

	We study the properties of $N_\sigma$ in the following two lemmas

	\begin{lemma}\label{model operator on i_+}
		a) Let $F\in\mathcal{C}^\infty(S^2)$, and $f\in \mathcal{A}_{phg}^{\mathcal{E}_f}+H^{a;\infty}_b(B_1)$ with $a>-\sigma-1$. Then, there exists $u\in \mathcal{A}_{phg}^{\mathcal{E}}+H^{a+1;\infty}_b(B_1)$ with $\mathcal{E}=\overline{\{(\sigma,0)\}}\overline{\cup}(\mathcal{E}_f+1)$ to
		\begin{equation}
			\begin{gathered}
				N_\sigma u=f\\
				((1-\rho)^\sigma u)|_{\partial B_1}=F.
			\end{gathered}
		\end{equation}
	\end{lemma}
	\begin{proof}
		
		This is a well known fact and can be deduced from Lemma 6.13 and Lemma 6.15 of \cite{mazzeo_meromorphic_1987}, where a much more general statement is also proved. The $H_b$ estimate is also given in Lemma 6 of \cite{zworski_resonances_2016}. Polyhomogeneity follows from the expansion of the Schwartz kernel, and we give a pedestrian proof for sake of completeness. Using the symmetries of $\mathbb{H}$ we know that  the Greens function
		\begin{equation}\label{eq:GreensFunction}
			\begin{gathered}
				(\Delta_{\mathbb{H},x}-4(\sigma^2-1))G_\sigma(x,x_0)=\delta(x-x_0)\\
				\lim_{\abs{x}\to 1}(1-\abs{x})^{-\sigma}G_\sigma(x,x_0)=0,\quad \forall x_0.
			\end{gathered}
		\end{equation}
		is merely a function of the geodesic distance between the two points $x,x_0$. In particular, we may write $G(x,x_0)=g_\sigma\circ \mathfrak{r} \circ d(x,x_0)$ for $d(x,x_0)=\frac{\norm{x-x_0}^2}{(1-\norm{x}^2)(1-\norm{x_0}^2)},\mathfrak{r}(x)=(1+1/x)^{-1/2}$  and $g_\sigma$ as in \cref{eq:greens_function}.
		
		We also know that the resolvent of $\Delta_{\mathbb{H}}$ can be written as
		\begin{equation}
			\begin{gathered}
				(\Delta_{\mathbb{H}}-4(\sigma^2-1))^{-1}(f)(x)=\int \d y \frac{\abs{y}^2}{(1-\abs{y}^2)^3}G_\sigma(x,y)f(y)=\\
				\int_{\abs{y-\hat{x}}\leq 2\abs{x-\hat{x}}}\d y \frac{\abs{y}^2}{(1-\abs{y}^2)^3}G_\sigma(x,y)f(y)+\int_{\abs{y-\hat{x}}\leq 2\abs{x-\hat{x}}}\d y \frac{\abs{y}^2}{(1-\abs{y}^2)^3}G_\sigma(x,y)f(y).
			\end{gathered}
		\end{equation}
		
		The first summand can be bounded by $d(x,0)^{-3}$ times the decay of $f$. For $f\in\mathcal{A}^k(B_1)$ we have the trivial $\jpns{1-\abs{x}}^{3+k}$ bound. For the second summand, we can Taylor expand $g$ in $d$ and then $d$ in $\epsilon=\abs{x-\hat{x}}$ to get an expansion near the boundary.
		
		To improve on the expansion that is obtained above, we write $f=(1-\abs{x})^kf_k(\omega)+\mathcal{O}^{\mathcal{E}'}(B_1)$ with $\min(\mathcal{E}')<k$. Inverting the leading order term by writing $u=(1-\abs{x})^k\frac{f_k(\omega)}{4(k^2+\sigma^2)}+u_{1}$ we get that $u_1$ satisfies the equation with a faster decaying error term, so we can get an expansion to higher order. Iterating the procedure yields polyhomogeneity of the solution. Similar argument holds for logarithmic terms.
		
		This resolves the case for $F=0$. For $F\neq0$, we simply write $\tilde{u}=u-F(1-\rho)^{-\sigma}$ to get a function with vanishing boundary term.
	\end{proof}

	We can also extend the above to the punctured ball.
	
	\begin{lemma}\label{model operator on i_+ punctured}
		a) Let $X=B_1\backslash\{0\}$, $F\in\mathcal{C}^\infty(S^2)$, and $f\in\mathcal{A}_{phg}^{\mathcal{E}_{f;0},\mathcal{E}_{f;1}}+ H^{a_0,a_1;\infty}_b(X)$ with $a_1>-\sigma-1,a_0>-2$. Then, there exists $u\in \mathcal{A}_{phg}^{\mathcal{E}_0,\mathcal{E}_1}+H^{a_0+2,a_1+1;\infty}_b(B_1)$ with $\mathcal{E}_1=\overline{\{(\-\sigma,0)\}}\overline{\cup}\mathcal{E}_{f;1}$ and $\mathcal{E}_0=\overline{\{(0,0)\}}\overline{\cup}(\mathcal{E}_{f;0}+2)$ to
		\begin{equation}
			\begin{gathered}
				N_\sigma u=f\\
				((1-\rho)^\sigma u)|_{\partial B_1}=F.
			\end{gathered}
		\end{equation}
		
		b) For $F=0,f=(t\Lambda W)_{I_+}$ we have $(N_\sigma^{-1}f)|_{y=0}\neq0$ for $\sigma>0$.
		
		c) For $F=0,f=(t^2\partial_i W)_{I_+}$ we have $(N_\sigma^{-1}f)|_{y=0}\neq0$ for $\sigma>0$.
	\end{lemma}
	
	\begin{proof}
		\textit{a)} The $H_b$ estimate follows the same way as the punctured estimate in \cite{hintz_lectures_2023} theorem 7.1.
		Polyhomogeneity close to $\partial B_1$ follows as in \cref{model operator on i_+}. Near $\rho=0$ we can iteratively invert the normal operator $\partial_r^2+\frac{2}{r}\partial_r+\frac{\slashed{\Delta}}{r^2}$ to get that $u=\tilde{u}_N+\mathcal{A}^{\mathcal{E}_0}$ where $\tilde{u}$ satisfies $N_\sigma \tilde{u}_N\in H_b^{N,a_1;\infty}$ for $N$ arbitrary. It remains to prove that $\tilde{u}\in\mathcal{A}^{\mathcal{E}_0}+H_b^{N+2,a_1;\infty}$. We write
		\begin{equation}
			\begin{gathered}
				(\Delta_{\mathbb{H}}-4(\sigma^2-1))^{-1}(f)(x)=\int_{\abs{y}\leq 2\abs{x}} \frac{\abs{y}^2}{(1-\abs{y}^2)^3}g(d(x,y))f(y)+\int_{\abs{y}\geq 2\abs{x}} \frac{\abs{y}^2}{(1-\abs{y}^2)^3}g(d(x,y))f(y)\\
				f\in H_b^{N,a_1;\infty}.
			\end{gathered}
		\end{equation}
		The first term is bounded by $\abs{y}^{N+2}$, while for the second, we may Taylor expand $\abs{x-y}$ in $x$. This expansion will contain terms and remainder of the form $\abs{x}^{k-1}\int \frac{f(y)}{\abs{y}^k} g_k(\hat{x},\frac{\abs{y}}{\abs{x}},y)$   where 
		\begin{equation}
			\begin{gathered}
				g_k:\mathcal{C}^\infty(S^2\times\R_{\geq 2}\times \overline{B_1})
			\end{gathered}
		\end{equation}
		are uniformly bounded. Since $\int \frac{f(y)}{\abs{y}^{k}}$ converges for $k<N+3$, the leading error term will be $\abs{x}^{N+2}$.
		
		\textit{b)} Note that using coordinate $y$ we can write $(t\Lambda W)=c(\frac{t}{r})=c\frac{1}{\rho}$. One can verify that for the spherically symmetric part of $N_\sigma$, $\frac{1}{x(1+x)^\sigma},\frac{1}{x(1-x)^\sigma}$ are the two homogeneous solutions corresponding to the well known $\frac{1}{ru^\sigma},\frac{1}{rv^\sigma}$ solutions. After some ODE analysis, one finds that the explicit solution is given by
		\begin{equation}
			\begin{gathered}
				u(\rho)=-\frac{1}{\sigma(1+\sigma)\rho}\Big(1-\frac{1}{(1+\rho)^\sigma}\Big).
			\end{gathered}
		\end{equation}
		Evaluating at $\rho=0$, we find that $u(0)=\frac{1}{1+\sigma}\neq$ for $\sigma>0$.
		
		\textit{c)} We could proceed similarly as in b), but it is in fact not necessary to find the explicit solution. The ODE governing the $l=1$ mode has two regular singular points at $\rho=0$ with roots $-2,1$ corresponding to the usual behaviour of spherical harmonic in $\R^3$. The forcing $(t^2\partial_iW)_{I_+}=\frac{\hat{y}_i}{\rho^2}$ has weight $-2$ and by general Fuchsian ode analysis it follows that $u\in\mathcal{A}^{\overline{(0,0)},\overline{(0,0)}}$. Furthermore, the leading order term is not vanishing as seen from the equation $-\frac{\slashed{\Delta}}{\rho^2}f\sim \frac{\hat{y}_i}{\rho^2}$ valid near $\rho=0$.
	\end{proof}
	
	Next, we turn our attention to the model problem on the front face. We observe, that on functions of the form $t^{-\sigma}f(x)$ the leading order contribution of $\Box+V$ is $\Delta+V$. 
	
	\begin{lemma}\label{model operator on F}
		Let $f\in \mathcal{A}^{\mathcal{E}_f}_{phg}+H^{a;\infty}_b$ with $a>3$ and $\jpns{f,\Lambda W}=\jpns{f,\partial_i W}=0$. Then, there exists solution $u\in\mathcal{A}^{\mathcal{E}}_{phg}+H^{a-2;\infty}_b$ with $\mathcal{E}=\overline{\{(2,0)\}}\overline{\cup}(\mathcal{E}_f-2)$ to 
		\begin{equation}
			\begin{gathered}
				(\Delta+V)u=f.
			\end{gathered}
		\end{equation}
	\end{lemma}
	\begin{proof}
		Using the orthogonality to the kernel, by elliptic theory, we get a solution $u\in L^2$. This can be improved to $u\in H^{1-;\infty}_b$ solution as in \cite{hintz_lectures_2023} Theorem 3.1. To obtain polyhomogeneity, we simply treat $V u$ as error term. In particular, we have that $-V u+f\in\mathcal{A}^{\mathcal{E}_f}_{phg}+H^{\tilde{a};\infty}_b$ for $\tilde{a}=\min(a,1+4-)$. We can solve $\Delta u=f-Vu$ to obtain $u\in\mathcal{A}^{\mathcal{E}}_{phg}+H^{\tilde{a}-2;\infty}$ via Proposition 3.4 in \cite{hintz_lectures_2023}. Iterating the above procedure, we may peel all polyhomogeneous terms up to the error $H^{a;\infty}_b$. The reason that the leading order decay of $u$ starts at $r^{-2}$ and not $r^{-1}$ follows by subtracting the kernel element $\Lambda W$ exactly to cancel the slower term.
	\end{proof}
	
	Next, let us note some easy nonlinear properties
	\begin{lemma}
		For $\phi\in\mathcal{A}_{phg}^{\vec{\mathcal{E}}}(\tilde{\mathcal{R}}),g\in\mathcal{A}_{phg}^{\vec{\mathcal{E}}'}(\tilde{\mathcal{R}})$, we have $\mathcal{N}[\phi]\in\mathcal{A}^{\vec{\mathcal{E}}^\mathcal{N}}(\tilde{R})$ and $\mathcal{N}[\phi]-\mathcal{N}[\phi+g]\in\mathcal{A}^{\vec{\mathcal{E}}^g}(\tilde{R})$ with 
		\begin{equation}
			\begin{gathered}
				\mathcal{E}^\mathcal{N}_F:=\sum_{i=2}^5i\mathcal{E}_F,\quad\mathcal{E}^\mathcal{N}_+=\sum_{i=2}^5i\mathcal{E}_++\overline{(5-i,0)}\\
				\mathcal{E}^g_F=\sum_{\substack{i+j+k=5\\
				i\geq1}} i\mathcal{E}^g_F+j\mathcal{E}_F,\quad \mathcal{E}^g_F=\sum_{\substack{i+j+k=5\\
				i\geq1}} i\mathcal{E}^g_F+j\mathcal{E}_F+\overline{(k,0)}
			\end{gathered}
		\end{equation}
	\end{lemma}

	\begin{theorem}\label{polyhomogeneity theorem}
		Let $\psi^{dat}\tau_2$ be as in \cref{main theorem}, with the extra requirement that there exists and index set $\mathcal{E}^{dat}$ such that the following holds
		\begin{equation}
			\begin{gathered}
				\sum_{\abs{\alpha}=k}\int_{\scri}\frac{d u}{\jpns{u}} \Big(\jpns{u}^{q}\Gamma^\alpha\psi^{dat}\Big)^2<\infty,\quad \Gamma\in\{\partial_\omega,\jpns{u}\partial_u\}, \forall k\geq0,\quad \iff \psi^{dat}\in H^{q,\infty}_b(\scri),\\
				\Big(\prod_{(k,\sigma)\in\mathcal{E}_c}(u\partial_u+k)\Big)\psi^{dat}\in H^{c;\infty}_b,\quad \iff \psi^{dat}\in\mathcal{A}_{phg}^{\mathcal{E}^{dat}}.
			\end{gathered}
		\end{equation}
		Then, the corresponding scattering solution to \eqref{main equation} satisfies $\psi\in\mathcal{A}^{\mathcal{E}_\scri,\mathcal{E}_+,\mathcal{E}_\mathcal{F}}(\tilde{\mathcal{R}}_{-\infty,\tau_2})$ for some index sets. In particular we have, $\min(\mathcal{E}_\scri)=(1,0)$ and $\min(\mathcal{E}_1,\mathcal{E}_0)>3$.
	\end{theorem}
	
	\begin{proof}
		\textit{Step 1: Polyhomogeneity at $\scri$}:
		
		This has to be done by revisiting the construction and proving estimates for $(r\partial_v+k)$ commuted quantities. Let us write $f=-V\psi+\mathcal{N}[\psi]$ for the rest of this step. Let us first note the following commutation
		\begin{claim}
			Let $\psi$ solve $(\partial_u\partial_v-\frac{\slashed{\Delta}}{r^2})r\psi=f$. Then
			\begin{equation}
				\begin{gathered}
				\underbrace{	(\partial_u\partial_v+\frac{2j+2}{r}\partial_t-\frac{\slashed{\Delta}}{r^2})}_{\mathcal{L}_j}\tilde{\Psi}_j=\tilde{f}_j\\
					\tilde{\Psi}_n=\prod_{k=0}^{n}(r\partial_v+k)r\psi,\quad f_n=\prod_{k=0}^n(r\partial_v+2+k)f.
				\end{gathered}
			\end{equation}
		\end{claim}
		\begin{proof}[Proof of claim]
			This follows by induction. Setting $\tilde{\Psi}_{-1}=r\psi, f_{-1}=f$ we have the starting case. For the induction step, we compute
			\begin{equation}
				\begin{gathered}
					r\partial_v\tilde{f}_j=r\partial_v\mathcal{L}_j\tilde{\Psi}_j=\mathcal{L}_jr\partial_v\tilde{\Psi}_j+(-2\frac{\slashed{\Delta}}{r^2}-\partial_u\partial_v+\partial_v^2-\frac{2j}{r}\partial_t)\tilde{\Psi}_j\\
					=\mathcal{L}_{j}r\partial_v\tilde{\Psi}_j -2\mathcal{L}_j\tilde{\Psi}_j+\frac{2}{r}\partial_tr\partial_v\tilde{\Psi}_j-\frac{2j}{r}\partial_t\tilde{\Psi}_j=\mathcal{L}_{j+1}r\partial_v\tilde{\Psi}_j-2\tilde{f}_j+\frac{2j}{r}\partial_t\tilde{\Psi}_j\\
					=\mathcal{L}_{j+1}(r\partial_v\tilde{\Psi}_j+j\tilde{\Psi}_j)-(2+j)\tilde{f}_j\\
					\implies \mathcal{L}_{j+1}\tilde{\Psi}_{j+1}=(r\partial_v+2+j)\tilde{f}_j.
				\end{gathered}
			\end{equation}
		\end{proof}
		
		Similarly, we can compute the commutations by replacing $r\partial_v$ by $\chi r\partial_v$ for $\chi=\tilde{\chi}^c(10\frac{r}{t})$. We get
		\begin{equation}
			\begin{gathered}
				\mathcal{L}_j\Psi_j=f_j+r^{-2}\mathfrak{D}_{\supp\chi'}^{k+1}r\phi\\
				{\Psi}_n=\prod_{k=0}^{n}(\chi r\partial_v+k)r\psi,\quad f_n=\prod_{k=0}^n(\chi r\partial_v+2+k)f
			\end{gathered}
		\end{equation}
		where $\mathfrak{D}_{\supp\chi'}^n$ denotes a sum of up to $n$ fold $b$ derivatives with $\mathcal{O}^{\infty,0,\infty}$ coefficients supported in the region $\supp\chi'$.
		
		We will use the boundedness statement derived in $H_b$ spaces about $\psi$ to control $\Psi_j$ in better function spaces. In \cref{unstable lemma}, we derived that there is a $\psi\in H_b^{q,q,1/2;N}(\mathcal{R}_{\tau_1,\tau_2})$ scattering solution in $\mathcal{R}_{\tau_1,\tau_2}$, where the inclusion is uniform in $\tau_2$.  By definition, we immediately get $\Psi_j\in H^{\infty,q-1,-1/2;N'}$ for some $N'(N,j)$ that satisfies $N'\to\infty$ as $N\to\infty$. A similar finite loss of derivatives applies for each estimate of $\Psi_j$ below. Since we only care about data in smooth class, we are going to ignore the regularity index.  We may write the equation for $\Psi_j$ as
		\begin{equation}\label{eq:wave_eq_Psi_j}
			\begin{gathered}
				(\partial_u\partial_v-\frac{2j}{r}\partial_r-\frac{\slashed{\Delta}}{r^2})\Psi_j=f_j+r^{-2}\mathfrak{D}_{\supp\chi'}^{k+1}r\phi-\frac{2j+2}{r}\partial_v\Psi_j.
			\end{gathered}
		\end{equation}
		\begin{claim}
			For $\Psi_j$ as above, we have for some 
			\begin{equation}
				\begin{gathered}
					\Psi_j\in H_b^{\infty,q-1,j+1/2}.
				\end{gathered}
			\end{equation}
		\end{claim}
		\begin{proof}[Proof of claim]
			Let us note, that the claim is equivalent to $\chi r\psi\in\mathcal{A}_{b,b,phg}^{\infty,q-1,\overline{(0,0)}}+H^{\infty,q-1,j+1/2}_b$. 
			
			We proceed by induction. For $\Psi_0$, we note that $\chi r\psi\in H^{\infty,q-1,-1/2}_b$ implies right hand side of \cref{eq:wave_eq_Psi_j} for $j=0$ is in $H^{\infty,q+1,3/2;\infty}_b$, while the left hand side is simply $\Box_{\R^{3+1}}\Psi_0$. A usual $T$ energy estimate implies $\Psi_0\in H^{q-1,q-1,1/2;0}_b$. Commuting with $S,\Omega,\mathcal{B}$ yield higher order estimates. This establishes the result for $\Psi_0$.
			
			To proceed with the inductive step, we note that
			\begin{equation}\label{eq:commutation_on_poly_spaces}
				\begin{gathered}
					\prod_{k=0}^{n}(\chi r\partial_v+2+k)\mathcal{A}_{b,b,phg}^{\infty,q-1,\overline{(5,0)}}=\mathcal{A}_{b,b,phg}^{\infty,q-1,\overline{(n+3,0)}}.
				\end{gathered}
			\end{equation}
			Using this, we get that $r\psi\in\mathcal{A}_{b,b,phg}^{\infty,q-1,\overline{(0,0)}}+H^{\infty,q-1,j-1/2}_b$ implies
			\begin{equation}
				\begin{gathered}
					f_j\in H_b^{\infty,4+q-1,j+5/2}.
				\end{gathered}
			\end{equation}
			Therefore, the right hand side of \cref{eq:wave_eq_Psi_j} for $j\neq0$ is in $H^{\infty,q+1,j+3/2}_B$, while the left hand side is $\Box_{\R^{(3+2j)+1}}\Psi_j$. A $T$ energy estimate now yields $\Psi_j\in H^{q-1,q-1,j+1/2;0}$.
		\end{proof}
		
		We can conclude that for $N(j)$ sufficiently large, the solution constructed in \cref{unstable lemma} $\psi\in H^{q,q,1/2;N}$ satisfies $\Psi_j\in H^{\infty,q,j+1/2;N/2}$ uniformly in $\tau_2$. Including $\Psi_j$ in the norm in which we use in finding the subsequence yields that the global scattering solution will also have good conormal property at $\scri$.
		
		For the rest of the proof, we are going to ignore and drop from notation the index set arising at $\scri$, since we already know polyhomogeneity there. More precisely, we are going to construct an ansatz $\phi^a\in\mathcal{A}_{phg}^{\vec{\mathcal{E}}}$, such that 
		\begin{equation}
			\begin{gathered}
				\Box\phi^a+(\phi^a)^5=\mathcal{O}^{N,N,5},\qquad 
				(r\phi^a)|_{\scri}-\psi^{dat}=\mathcal{O}^N({\scri\times S^2})
			\end{gathered}
		\end{equation}
		for arbitrary high value of $N$. Note, that $\phi^a$, and thus $\Box\phi^a+(\phi^a)^5$ may contain logarithmic terms towards $\scri$. When solving $\Box(\phi+\phi^a)+(\phi^a+\phi)^5=0$ for the nonlinear profile, we will no longer have a clear index set as in \cref{eq:commutation_on_poly_spaces}. To overcome this, let us compute
		\begin{equation}
			\begin{gathered}
				[v\partial_v,v\Big(\partial_u\partial_v-\frac{j}{r}(\partial_v-\partial_u)-\frac{\slashed{\Delta}}{r^2}\Big)]=\mathcal{O}^{(1,0),(1,0),(1,0)}\mathfrak{D}^2.
			\end{gathered}
		\end{equation}
		From this, it follows that 
		\begin{equation}
			\begin{gathered}
				\Box_{\R^{(3+2)+1}}(\chi v\partial_v+\sigma)\Psi_j=(v\partial_v+\sigma+1)\Box_{\R^{(3+2)+1}}\Psi_j+\mathcal{O}^{(2,0),(2,0),(2,0)}\mathfrak{D}^2\Psi_n
			\end{gathered}
		\end{equation}
		Commuting with sufficiently many choices of $\sigma$, we may improve the first term on the right hand side to be in $H_b^{\infty,q-1,j+5/2}$. A $T$ energy estimate yields that
		\begin{equation}
			\begin{gathered}
				\prod_{\sigma_j}(\chi v\partial_v+\sigma_j)\Psi_j\in H^{\infty,q-1,j+1/2}_b.
			\end{gathered}
		\end{equation}
		
		\textit{Step 2:} We first set up notation to solve error terms on $F,I_+$ iteratively. Let's write $\phi=\phi^n_a+\phi^n_e$ for some ansatz $\phi_a^n\in\mathcal{A}^{\vec{\mathcal{E}}}$ with $\phi_a^0=0$ and error terms . We will iteratively upgrade our knowledge about $\phi^n_a$ and make $\phi^n_e$ faster decaying. We similarly write $\psi^{dat}=\psi^{dat}_s+\psi^{dat}_a$ (and $\mathcal{N}[\phi]=\mathcal{N}[\phi_a^n]+(\mathcal{N}[\phi]-\mathcal{N}[\phi_a^n])=f^n_s+f^n_a+f^n_e$) with $\psi^{dat}_s$  ($f_s^n$) representing the part of the data (inhomogeneity) already solved for and $\psi^{dat}_e\in\mathcal{A}^{\mathcal{E}^{dat;n}}$  ($f_a^n\in\mathcal{A}^{\vec{\mathcal{E}}^{f;n}}$) the part decaying faster . Let's introduce the leading order term that we solve for $(p^n_+,k_+^n):=\min(\mathcal{E}^{dat;n}+1,\mathcal{E}_{+}^{f;n}-2)$ and $(p^n_F;k^n_F)=\min(\mathcal{E}^{f;n}_F)$. 
		
	 	\textit{Case 1:} First, let us consider the case $p^n_F\geq p^n_+$. Let $g$ solve
	 	\begin{equation}
	 		\begin{gathered}
	 			N_{p_+^n-1}g=P^+_{p_+^n+2,k_+^n}f_a^n\\
	 			((1-\rho)^{p_+^n-1} g)|_{\partial B}=P_{p_+^n-1,k_+^n}\psi^{dat;n}.
	 		\end{gathered}
	 	\end{equation}
		We use $g$ as an ansatz in the region $r>R$ which is captured by the cutoff $\chi=1-\tilde{\chi}(r/R)$. By \cref{lemma: projection operators,model operator on i_+ punctured} we get that
		\begin{equation}
			\begin{gathered}
				g\in \mathcal{A}^{(\mathcal{E}^{f;n}_F-(p^n_+,0))\overline{\cup}\overline{(0,0)}}(\dot{B}),
			\end{gathered}
		\end{equation}
		and thus $g_2=\chi t^{-p_+^n} \log^{k_+^n}(\rho_+)g(\frac{x}{t})\in\mathcal{A}^{\mathcal{E}^{f;n}_F\overline{\cup}\overline{(p_+^n,0)},\overline{(p_+^n,k_+^n)}}$. For a moment, let's focus on the case $k_+^n=0$. In this case, we simply compute
		\begin{equation}
			\begin{gathered}
				(\Box+V)g_2=(\Box+V)\chi t^{-p_+^n+1}t^{-1}g(\frac{x}{t})=\chi t^{-p-2} N_{p_+^n-1}g+\partial\chi\cdot \partial g_2+g_2\Box\chi+Vg_2\\=t^{-p_+^n-2}\chi P^+_{p_+^n+2,0}f_a^n+\mathcal{O}^{\mathcal{E}_F^{f;n}\cup(p_+^n,0),{(p^n_++4,0)}}.
			\end{gathered}
		\end{equation}
		For $k_+^n\neq 0$, we only improve on the log terms at $I_+$. We compute
		\begin{equation}
			\begin{gathered}
				(\Box+V)g_2=t^{-p_+^n-2}\log^{k_+^n}(\rho_+)\chi P^+_{p_+^n+2,k_+^n}f_a^n+\mathfrak{Err}_1,\quad 
				\mathfrak{Err}_1\in \mathcal{O}^{\mathcal{E}^{f;n}_{F}\overline{\cup}\overline{(p_+^n,0)},\overline{(p_+^n+2,k_+^n-1)}}\\\lim _{u\to\infty}u^{p_+^n-1}\log^{k^n_+}(u)(tg_2)|_{\scri}=P_{p_+^n-1,k_+^n}\psi^{dat;n}.
			\end{gathered}
		\end{equation}
		We can update our ansatz as
		\begin{equation}
			\begin{gathered}
				\phi^{n+1}_a:=\phi^n_a+g_2\\
				 \psi^{dat;n+1}_a:=(1-P_{p_+^n-1,k_+^n})\psi^{dat;n}_a\in\mathcal{A}^{\mathcal{E}^{dat;n+1}}\\
				 f^{n+1}_a:=f^{n}_a-\chi t^{-p_+^n}\log^{k_+^n}(\rho_+)P^+_{p_+^n,k_+^n}f_a^n+(\mathcal{N}[\phi^{n}_a]-\mathcal{N}[\phi^{n+1}_a])+\mathfrak{Err}_1.
			\end{gathered}
		\end{equation}
		Importantly, we find that $\mathcal{E}^{f;n+1}_+=\min(\mathcal{E}_+^{f;,n})\cup \mathcal{E}'$ and $\mathcal{E}^{f;n+1}_F=\mathcal{E}^{f;n}_F\cup \mathcal{E}''$ with $\min(\mathcal{E}')\geq p_++3$ and $\min(\mathcal{E}'')\geq p_+$.
		
		\textit{Case 2:} Now, consider the case $p_F^n<p_+^n$. In this case, we solve a model problem on $F$. However, $P^F_{p_F^n,k^n_F}f_a^n$ might not be orthogonal to the kernel of $(\Delta+V)$, so we need to correct for this close to the face $F$. Let us  calculate for $\Lambda W\in\ker(\Delta+V)$
		\begin{equation}
			\begin{gathered}
				(\Box+V)t^{-p_F^n+1}\log^{k_F^n}(t) \Lambda W=(p_F^n-1)p_F^nt^{-p_F^n-2}\log^{k_F^n}(t)\cdot t \Lambda W+\mathcal{O}^{\overline{(p_F^n+1,k_+^n-1)},\overline{(p_F^n+2,k_+^n-1)}}.
			\end{gathered}
		\end{equation}
		We remove the leading order error term by solving
		\begin{equation}
			\begin{gathered}
				N_{p_F^n-1}g_1=P^F_0(t\Lambda W)|_{I_+}=(\frac{t}{r})|_{I_+}\\
				((1-\rho)^{p_F^n-1} g_1)|_{\partial B}=0.
			\end{gathered}
		\end{equation}
		By \cref{model operator on i_+ punctured} we get that $g_1\in\mathcal{A}^{\overline{(0,0)},\overline{(0,0)}}(B)$ and $g_1(0)\neq0$. We can't extend $g_1$ to a smooth function on $\mathcal{R}$, because it has a Taylor expansion at $r=0$ of the form $g_1\sim c_1+c_2\abs{y}+\mathcal{O}(\abs{y}^2)$ with $c_1,c_2\in\R$ and potentially $c_2\neq0$. In order to deal with this, we write
		\begin{equation}\label{eq:g_2 LambdaW}
			\begin{gathered}
				g_2(t,x)=t^{-p_F^n}\log^{k_F^n}(t)\Big(\chi g_1(\frac{x}{t})+(1-\chi)c_1\Big),\quad \chi=1-\tilde{\chi}(r/R)\\
				\implies g_2\in\mathcal{A}^{\overline{(p_F^n,k_F^n)},\overline{(p_F^n,k_F^n)}},\quad g_2(t,x)-t^{-p_F^n}\log^{k_F^n}(t)g_1(0)\in\mathcal{A}^{\overline{(p_F^n+1,k_F^n)},\overline{(p_F^n,k_F^n)}}.
			\end{gathered}
		\end{equation}
		For $k_F^n=0$ we simply get
		\begin{equation}
			\begin{gathered}
					(\Box+V)g_2=(\Box+V)t^{-p_F^n}\big((\chi g_1(\frac{x}{t})+(1-\chi)c_1\big)=
				\chi t^{-p_F^n-2}N_{p_F^n-1}g_1+c_1 V\\
				+V(g_2(t,x)-t^{-p_F^n}\log^{k_F^n}(t)g_1(0))+\partial\chi\cdot t^{-p_F^n}(g_2(t,x)-t^{-p_F^n}\log^{k_F^n}(t)g_1(0))\\+(g_2(t,x)-t^{-p_F^n}\log^{k_F^n}(t)g_1(0))\Box\chi=\chi t^{-p_F^n-2}P_0^F(t\Lambda W)+t^{-p_F^n}c_1V+\mathcal{O}^{\overline{(p_F^n+1,0)},\overline{(p_F^n+4,0)}}.
			\end{gathered}
		\end{equation}
		For $k_F^n\neq0$, we get error terms from the derivatives acting on $\log^{k_F^n}(t)$. A  simple computation yields 
		\begin{equation}
			\begin{gathered}
				\begin{multlined}
					(\Box+V)g_2=\log^{k_F^n}(t)t^{-p_F^n}\Big(t^{-1}\Lambda W\chi+Vg_1(0)\Big)\\
					+\mathcal{O}^{\overline{(p_F^n+1,k_F^n-1)},\overline{(p_F^n+2,k_F^n-1)}}+\mathcal{O}^{\overline{(p_F^n+1,k_F^n)},\overline{(p_F^n+4,k_F^n)}}
				\end{multlined}
			\end{gathered}
		\end{equation}
		
		The importance of this calculation is the following. Solving away an error term on $I_0$ with decay $p_+^n+2$ generally produces a term with decay $p_+^n$. When $(\Box+V)$ acts on this near $F$, we can expect a term that will also decay at the rate $p_+^n$. In this calculation, we capture this leading order term. We may use this to write
		\begin{equation}
			\begin{gathered}
				\begin{multlined}
					(\Box+V)\big((p_F^n-1)p_F^ng_2-t^{-p_F^n}\log^{k_F^n}(t)t\Lambda W\big)=(p_F^n-1)p_F^nt^{-p_F^n}\log^{k_F^n}(t)g_1(0)V\\+\mathcal{O}^{\overline{(p_F^n+1,k_F^n)},\overline{(p_F^n+2,k_F^n-1)}}+\mathcal{O}^{\overline{(p_F^n+1,k_F^n)},\overline{(p_F^n+4,k_F^n)}}
				\end{multlined}
			\end{gathered}
		\end{equation}
		The second of the error terms is admissible, because it falls of faster than $p_F^n+2$ at $I_+$, but we must solve away the first one. Repeating the same procedure as for $(t\Lambda W)$ term on $I_+$, given an error term in $\mathcal{O}^{\overline{(p_F^n+1,k_F^n)},\overline{(p_F^n+2,k_F^n-1)}}$, we can find $\tilde{g}_2\in\mathcal{A}_{phg}^{\overline{(p_F^n,k_F^n-1)},\overline{(p_F^n,k_F^n-1)}}$ that cancels it with a remaining error term from the $V$ linear piece in $\mathcal{O}^{\overline{(p_F^n,k_F^n-1)},\overline{(p_F^n+4,k_F^n)}}$. In particular, there exists $\tilde{g}_2$ such that
				\begin{equation}
			\begin{gathered}
				\begin{multlined}
					(\Box+V)\big(\underbrace{(p_F^n-1)p_F^n\tilde{g}_2-t^{-p_F^n}\log^{k_F^n}(t)t\Lambda W}_{g_3}\big)=(p_F^n-1)p_F^nt^{-p_F^n}\log^{k_F^n}(t)g_1(0)V\\+\mathcal{O}^{\overline{(p_F^n,k_F^n-1)},\overline{(p_F^n+4,k_F^n)}}+\mathcal{O}^{\overline{(p_F^n+1,k_F^n)},\overline{(p_F^n+4,k_F^n)}}
				\end{multlined}
			\end{gathered}
		\end{equation}

		A similar result also holds for the other kernel elements $\{\partial_i W\}$, with the sole exception, that we need to use $t^2\partial_iW$ instead $t\Lambda W$. For the sake of completeness, we repeat the argument
		\begin{equation}
		\begin{gathered}
			(\Box+V)t^{-p_F^n+2}\log^{k_F^n}(t) \partial_iW=(p_F^n-2)(p_F^n-1)t^{-p_F^n-2}\log^{k_F^n}(t)\cdot t^2\partial_iW+\mathcal{O}^{\overline{(p_F^n+2,k_+^n-1)},\overline{(p_F^n,k_+^n-1)}}.
		\end{gathered}
	\end{equation}
	We remove the leading order error term by solving
	\begin{equation}
		\begin{gathered}
			N_{p_F^n}g_{1;i}=P^F_0(t^2\partial_iW)|_{I_+}\\
			((1-v)^{p_F^n}g_{1;i})|_{\partial B}=0.
		\end{gathered}
	\end{equation}
	
	By \cref{model operator on i_+ punctured} we get that $g_{1;i}\in\mathcal{A}^{\overline{(0,0)},\overline{(0,0)}}(\dot{B})$ and $g_{1;i}(0)\neq0$  where $g_{1;i}(0)=c\hat{x}_i\in\mathcal{C}^\infty(S^2)$ is an $l=1$ spherical mode. We set set 
	
	\begin{equation}
		\begin{gathered}
			g_{2;i}=t^{-p_F^n}\log^{k^n_F}(t)\chi g_{1;i},\quad \chi=\bar{\chi}^c(r/R)
		\end{gathered}
	\end{equation}
	for $R\gg 1$. As before, we conclude
	\begin{equation}
		\begin{gathered}
			(\Box+V)g_{2;i}=\log^{k_F^n}t^{-p_F^n}\Big(\chi\partial_iW+Vg_{1;i}(0)\Big)\\
			+\mathcal{A}^{\overline{(p_F^n+1,k_F^n-1)},\overline{(p_F^n+2,k_F^n-1)}}+\mathcal{A}^{\overline{(p_F^n+1,k_F^n)},\overline{(p_F^n+4,k_F^n)}}
			\\
			\begin{multlined}
				\implies (\Box+V)\big(\underbrace{(p_F^n-2)(p_F^n-1)\tilde{g}_{2;i}-t^{-p_F^n}\log^{k_F^n}(t)t^2\partial_iW}_{g_{3;i}}\big)=(p_F^n-2)(p_F^n-1)t^{-p_F^n}\log^{k_F^n}(t)(\chi-1)\partial_i W\\ +(p_F^n-2)(p_F^n-1)t^{-p_F^n}\log^{k_F^n}(t)g_{1;i}(0)V+\mathcal{A}^{\overline{(p_F^n,k_F^n-1)},\overline{(p_F^n+4,k_F^n)}}+\mathcal{A}^{\overline{(p_F^n+1,k_F^n)},\overline{(p_F^n+4,k_F^n)}}
			\end{multlined}
		\end{gathered}
	\end{equation}
	Taking $R$ sufficient large in the cutoff $\chi$, we get that 
	\begin{equation}
		\begin{gathered}
			((1-\chi)\partial_i W+\chi g_{1;i}(0)V,\partial_iW)_{L^2}\neq 0.
		\end{gathered}
	\end{equation}	
	Therefore, we can use $g_{2},g_{2;i}$ terms to project out the kernel elements. That is, there exists $a_\bullet$ such that
	\begin{equation}
		\begin{gathered}
			P^F_{p_F^n,k^n_F}f_a^n+a_\Lambda g_{1;i}(0)V+a_i\big(g_1(0) V+(\chi-1)\partial_i W\big)\perp\{\Lambda W,\partial_i W\}.
		\end{gathered}
	\end{equation}	
	Using this orthogonality, we can finally invert the normal operator to conclude that there exist a solution to 
	\begin{equation}
		\begin{gathered}
			(\Delta+V)g_4=	P^F_{p_F^n,k^n_F}f_a^n+a_i \big(g_{1;i}(0) V+(\chi-1)\partial_i W\big)+a_\Lambda g_1 V\\
			g_4\to0 \text{ as } r\to\infty,\quad g_4\in\mathcal{A}^{\big((\mathcal{E}_+^{f;n}-(p_F^n+2,0))\cup\overline{\{(2,0)\}}\big)\overline{\cup}(1,0)}.
		\end{gathered}
	\end{equation}
	Thus, we may compute
	\begin{equation}
		\begin{gathered}
			(\Box+V)(t^{-p_F^n}\log^{k_F^n}(t)g_4-a_\Lambda \tilde{g}_{3}-a_{i}\tilde{g}_{3;i})=t^{-p_F^n}\log^{k_F^n}(t)P^F_{p_F^n,k^n_F}f_a^n  +\mathcal{O}^{\overline{(p_F^n,k_F^n-1)},\overline{(p_F^n+2,k_F^n-1)}}
			\\+\mathcal{O}^{\overline{\{(p_F^n+2,k_F^n)\}},\big((\mathcal{E}_+^{f;n})\cup\overline{\{(4+p_F^n,k_F^n)\}}\big)\overline{\cup}(3+p_F^n,k_F^n)}+\mathcal{O}^{\overline{(p_F^n+1,k_F^n)},\overline{(p_F^n+4,k_F^n)}}
		\end{gathered}
	\end{equation}
	Therefore, we can update the ansatz as follows:
	\begin{equation}
		\begin{gathered}
			\phi^{n+1}_a=\phi^n_a-\Big(t^{-p_F^n}\log^{k_F^n}(t)g_4-a_\Lambda \tilde{g}_{3}-a_{i}\tilde{g}_{3;i}\Big),\\
			f^{n+1}_a=(1-t^{-p_F^n}\log^{k_F^n}(t)P^F_{p_F^n,k^n_F})f^n_a+(\mathcal{N}[\phi_a^{n+1}]-\mathcal{N}[\phi_a^n])+\mathcal{A}^{\overline{(p_F^n,k_F^n-1)},\overline{(p_F^n+2,k_F^n-1)}}
			\\+\mathcal{A}^{\overline{\{(p_F^n+2,k_F^n)\}},\big((\mathcal{E}_+^{f;n})\cup\overline{\{(4+p_F^n,k_F^n)\}}\big)\overline{\cup}(3+p_F^n,k_F^n)}+\mathcal{A}^{\overline{(p_F^n+1,k_F^n)},\overline{(p_F^n+4,k_F^n)}}.			
			\end{gathered}
	\end{equation}
	In particular, the new error term satisfies $\mathcal{E}^{f;n+1}_F=\mathcal{E}^{f;n}_F\cup\mathcal{E}'$ and $\mathcal{E}^{f;n+1}_+=\min(\mathcal{E}^{f;n}_+)\cup\mathcal{E}''$ with $\min(\mathcal{E}')\geq p_F+1,\min(\mathcal{E}'')\geq p_F+2$.

	\textit{Step 3:}	The above cases are exhausting iteratively generated index sets and therefore, the polyhomogeneity follows. Let us expand.
	
	Without loss of generality start with $p_F^{n_0}\geq p_+^{n_0}$. By definition of the index set, $\mathcal{E}^{f;n_0}_+\cap\{(p,k):p< p_+^{n_0}+3\}$ is finite. Therefore, we can iterate \textit{case 1} a finite number of times until either $p_F^{n_1}<p_+^{n_1}$ and $p_F^{n_0}\leq p_+^{n_1}$ or $p_F^{n_1}\geq p_F^{n_0}+1$. If the latter holds, we repeat \textit{case 1}.
	
	If $p_F^{n_1}<p_+^{n_1}$, then we use again the definition of an index set, to get that $\mathcal{E}^{f;n_1}_F\cap\{(p,k):p<p_F^{n_1}+1\}$ is finite. Therefore, a finite number of iteration with \textit{case 2} yields an iterate with $p_F^{n_2}\geq p_+^{n_2}$ and $p_+^{n_2}\geq\min(p_+^{n_1},p_F^{n_1}+2)$ or $p_F^{n_2}\geq p_F^{n_1}+1$.
	
	Iterating the above two, we can in a finite number of steps improve by 1 at least one of the index sets. The other must be at least similarly fast decaying. This implies that the procedure is exhaustive.
	
	\end{proof}

	\section{Construction of Cauchy data}\label{sec:cauchy}
	In this section, we extend the solution that we constructed to $i_0$ and thereby to a Cauchy surface of the form $t=\textit{const}$. We proceed by two steps. First, we prove that we can add in arbitrary size scattering data in the interior of $\scri$. This is detailed in \cref{lemma:local_cauchy} Second, we prove that provided no incoming radiation from $\scri$ in neighbourhood of $i_0$ we can construct a solution in this neighbourhood as well. This is proved in \cref{lemma:global_cauchy}.
	
	\begin{lemma}\label{lemma:local_cauchy}
		Let $\psi^\Sigma,(r\psi)^\scri$ be a scattering data for \cref{main equation} of order $N>6$ and size 1 on $\Sigma_{\tau_2},\scri_{\tau_1,\tau_2}$. There exists $d$ sufficiently large such that for any $v_\infty>d$ there is a unique scattering solution to \cref{main equation} with scattering data $\psi^\Sigma,(r\psi)^\scri$ in the region $\mathcal{R}_{\tau_1,\tau_2}\cap\{v>v_\infty\}$.
	\end{lemma}
	
	\begin{proof}
		Let us begin by noticing that in the domain of interest  $\mathcal{R}_{\tau_1,\tau_2}\cap\{v>v_\infty\}$, $r\gtrsim_{\tau_1,\tau_2}\abs{u}$. We will use the $r$ weights in the equations as a replacement for proving any amount of decay. Indeed, no ILED estimates are necessary.
		
		Without loss of generality, take $d$ sufficiently large such that $\Sigma_{\tau}\cap\{v>d\}$ is null for $\tau\in(\tau_1,\tau_2)$. After shifting the definition of $t_\star$ by a constant, we write these cones as $\mathfrak{C}_{\tau,d}=\{u=\tau\}\cap\{v>d\}$. Let us construct \textit{master currents} for this region explicitly. For the rest of the proof, let us redefine $\masterJ$ as follows
		\begin{equation}
			\begin{gathered}
				\masterJ=\sum_{\abs{\alpha}\leq k} J[\Gamma^\alpha\psi],\quad\Gamma\in\{T,S,\Omega\},\quad J[\psi]=T\cdot(\T^0+\tilde{\T}^0)[\psi].
			\end{gathered}
		\end{equation}
		Because we have the usual commutation relations $[T,\Box]=[\Omega,\Box]=[S,r^2\Box]=0$ we get that for $\tilde{\phi}=W+\phi$ a scattering solution to \cref{main equation}
		\begin{equation}
			\begin{gathered}
				\abs{\Box\Gamma^\alpha\phi}\lesssim_{\alpha}\sum_{\abs{\beta}\leq\abs{\alpha}}\jpns{r}^{-4}\Gamma^\beta\abs{\phi}+\Gamma^\beta\abs{(\Box+V)\phi}\lesssim\sum_{\substack{\abs{\beta}\leq\abs{\alpha}\\ 1\leq i\leq5}}(\Gamma^\beta\phi)^i\jpns{r}^{i-5}
			\end{gathered}
		\end{equation}
		Therefore, a $T$ energy estimate in the region $\overline{\mathcal{R}}_{\tau_1,\tau_2}:=\{u\in(\tau_1,\tau_2)\}\cap\{v>d\}$ yields
		\begin{equation}\label{eq:cauchy_error_loc}
			\begin{gathered}
				\mathfrak{C}_{\tau,d}[\masterJ]\lesssim_\alpha\mathfrak{C}_{\tau_2,d}[\masterJ]+\scri_{\tau,\tau_2}[\masterJ]+\sum_{\substack{\abs{\beta}\leq k\\ 1\leq i\leq5}}\int_{\overline{\mathcal{R}}_{\tau,\tau_2}}\jpns{r}^{2(5-i)}(\Gamma^\beta\phi^i)^2
			\end{gathered}
		\end{equation}
		Note, that \cref{weighted estimates} still implies that for $k\geq3$ 
		\begin{equation}
			\begin{gathered}
				(\mathfrak{C}_{\tau,d}[\masterJ])^{1/2}\gtrsim \sup_{(t,x)\in\mathfrak{C}_{\tau,d}}\jpns{r}^{1/2}\abs{\phi(t,x)}.
			\end{gathered}
		\end{equation}
		Hence, we can write \cref{eq:cauchy_error_loc} for $k\geq6$
		\begin{equation}\label{eq:cauchy_error_loc2}
			\begin{gathered}
				\mathfrak{C}_{\tau',d}[\masterJ]\lesssim_\alpha\mathfrak{C}_{\tau_2,d}[\masterJ]+\scri_{\tau,\tau_2}[\masterJ]+\sum_{\substack{\abs{\beta}\leq k\\ 1\leq i\leq5}}\int_{\overline{\mathcal{R}}_{\tau,\tau_2}}\jpns{r}^{-(5-i)}(\Gamma^\beta\phi)^2(\mathfrak{C}_{\tau,d}[\masterJ])^{2(i-1)}\jpns{r}^{-(i-1)}\\
				\lesssim\mathfrak{C}_{\tau_2,d}[\masterJ]+\scri_{\tau,\tau_2}[\masterJ]+\sum_{1\leq i\leq5}\int_{\tau_1}^{\tau'}d^{-7+i}(\mathfrak{C}_{\tau,d}[\masterJ])^{i}\\
				\leq \mathfrak{C}_{\tau_2,d}[\masterJ]+\scri_{\tau_1,\tau_2}[\masterJ]+d^{-2}\sum_{1\leq i\leq5}\int_{\tau_1}^{\tau'}(\mathfrak{C}_{\tau,d}[\masterJ])^{i}
			\end{gathered}
		\end{equation}
		Choosing $d$ and $C$ sufficiently large, this estimate allows us to close a bootstrap argument with $\mathfrak{C}_{\tau,d}[\masterJ]\leq C(\mathfrak{C}_{\tau_1,d}[\masterJ]+\scri_{\tau_1,\tau_2}[\masterJ])$ in the region $\tau\in(\tau_1,\tau_2)$ independent of the size of the two incoming energies.
	\end{proof}
	
	\begin{lemma}\label{lemma:global_cauchy}
		Let $\psi^\Sigma$ be a scattering data for \cref{main equation} of order $N>6$ and size 1 on $\Sigma_{\tau_2}$. There exists $d$ sufficiently large such that for any $V_\infty>d$ there is a unique scattering solution to \cref{main equation} with scattering data $\psi^\Sigma,(r\psi)^\scri=\chi(\tau_2-u)(r\psi)_{\tau_1}$ in the region $\{t_\star<\tau_2\}\cap\{v>v_\infty\}$.
	\end{lemma}
	
	\begin{proof}
		We repeat the proof of \cref{lemma:local_cauchy}. In particular, for any finite $\tau>\tau_1$, we can still use \cref{eq:cauchy_error_loc2} and write
		\begin{equation}
			\begin{gathered}
				\mathfrak{C}_{\tau_1,d}[\masterJ]\lesssim_\alpha\mathfrak{C}_{\tau_2,d}[\masterJ]+\cancel{\scri_{\tau_1,\tau_2}[\masterJ]}+\sum_{\substack{\abs{\beta}\leq k\\ 1\leq i\leq5}}\int_{\overline{\mathcal{R}}_{\tau_1,\tau_2}}\jpns{r}^{-(5-i)}(\Gamma^\beta\phi)^2(\mathfrak{C}_{\tau,d}[\masterJ])^{2(i-1)}\jpns{r}^{-(i-1)}
			\end{gathered}
		\end{equation}
		where the cancelled term vanishes by assumption of no incoming radiation. For the leftover term, we use that in the region of interest (${\mathcal{R}}_{-\infty,\tau_1}\cap\{v>v_\infty\}$) $r\gtrsim\jpns{v}$. Therefore, we can estimate the right hand side of the above equation as
		\begin{equation}
			\begin{gathered}
				\mathfrak{C}_{\tau_1,d}[\masterJ]\lesssim_\alpha\mathfrak{C}_{\tau_2,d}[\masterJ]+d^{-0.5}\sum_{1\leq i\leq5}\int_{\tau_1}^{\tau_2} dv\jpns{v}^{-1.5}(\mathfrak{C}_{v,d}[\masterJ])^{i}.
			\end{gathered}
		\end{equation}
		As before, choosing $d,C$ sufficiently large, we can simply close a bootstrap of the form $\mathfrak{C}_{\tau_1,d}[\masterJ]\lesssim C\mathfrak{C}_{\tau_2,d}[\masterJ]$ for $\tau_2>\tau_1$.
	\end{proof}
	
	\paragraph{\textbf{Acknowledgement:}} The author would like to thank Claude Warnick, Leonhard Kehrberger for  many helpful discussions on both the technical and conceptual parts of the work.
	The author would like to thank Peter Hintz for introducing him to geometric singular analysis and many helpful comments about the project.
	This project was founded by EPSRC.

	\appendix
	\section{Appendix}
	\subsection{Calculatioins}
	We note some easy computations 
	\begin{subequations}
		\begin{gather}
			W'=-\frac{W}{r}(1-W^2)\\
			\Lambda W=W(W^2-\frac{1}{2})
		\end{gather}
	\end{subequations}
	\subsection{Energy}\label{energy calculation}
	We compute the $T$ energy through a hypersurface $\tilde{\Sigma}=\{t=-h(x)\}$ for some $h\in\mathcal{C}^\infty$ with $\norm{\nabla h}\leq1$. As a simple application of this computation we will find \cref{energy integral}. Let us use coordinates $x,s=t+h(x)$ in a neighbourhood of $\tilde{\Sigma}$. We remark the following computations
	\begin{equation}
		\begin{gathered}
			\partial_i|_t=\partial_i|_s-h_iT,\quad ds=dt+\partial_ih\cdot dx_i.
		\end{gathered}
	\end{equation}
	The induced measure on $\tilde{\Sigma}$ with the above coordinates is $\det(\delta_{ij}-\partial_ih\partial_jh)^{1/2}=(1-\abs{\partial h}^2)^{1/2}$. Thus, we find the induced energy to be
	\begin{equation}
		\begin{gathered}
			\int_{\Sigma_\tau}\T^w[\phi]\bigg(\frac{ds}{\abs{ds}},dt\bigg)=\int_{\R^3}\frac{1}{2}\big((T\phi)^2+\partial_i|_t\phi\cdot\partial_i|_t\phi-w\phi^2\big)+\partial_ih\partial_i|_{t}\phi T\phi\\
			=\int_{\R^3}\frac{1}{2}\Big((1-\partial h\cdot\partial h)(T\phi)^2+\partial_i|_s\phi\cdot\partial_i|_s\phi-w\phi^2\Big)
		\end{gathered}
	\end{equation}
	
	We can similarly calculate the bilinear energy content, with the difference that we need to replace $\T^w[\phi](ds,dt)$ with
	\begin{equation}
		\begin{gathered}
			\T^w[\phi,\psi](dt,ds)=(1-\abs{\nabla h}^2)T\phi T\psi+\partial_i|_s\phi\cdot\partial_i|_s\psi-w\phi\psi.
		\end{gathered}
	\end{equation}

	\subsection{Standard estimates}
	
	\begin{lemma}(Weighted Sobolev)
		For $f\in\mathcal{C}^\infty_0(\R^3)$ we have
		\begin{equation*}
			\begin{gathered}
					\sup \jpns{r}^{1/2}f\lesssim \sum_{\abs{\alpha}\leq 3}\norm{\partial\Gamma^\alpha f}_{L^2},\quad\Gamma\in\{\partial,\Omega_{ij},r\partial_r\}.
			\end{gathered}
		\end{equation*}
	\end{lemma}

	\begin{lemma}(Hardy inequality)
		For $f\in\mathcal{C}^\infty_0(\R^3)$ we have
		\begin{equation*}
			\begin{gathered}
				\int f^2\lesssim \int r^2\abs{\nabla f}^2.
			\end{gathered}
		\end{equation*}
	\end{lemma}
	
	\printbibliography

@article{angelopoulos_asymptotics_2018,
  title = {Asymptotics for Scalar Perturbations from a Neighborhood of the Bifurcation Sphere},
  author = {Angelopoulos, Yannis and Aretakis, Stefanos and Gajic, Dejan},
  date = {2018-08-09},
  journaltitle = {Classical and Quantum Gravity},
  shortjournal = {Class. Quantum Grav.},
  volume = {35},
  number = {15},
  issn = {0264-9381, 1361-6382},
  doi = {10.1088/1361-6382/aacc1e},
  url = {http://arxiv.org/abs/1802.05692}
}

@article{angelopoulos_late-time_2018,
  title = {Late-Time Asymptotics for the Wave Equation on Spherically Symmetric, Stationary Spacetimes},
  author = {Angelopoulos, Y. and Aretakis, S. and Gajic, D.},
  date = {2018},
  journaltitle = {Advances in Mathematics},
  volume = {323},
  pages = {529--621},
  issn = {10902082},
  doi = {10.1016/j.aim.2017.10.027}
}

@article{baskin_asymptotics_2015,
  title = {Asymptotics of Radiation Fields in Asymptotically {{Minkowski}} Space},
  author = {Baskin, Dean and Vasy, András and Wunsch, Jared},
  date = {2015},
  journaltitle = {American Journal of Mathematics},
  shortjournal = {American Journal of Mathematics},
  volume = {137},
  number = {5},
  pages = {1293--1364},
  issn = {1080-6377},
  doi = {10.1353/ajm.2015.0033},
  url = {https://muse.jhu.edu/content/crossref/journals/american_journal_of_mathematics/v137/137.5.baskin.html}
}

@article{baskin_asymptotics_2018,
  title = {Asymptotics of Scalar Waves on Long-Range Asymptotically {{Minkowski}} Spaces},
  author = {Baskin, Dean and Vasy, András and Wunsch, Jared},
  date = {2018},
  journaltitle = {Advances in Mathematics},
  volume = {328},
  pages = {160--216},
  issn = {10902082},
  doi = {10.1016/j.aim.2018.01.012}
}

@unpublished{baskin_explicit_2016,
  title = {An Explicit Description of the Radiation Field in 3+1-Dimensions},
  author = {Baskin, Dean},
  date = {2016-04-11},
  eprint = {1604.02984},
  eprinttype = {arxiv},
  url = {http://arxiv.org/abs/1604.02984}
}

@article{beceanu_center-stable_2014,
  title = {A center-stable manifold for the energy-critical wave equation in {$\R^3$} in the symmetric setting},
  author = {Beceanu, Marius},
  date = {2014-09},
  journaltitle = {Journal of Hyperbolic Differential Equations},
  shortjournal = {J. Hyper. Differential Equations},
  volume = {11},
  number = {03},
  pages = {437-476},
  publisher = {World Scientific Publishing Co.},
  issn = {0219-8916},
  doi = {10.1142/S021989161450012X},
  url = {https://www.worldscientific.com/doi/abs/10.1142/S021989161450012X}
}

@book{christodoulou_formation_2009,
  title = {The {{Formation}} of {{Black Holes}} in {{General Relativity}}},
  author = {Christodoulou, Demetrios},
  date = {2009-01-10},
  series = {Monographs in {{Mathematics}}},
  publisher = {European Mathematical Soc}
}

@book{christodoulou_global_1993,
  title = {The {{Global Nonlinear Stability}} of the {{Minkowski Space}} ({{PMS-41}})},
  author = {Christodoulou, Demetrios and Klainerman, Sergiu},
  date = {1993},
  publisher = {Princeton University Press}
}

@article{dafermos_linear_2019,
  title = {The Linear Stability of the {{Schwarzschild}} Solution to Gravitational Perturbations},
  author = {Dafermos, Mihalis and Holzegel, Gustav and Rodnianski, Igor},
  date = {2019},
  journaltitle = {Acta Mathematica},
  volume = {222},
  number = {1},
  eprint = {1601.06467},
  eprinttype = {arxiv},
  pages = {1--214},
  issn = {00015962, 18712509},
  doi = {10.4310/ACTA.2019.v222.n1.a1},
  url = {http://arxiv.org/abs/1601.06467}
}

@incollection{dafermos_new_2010,
  title = {A New Physical-Space Approach to Decay for the Wave Equation with Applications to Black Hole Spacetimes},
  booktitle = {{{XVIth International Congress}} on {{Mathematical Physics}}},
  author = {Dafermos, Mihalis and Rodnianski, Igor},
  date = {2010-03},
  pages = {421--432},
  publisher = {World Scientific Publishing Company},
  doi = {10.1142/9789814304634_0032},
  url = {https://www.worldscientific.com/doi/abs/10.1142/9789814304634_0032},
  isbn = {978-981-4304-62-7}
}

@unpublished{dafermos_quasilinear_2022,
  title = {Quasilinear Wave Equations on Asymptotically Flat Spacetimes with Applications to {{Kerr}} Black Holes},
  author = {Dafermos, Mihalis and Holzegel, Gustav and Rodnianski, Igor and Taylor, Martin},
  date = {2022-12-28},
  eprint = {2212.14093},
  eprinttype = {arxiv},
  url = {http://arxiv.org/abs/2212.14093}
}

@unpublished{dafermos_scattering_2013,
  title = {A Scattering Theory Construction of Dynamical Vacuum Black Holes},
  author = {Dafermos, Mihalis and Holzegel, Gustav and Rodnianski, Igor},
  date = {2013},
  eprint = {1306.5364},
  eprinttype = {arxiv},
  url = {http://arxiv.org/abs/1306.5364}
}

@article{Dafermos2016,
  title = {Decay for Solutions of the Wave Equation on {{Kerr}} Exterior Spacetimes {{III}}: {{The}} Full Subextremal Case |a| {$<$} {{M}}},
  author = {Dafermos, Mihalis and Rodnianski, Igor and Shlapentokh-Rothman, Yakov},
  date = {2016},
  journaltitle = {Annals of Mathematics},
  volume = {183},
  number = {3},
  eprint = {1402.7034},
  eprinttype = {arxiv},
  pages = {787--913},
  issn = {19398980},
  doi = {10.4007/annals.2016.183.3.2}
}

@article{derrick_comments_1964,
  title = {Comments on {{Nonlinear Wave Equations}} as {{Models}} for {{Elementary Particles}}},
  author = {Derrick, G. H.},
  date = {1964-09},
  journaltitle = {Journal of Mathematical Physics},
  shortjournal = {Journal of Mathematical Physics},
  volume = {5},
  number = {9},
  pages = {1252--1254},
  issn = {0022-2488, 1089-7658},
  doi = {10.1063/1.1704233},
  url = {https://pubs.aip.org/aip/jmp/article/5/9/1252-1254/231422}
}

@unpublished{disconzi_recent_2023,
  title = {Recent Developments in Mathematical Aspects of Relativistic Fluids},
  author = {Disconzi, Marcelo M.},
  date = {2023-08-18},
  eprint = {2308.09844},
  eprinttype = {arxiv},
  url = {http://arxiv.org/abs/2308.09844}
}

@article{donninger_nonscattering_2013,
  title = {Nonscattering Solutions and Blowup at Infinity for the Critical Wave Equation},
  author = {Donninger, Roland and Krieger, Joachim},
  date = {2013-09-01},
  journaltitle = {Mathematische Annalen},
  shortjournal = {Math. Ann.},
  volume = {357},
  number = {1},
  pages = {89--163},
  issn = {1432-1807},
  doi = {10.1007/s00208-013-0898-1},
  url = {https://doi.org/10.1007/s00208-013-0898-1}
}

@article{duyckaerts_classification_2013,
  title = {Classification of the Radial Solutions of the Focusing, Energy-Critical Wave Equation},
  author = {Duyckaerts, Thomas and Kenig, Carlos and Merle, Frank},
  date = {2013},
  journaltitle = {Cambridge Journal of Mathematics},
  shortjournal = {Cambridge J. Math.},
  volume = {1},
  number = {1},
  pages = {75--144},
  publisher = {International Press of Boston},
  issn = {2168-0949},
  doi = {10.4310/CJM.2013.v1.n1.a3},
  url = {https://www.intlpress.com/site/pub/pages/journals/items/cjm/content/vols/0001/0001/a003/abstract.php}
}

@article{duyckaerts_dynamics_2008,
  title = {Dynamics of {{Threshold Solutions}} for {{Energy-Critical Wave Equation}}},
  author = {Duyckaerts, Thomas and Merle, Frank},
  date = {2008-01-01},
  journaltitle = {International Mathematics Research Papers},
  shortjournal = {International Mathematics Research Papers},
  volume = {2008},
  pages = {rpn002},
  issn = {1687-3017},
  doi = {10.1093/imrp/rpn002},
  url = {https://doi.org/10.1093/imrp/rpn002}
}

@article{duyckaerts_soliton_2022,
  title = {Soliton Resolution for Critical Co-Rotational Wave Maps and Radial Cubic Wave Equation},
  author = {Duyckaerts, Thomas and Kenig, Carlos and Martel, Yvan and Merle, Frank},
  date = {2022-04},
  journaltitle = {Communications in Mathematical Physics},
  shortjournal = {Commun. Math. Phys.},
  volume = {391},
  number = {2},
  eprint = {2103.01293},
  eprinttype = {arxiv},
  eprintclass = {math},
  pages = {779--871},
  issn = {0010-3616, 1432-0916},
  doi = {10.1007/s00220-022-04330-z},
  url = {http://arxiv.org/abs/2103.01293}
}

@article{duyckaerts_soliton_2023,
  title = {Soliton Resolution for the Radial Critical Wave Equation in All Odd Space Dimensions},
  author = {Duyckaerts, Thomas and Kenig, Carlos and Merle, Frank},
  date = {2023-03},
  journaltitle = {Acta Mathematica},
  shortjournal = {Acta Math.},
  volume = {230},
  number = {1},
  pages = {1--92},
  publisher = {International Press of Boston},
  issn = {1871-2509},
  doi = {10.4310/ACTA.2023.v230.n1.a1},
  url = {https://www.intlpress.com/site/pub/pages/journals/items/acta/content/vols/0230/0001/a001/abstract.php}
}

@article{duyckaerts_solutions_2016,
  title = {Solutions of the Focusing Nonradial Critical Wave Equation with the Compactness Property},
  author = {Duyckaerts, Thomas and Kenig, Carlos E. and Merle, Frank},
  date = {2016},
  journaltitle = {Annali della Scuola Normale Superiore di Pisa. Classe di scienze},
  volume = {15},
  number = {1},
  pages = {731--808},
  publisher = {Classe di Scienze},
  issn = {0391-173X},
  url = {https://dialnet.unirioja.es/servlet/articulo?codigo=5482814}
}

@article{friedlander_radiation_1980,
  title = {Radiation Fields and Hyperbolic Scattering Theory},
  author = {Friedlander, F. G.},
  date = {1980-11},
  journaltitle = {Mathematical Proceedings of the Cambridge Philosophical Society},
  volume = {88},
  number = {3},
  pages = {483--515},
  publisher = {Cambridge University Press},
  issn = {1469-8064, 0305-0041},
  doi = {10.1017/S0305004100057819},
  url = {https://www.cambridge.org/core/journals/mathematical-proceedings-of-the-cambridge-philosophical-society/article/abs/radiation-fields-and-hyperbolic-scattering-theory/7B6E97C7182B5C3C03396F615CF7EA2F}
}

@article{gajic_relation_2022,
  title = {On the Relation between Asymptotic Charges, the Failure of Peeling and Late-Time Tails},
  author = {Gajic, Dejan and Kehrberger, Leonhard M. A.},
  date = {2022},
  journaltitle = {Classical and Quantum Gravity},
  shortjournal = {Class. Quantum Grav.},
  volume = {39},
  number = {19},
  pages = {195006},
  issn = {0264-9381},
  doi = {10.1088/1361-6382/ac8863},
  url = {https://dx.doi.org/10.1088/1361-6382/ac8863}
}

@article{grieser_basics_2001,
  title = {Basics of the B-{{Calculus}}},
  author = {Grieser, Daniel},
  date = {2001},
  journaltitle = {Approaches to Singular Analysis},
  eprint = {math/0010314},
  eprinttype = {arxiv},
  pages = {30--84},
  doi = {10.1007/978-3-0348-8253-8_2}
}

@article{grillakis_stability_1987,
  title = {Stability Theory of Solitary Waves in the Presence of Symmetry, {{I}}},
  author = {Grillakis, Manoussos and Shatah, Jalal and Strauss, Walter},
  date = {1987-09-01},
  journaltitle = {Journal of Functional Analysis},
  shortjournal = {Journal of Functional Analysis},
  volume = {74},
  number = {1},
  pages = {160--197},
  issn = {0022-1236},
  doi = {10.1016/0022-1236(87)90044-9},
  url = {https://www.sciencedirect.com/science/article/pii/0022123687900449}
}

@unpublished{hintz_gluing_2023,
  title = {Gluing Small Black Holes along Timelike Geodesics {{I}}: Formal Solution},
  shorttitle = {Gluing Small Black Holes along Timelike Geodesics {{I}}},
  author = {Hintz, Peter},
  date = {2023-06-12},
  eprint = {2306.07409},
  eprinttype = {arxiv},
  url = {http://arxiv.org/abs/2306.07409}
}

@online{hintz_lectures_2023,
  title = {Lectures on Geometric Singular Analysis with Applications to Elliptic and Hyperbolic {{PDE}}},
  author = {Hintz, Peter},
  date = {2023},
  pubstate = {preprint}
}

@unpublished{hintz_linear_2023,
  title = {Linear Waves on Asymptotically Flat Spacetimes. {{I}}},
  author = {Hintz, Peter},
  date = {2023-02-27},
  eprint = {2302.14647},
  eprinttype = {arxiv},
  url = {http://arxiv.org/abs/2302.14647}
}

@article{hintz_stability_2020,
  title = {Stability of {{Minkowski}} Space and Polyhomogeneity of the Metric},
  author = {Hintz, Peter and Vasy, András},
  date = {2020-06},
  journaltitle = {Annals of PDE},
  shortjournal = {Ann. PDE},
  number = {1},
  issn = {2524-5317, 2199-2576},
  doi = {10.1007/s40818-020-0077-0},
  url = {http://link.springer.com/10.1007/s40818-020-0077-0}
}

@article{holzegel_boundedness_2014,
  title = {Boundedness and Growth for the Massive Wave Equation on Asymptotically Anti-de {{Sitter}} Black Holes},
  author = {Holzegel, Gustav H. and Warnick, Claude M.},
  date = {2014-02},
  journaltitle = {Journal of Functional Analysis},
  shortjournal = {Journal of Functional Analysis},
  volume = {266},
  number = {4},
  eprint = {1209.3308},
  eprinttype = {arxiv},
  pages = {2436--2485},
  issn = {00221236},
  doi = {10.1016/j.jfa.2013.10.019},
  url = {http://arxiv.org/abs/1209.3308}
}

@article{jendrej_construction_2020,
  title = {Construction of Multi-Bubble Solutions for the Energy-Critical Wave Equation in Dimension 5},
  author = {Jendrej, Jacek and Martel, Yvan},
  date = {2020-07-01},
  journaltitle = {Journal de Mathématiques Pures et Appliquées},
  shortjournal = {Journal de Mathématiques Pures et Appliquées},
  volume = {139},
  pages = {317--355},
  issn = {0021-7824},
  doi = {10.1016/j.matpur.2020.02.007},
  url = {https://www.sciencedirect.com/science/article/pii/S0021782420300374}
}

@article{johnson_linear_2019,
  title = {The {{Linear Stability}} of the {{Schwarzschild Solution}} to {{Gravitational Perturbations}} in the {{Generalised Wave Gauge}}},
  author = {Johnson, Thomas William},
  date = {2019-09-07},
  journaltitle = {Annals of PDE},
  shortjournal = {Ann. PDE},
  volume = {5},
  number = {2},
  pages = {13},
  issn = {2199-2576},
  doi = {10.1007/s40818-019-0069-0},
  url = {https://doi.org/10.1007/s40818-019-0069-0}
}

@unpublished{keir_weak_2018,
  title = {The Weak Null Condition and Global Existence Using the P-Weighted Energy Method},
  author = {Keir, Joseph},
  date = {2018-09-30},
  eprint = {1808.09982},
  eprinttype = {arxiv},
  url = {http://arxiv.org/abs/1808.09982}
}

@article{krieger_center-stable_2015,
  title = {Center-Stable Manifold of the Ground State in the Energy Space for the Critical Wave Equation},
  author = {Krieger, Joachim and Nakanishi, Kenji and Schlag, Wilhelm},
  date = {2015-02-01},
  journaltitle = {Mathematische Annalen},
  shortjournal = {Math. Ann.},
  volume = {361},
  number = {1},
  pages = {1--50},
  issn = {1432-1807},
  doi = {10.1007/s00208-014-1059-x},
  url = {https://doi.org/10.1007/s00208-014-1059-x}
}

@article{krieger_focusing_2007,
  title = {On the {{Focusing Critical Semi-Linear Wave Equation}}},
  author = {Krieger, J. and Schlag, W.},
  date = {2007},
  journaltitle = {American Journal of Mathematics},
  volume = {129},
  number = {3},
  pages = {843--913},
  publisher = {Johns Hopkins University Press},
  issn = {0002-9327}
}

@article{krieger_renormalization_2008,
  title = {Renormalization and Blow up for Charge One Equivariant Critical Wave Maps},
  author = {Krieger, J. and Schlag, W. and Tataru, D.},
  date = {2008-03-01},
  journaltitle = {Inventiones mathematicae},
  shortjournal = {Invent. math.},
  volume = {171},
  number = {3},
  pages = {543--615},
  issn = {1432-1297},
  doi = {10.1007/s00222-007-0089-3},
  url = {https://doi.org/10.1007/s00222-007-0089-3}
}

@article{kwong_ground_1992,
  title = {On ground state solutions of {$-\Delta u=u^p-u^q$}},
  author = {Kwong, M. K and McLeod, J. B and Peletier, L. A and Troy, W. C},
  date = {1992-02-01},
  journaltitle = {Journal of Differential Equations},
  shortjournal = {Journal of Differential Equations},
  volume = {95},
  number = {2},
  pages = {218-239},
  issn = {0022-0396},
  doi = {10.1016/0022-0396(92)90030-Q},
  url = {https://www.sciencedirect.com/science/article/pii/002203969290030Q}
}

@article{lee_stability_1989,
  title = {Stability of Mini-Boson Stars},
  author = {Lee, T.D. and Pang, Yang},
  date = {1989-03},
  journaltitle = {Nuclear Physics B},
  shortjournal = {Nuclear Physics B},
  volume = {315},
  number = {2},
  pages = {477--516},
  issn = {05503213},
  doi = {10.1016/0550-3213(89)90365-9},
  url = {https://linkinghub.elsevier.com/retrieve/pii/0550321389903659}
}

@article{lewin_double-power_2020,
  title = {The Double-Power Nonlinear {{Schrödinger}} Equation and Its Generalizations: Uniqueness, Non-Degeneracy and Applications},
  shorttitle = {The Double-Power Nonlinear {{Schrödinger}} Equation and Its Generalizations},
  author = {Lewin, Mathieu and Rota Nodari, Simona},
  date = {2020-10-31},
  journaltitle = {Calculus of Variations and Partial Differential Equations},
  shortjournal = {Calc. Var.},
  volume = {59},
  number = {6},
  pages = {197},
  issn = {1432-0835},
  doi = {10.1007/s00526-020-01863-w},
  url = {https://doi.org/10.1007/s00526-020-01863-w}
}

@unpublished{luhrmann_stability_2022,
  title = {Stability of the {{Catenoid}} for the {{Hyperbolic Vanishing Mean Curvature Equation Outside Symmetry}}},
  author = {Luhrmann, Jonas and Oh, Sung-Jin and Shahshahani, Sohrab},
  date = {2022-12-11},
  eprint = {2212.05620},
  eprinttype = {arxiv},
  url = {http://arxiv.org/abs/2212.05620}
}

@article{luk_weak_2018,
  title = {Weak Null Singularities in General Relativity},
  author = {Luk, Jonathan},
  date = {2018-01},
  journaltitle = {Journal of the American Mathematical Society},
  shortjournal = {J. Amer. Math. Soc.},
  volume = {31},
  number = {1},
  pages = {1--63},
  issn = {0894-0347, 1088-6834},
  doi = {10.1090/jams/888},
  url = {https://www.ams.org/jams/2018-31-01/S0894-0347-2017-00888-9/}
}

@article{martel_construction_2016,
  title = {Construction of Multi-Solitons for the Energy-Critical Wave Equation in Dimension 5},
  author = {Martel, Yvan and Merle, Frank},
  date = {2016-12},
  journaltitle = {Archive for Rational Mechanics and Analysis},
  shortjournal = {Arch Rational Mech Anal},
  volume = {222},
  number = {3},
  eprint = {1504.01595},
  eprinttype = {arxiv},
  eprintclass = {math},
  pages = {1113--1160},
  issn = {0003-9527, 1432-0673},
  doi = {10.1007/s00205-016-1018-7},
  url = {http://arxiv.org/abs/1504.01595}
}

@article{mazzeo_meromorphic_1987,
  title = {Meromorphic Extension of the Resolvent on Complete Spaces with Asymptotically Constant Negative Curvature},
  author = {Mazzeo, Rafe R and Melrose, Richard B},
  date = {1987-12},
  journaltitle = {Journal of Functional Analysis},
  shortjournal = {Journal of Functional Analysis},
  volume = {75},
  number = {2},
  pages = {260--310},
  issn = {00221236},
  doi = {10.1016/0022-1236(87)90097-8},
  url = {https://linkinghub.elsevier.com/retrieve/pii/0022123687900978}
}

@article{moschidis_rp_2016,
  title = {The Rp -{{Weighted Energy Method}} of {{Dafermos}} and {{Rodnianski}} in {{General Asymptotically Flat Spacetimes}} and {{Applications}}},
  author = {Moschidis, Georgios},
  date = {2016},
  journaltitle = {Annals of PDE},
  volume = {2},
  number = {1},
  eprint = {1509.08489},
  eprinttype = {arxiv},
  issn = {21992576},
  doi = {10.1007/s40818-016-0011-7}
}

@unpublished{pasqualotto_asymptotics_2023,
  title = {The Asymptotics of Massive Fields on Stationary Spherically Symmetric Black Holes for All Angular Momenta},
  author = {Pasqualotto, Federico and Shlapentokh-Rothman, Yakov and Van de Moortel, Maxime},
  date = {2023-03-30},
  eprint = {2303.17767},
  eprinttype = {arxiv},
  url = {http://arxiv.org/abs/2303.17767}
}

@article{schlue_decay_2013,
  title = {Decay of Linear Waves on Higher Dimensional {{Schwarzschild}} Black Holes},
  author = {Schlue, Volker},
  date = {2013-07-11},
  journaltitle = {Analysis \& PDE},
  shortjournal = {Anal. PDE},
  volume = {6},
  number = {3},
  pages = {515--600},
  issn = {1948-206X, 2157-5045},
  doi = {10.2140/apde.2013.6.515},
  url = {http://arxiv.org/abs/1012.5963}
}

@unpublished{sussman_massive_2023,
  title = {Massive Wave Propagation near Null Infinity},
  author = {Sussman, Ethan},
  date = {2023-05-01},
  eprint = {2305.01119},
  eprinttype = {arxiv},
  url = {http://arxiv.org/abs/2305.01119}
}

@article{yu_modified_2021,
  title = {Modified Wave Operators for a Scalar Quasilinear Wave Equation Satisfying the Weak Null Condition},
  author = {Yu, Dongxiao},
  date = {2021-03},
  journaltitle = {Communications in Mathematical Physics},
  shortjournal = {Commun. Math. Phys.},
  volume = {382},
  number = {3},
  pages = {1961--2013},
  issn = {0010-3616, 1432-0916},
  doi = {10.1007/s00220-021-03989-0},
  url = {http://arxiv.org/abs/2002.05355}
}

@unpublished{yu_nontrivial_2022,
  title = {Nontrivial Global Solutions to Some Quasilinear Wave Equations in Three Space Dimensions},
  author = {Yu, Dongxiao},
  date = {2022-04-27},
  eprint = {2204.12870},
  eprinttype = {arxiv},
  url = {http://arxiv.org/abs/2204.12870}
}

@article{yuan_construction_2022,
  title = {Construction of Excited Multi-Solitons for the Focusing {{4D}} Cubic Wave Equation},
  author = {Yuan, Xu},
  date = {2022-02-15},
  journaltitle = {Journal of Functional Analysis},
  shortjournal = {Journal of Functional Analysis},
  volume = {282},
  number = {4},
  pages = {109336},
  issn = {0022-1236},
  doi = {10.1016/j.jfa.2021.109336},
  url = {https://www.sciencedirect.com/science/article/pii/S0022123621004183}
}

@article{yuan_multi-solitons_2019,
  title = {On Multi-Solitons for the Energy-Critical Wave Equation in Dimension 5},
  author = {Yuan, Xu},
  date = {2019-12-01},
  journaltitle = {Nonlinearity},
  shortjournal = {Nonlinearity},
  volume = {32},
  number = {12},
  eprint = {1809.05414},
  eprinttype = {arxiv},
  eprintclass = {math},
  pages = {5017--5048},
  issn = {0951-7715, 1361-6544},
  doi = {10.1088/1361-6544/ab46ec},
  url = {http://arxiv.org/abs/1809.05414}
}

@article{zworski_resonances_2016,
  title = {Resonances for Asymptotically Hyperbolic Manifolds: {{Vasy}}’s Method Revisited},
  shorttitle = {Resonances for Asymptotically Hyperbolic Manifolds},
  author = {Zworski, Maciej},
  date = {2016-12-09},
  journaltitle = {Journal of Spectral Theory},
  volume = {6},
  number = {4},
  pages = {1087--1114},
  issn = {1664-039X},
  doi = {10.4171/jst/153},
  url = {https://ems.press/journals/jst/articles/14423}
}
	
\end{document}